\documentclass[11pt]{article}

\usepackage{amsmath, amsfonts, amssymb, amsthm,  graphicx, mathtools, enumerate}

\usepackage[blocks, affil-it]{authblk}

\usepackage{mathrsfs}

\usepackage[numbers, square]{natbib}
\usepackage[CJKbookmarks=true,
            bookmarksnumbered=true,
			bookmarksopen=true,
			colorlinks=true,
			citecolor=red,
			linkcolor=blue,
			anchorcolor=red,
			urlcolor=blue]{hyperref}
\usepackage[usenames]{color}

\usepackage[letterpaper, left=1.2truein, right=1.2truein, top = 1.2truein, bottom = 1.2truein]{geometry}
\usepackage[ruled, vlined, lined, commentsnumbered]{algorithm2e}

\usepackage{prettyref,soul}

\newtheorem{lemma}{Lemma}[section]

\newtheorem{thm}{Theorem}[section]

\newtheorem{corollary}{Corollary}[section]
\newtheorem{remark}{Remark}[section]

\newrefformat{eq}{(\ref{#1})}
\newrefformat{chap}{Chapter~\ref{#1}}
\newrefformat{sec}{Section~\ref{#1}}
\newrefformat{algo}{Algorithm~\ref{#1}}
\newrefformat{fig}{Fig.~\ref{#1}}
\newrefformat{tab}{Table~\ref{#1}}
\newrefformat{rmk}{Remark~\ref{#1}}
\newrefformat{clm}{Claim~\ref{#1}}
\newrefformat{def}{Definition~\ref{#1}}
\newrefformat{cor}{Corollary~\ref{#1}}
\newrefformat{lmm}{Lemma~\ref{#1}}
\newrefformat{lemma}{Lemma~\ref{#1}}
\newrefformat{prop}{Proposition~\ref{#1}}
\newrefformat{app}{Appendix~\ref{#1}}
\newrefformat{ex}{Example~\ref{#1}}
\newrefformat{exer}{Exercise~\ref{#1}}
\newrefformat{soln}{Solution~\ref{#1}}
\newrefformat{cond}{Condition~\ref{#1}}

\def\text#1{\mbox{\rm #1}}

\def\X{\mathscr{X}}
\def\Z{\mathcal{Z}}
\def\T{\mathcal{T}}

\DeclarePairedDelimiter{\ceil}{\lceil}{\rceil}

\newcommand{\argmin}{\mathop{\rm argmin}}
\newcommand{\argmax}{\mathop{\rm argmax}}

\newcommand{\norm}[1]{\|{#1} \|}

\newcommand{\fnorm}[1]{\|#1\|_{\rm F}}
\newcommand{\opnorm}[1]{\|#1\|_{\rm op}}

\newcommand{\supp}{{\rm supp}}
\newcommand{\iprod}[2]{\left \langle #1, #2 \right\rangle}

\newcommand{\floor}[1]{{\left\lfloor {#1} \right \rfloor}}

\title{A General Framework for Bayes Structured Linear Models
}
\author[1]{Chao Gao}
\author[2]{Aad W.~van der Vaart}
\author[3]{Harrison H.~Zhou}
\affil[1]{
University of Chicago
}
\affil[2]{
Leiden University
}
\affil[3]{
Yale University
}

\begin{document}
\maketitle

\begin{abstract}
High dimensional statistics deals with the challenge of extracting structured information from complex model settings. Compared with a large number of frequentist methodologies, there are rather few theoretically optimal Bayes methods for high dimensional models. 
This paper provides a unified approach to both Bayes high dimensional statistics and Bayes nonparametrics in a general framework of structured linear models. With a proposed two-step  prior, we prove a general oracle inequality for posterior contraction under an abstract setting that allows model misspecification. The general result can be used to derive new results on optimal posterior contraction under many complex model settings including recent works for stochastic block model, graphon estimation and dictionary learning. It can also be used to improve upon posterior contraction results in literature including sparse linear regression and nonparametric aggregation.
The key of the success lies in the novel two-step prior distribution: one for model structure, i.e., model selection, and the other one for model parameters. The prior on the parameters of a model 
is an elliptical Laplace distribution that is capable of modeling signals with large magnitude, and the prior on the model structure involves a factor that compensates the effect of the normalizing constant of the elliptical Laplace distribution, which is important to attain rate-optimal posterior contraction.
\smallskip

\textbf{Keywords.} Oracle inequality, Stochastic block model, Graphon, Sparse linear regression, Aggregation, Dictionary learning, Posterior contraction
\end{abstract}


\section{Introduction}

Theory for posterior distribution has been extensively investigated in Bayes nonparametrics recently. Important works such as \cite{barron88,barron1999consistency,ghosal1999posterior,ghosal2000convergence,shen2001rates,van2008rates,hoffmann2013adaptive,castillo2014bayesian} established that the posterior distribution contracts to a small neighborhood of the truth under proper conditions on likelihood functions and priors. These works bridge the gap between frequentist and Bayesian views of statistics from a fundamental perspective.

Despite the success of theoretical advancements of Bayes nonparametrics, there are not many theories developed for Bayes high dimensional statistics. A few exceptions are \cite{castillo2012needles} on sparse Gaussian sequence model, \cite{banerjee2014posterior} on bandable precision matrix estimation and \cite{gao2013rate} on sparse PCA. Recently, \cite{castillo2014} established posterior contraction rates for sparse linear regression with a spike and slab prior under comparable assumptions for theoretical justification of the Lasso estimator \citep{tibshirani1996regression,bickel2009simultaneous}. The results of \cite{castillo2014} include posterior contraction rates for prediction error, estimation error, oracle inequalities and model selection consistency. However, sparse linear regression is just one example of high dimensional statistics. There is an indispensable demand of a Bayes theory on more complicated model settings such as dictionary learning, stochastic block model, multi-task learning, etc. It is not clear whether the theory in \cite{castillo2014} can be extended to these more complex settings.

This paper provides a unified methodology and theory for both Bayes high dimensional statistics and Bayes nonparametric statistics in a general framework of structured linear models. We first introduce a unified view of various high-dimensional and nonparametric models, and then propose a single prior distribution for all models considered in our framework. Optimal rates of convergence of the posterior distributions are established under appropriate conditions. The results directly lead to exact minimax posterior contraction rates in stochastic block model, biclustering, sparse linear regression, regression with group sparsity, multi-task learning and dictionary learning. Moreover, we also derive a general posterior oracle inequality that allows arbitrary model misspecification. Applications of the posterior oracle inequality help us obtain posterior contraction rates even for models that are not included in our framework of structured linear models. Examples considered in this paper include nonparametric graphon estimation, various forms of nonparametric aggregation, linear regression with approximate sparsity and wavelet estimation under Besov spaces.

In the heart of the general theory is a novel two-step prior distribution, which naturally accommodates the structured linear model by first modeling the structure and then modeling the parameters. This two-step modeling strategy was first investigated by  \cite{castillo2012needles} for Gaussian sequence models. A key ingredient of the prior distribution is that the tail of the distribution on  the model parameter  cannot be too light \citep{castillo2012needles,castillo2014}, which motivates \cite{castillo2012needles,castillo2014} to use the independent Laplace prior on the parameter. Though the prior distribution leads to optimal posterior contraction rates in Gaussian sequence model \cite{castillo2012needles}, it requires some excessive assumptions on the design matrix when it is applied to sparse linear regression \citep{castillo2014}. The proposal in this paper is an elliptical Laplace prior. With this choice, not only are we able to weaken the assumptions in \citep{castillo2014}, but we can also  solve a more general class of problems in a unified way. To compensate the influence of the normalizing constant of the elliptical Laplace distribution, a correction factor on the prior mass is added in the model selection step. Without this correction factor, the posterior contraction rate would be sub-optimal.

The paper is organized as follows. Section \ref{sec:model} introduces the general framework of structured linear models. A general prior distribution is proposed in Section \ref{sec:prior}. Section \ref{sec:main} presents the main results of the paper including a rate optimal posterior oracle inequality and a posterior rate of contraction. The main results are illustrated by ten examples ranging from nonparametric estimation to high dimensional statistics in Section \ref{sec:app}. In Section \ref{sec:more}, we present further results on sparse linear regression. All technical proofs are gathered in Section \ref{sec:proof} and the supplement \citep{gao16}.

We close this section by introducing some notations. Given an integer $d$, we use $[d]$ to denote the set $\{1,2,...,d\}$, and $[d]^n$ to denote $\{(i_1,...,i_n)\in\mathbb{R}^n:i_1,...,i_n\in[d]\}$. For a set $S$, $|S|$ denotes its cardinality and $\mathbb{I}_S$ denotes the indicator function. For a vector $u=(u_i)$, $\norm{u}=\sqrt{\sum_iu_i^2}$ denotes the $\ell_2$ norm. For a matrix $A=(A_{ij})\in\mathbb{R}^{n\times p}$, and a subset $T\subset[n]\times [p]$, $A_T$ denotes the array $\{A_t\}_{t\in T}$. For any $I\subset[n]$ and $J\subset [p]$, we let $A_{I*}=A_{I\times [p]}$ and $A_{*J}=A_{[n]\times J}$. The Frobenius norm, $\ell_1$ norm and $\ell_{\infty}$ norm are defined by $\fnorm{A}=\sqrt{\sum_{ij}A_{ij}^2}$, $\norm{A}_1=\sum_{ij}|A_{ij}|$ and $\norm{A}_{\infty}=\max_{ij}|A_{ij}|$, respectively. When $A=A^T\in\mathbb{R}^{p\times p}$ is symmetric, the operator norm $\opnorm{A}$ is defined by its largest singular value and the matrix $\ell_1$ norm $\norm{A}_{\ell_1}$ is defined by the maximum row sum. The inner product is defined by $\iprod{u}{v}=\sum_i u_iv_i$ when applied to vectors and is defined by $\iprod{A}{B}=\sum_{ij}A_{ij}B_{ij}$ when applied to matrices. Given two numbers $a,b\in\mathbb{R}$, $a\vee b=\max(a,b)$ and $a\wedge b=\min(a,b)$. The floor function $\floor{a}$ is the largest integer no greater than $a$, and the ceiling function $\ceil{a}$ is the smallest integer no less than $a$. For two positive sequences $\{a_n\},\{b_n\}$, $a_n\lesssim b_n$ means $a_n\leq Cb_n$ for some constant $C>0$ independent of $n$, and $a_n\asymp b_n$ means $a_n\lesssim b_n$ and $b_n\lesssim a_n$.  The symbols $\mathbb{P}$ and $\mathbb{E}$ denote generic probability and expectation operators whose distribution is determined from the context.

\section{Structured linear models}\label{sec:model}

Consider the following structured linear model
\begin{equation}
Y=\X_Z(Q)+W\in\mathbb{R}^N,\label{eq:SLMs}
\end{equation}
where $W\in\mathbb{R}^N$ is a noise vector and $\X_Z(\cdot)$ is a linear operator. The signal $\X_Z(Q)$ has two elements, the parameter $Q$ and the structure/model $Z$ that indexes the linear operator $\X_Z(\cdot)$. In the example of sparse linear regression $Y=X\beta+W$ with a sparse regression coefficient $\beta=(\beta_S^T,0_{S^c}^T)^T$ for some subset $S$, we have $Q=\beta_S$, $Z=S$ and $\X_Z(\cdot)=X_{*S}$. This gives the representation $X\beta=X_{*S}\beta_S=\X_Z(Q)$. In the general theory, the structure $Z$ is in some discrete space $\Z_{\tau}$, which is further indexed by $\tau\in\T$ for some finite set $\T$. For sparse linear regression, $\Z_{\tau}$ is the set of models of size $\tau$.
We introduce a function $\ell(\Z_{\tau})$ to denote the dimension of the parameter $Q$. In other words, $Q\in\mathbb{R}^{\ell(\Z_{\tau})}$, and $\ell(\Z_{\tau})$ is referred to as the intrinsic dimension of the structured linear model. The complexity of the model is defined by the quantity
\begin{equation}
\ell(\Z_{\tau})+\log|\Z_{\tau}|, \label{eq:complexity}
\end{equation}
the sum of the intrinsic dimension and the logarithmic cardinality of the structure space. The definition of (\ref{eq:complexity}) has a frequentist root (see, for example, \cite{yang1998asymptotic,barron1999risk,birge2001gaussian}). As we are going to show later, (\ref{eq:complexity}) will be the posterior contraction rate that we target at. Moreover, in all the examples considered in the paper, (\ref{eq:complexity}) will be the minimax rate under the prediction loss. We require linearity of the operator $\X_Z(\cdot)$. That is, given any $Z\in\Z_{\tau}$ with any $\tau\in\T$, we have
\begin{equation}
\X_Z(Q_1+Q_2) = \X_Z(Q_1)+\X_Z(Q_2),\quad\text{for all }Q_1,Q_2\in\mathbb{R}^{\ell(\Z_{\tau})}. \label{eq:linearity}
\end{equation}
Therefore, we can also view $\X_Z$ as a matrix in $\mathbb{R}^{N\times \ell(\Z_{\tau})}$. From now on, whenever we apply a matrix operation with $\X_Z$, the operator $\X_Z$ is understood to be a matrix with slight abuse of notation.

The above framework of structured linear models includes many examples. In this paper, we consider only the following six representative instances.
\begin{enumerate}
\item \textit{Stochastic block model.} Consider $\X_Z(Q)\in [0,1]^{n\times n}$ to be the mean matrix of a random graph with specification $[\X_Z(Q)]_{ij}=Q_{z(i)z(j)}$. The object $z\in [k]^n$ is the labels of the graph nodes. Moreover, it is easy to see that the parameter $Q$ is of dimension $k^2$, when we do not impose symmetry for $Q$. Therefore, stochastic block model is a special case of our general framework in view of the relation $Z=z$, $\tau=k$, $\T=[n]$, $\Z_k=[k]^n$ and $\ell(\Z_k)=k^2$.

\item \textit{Biclustering.} For a matrix $\X_Z(Q)\in\mathbb{R}^{n\times m}$, a biclustering model means that both rows and columns have clustering structures. That is, $[\X_Z(Q)]_{ij}=Q_{z_1(i)z_2(j)}$ for some $z_1\in[k]^n$ and $z_2\in[l]^m$. The parameter $Q$ has dimension $kl$. Thus, biclustering model is a special case of our general framework by the relation $Z=(z_1,z_2)$, $\tau=(k,l)$, $\T=[n]\times[m]$, $\Z_{k,l}=[k]^n\times [l]^m$ and $\ell(\Z_{k,l})=kl$.

\item \textit{Sparse linear regression.} A $p$-dimensional sparse linear regression model refers to $X\beta$, where $\beta\in\mathbb{R}^p$ has a subset of nonzero entries and it can be represented by $\beta^T=(\beta_S^T,0_{S^c}^T)$ for some subset $S\subset[p]$. In other words, $X\beta=X_{*S}\beta_S$. It can be represented in a general way by letting $Z=S$, $\tau=s$, $\T=[p]$, $\Z_s=\{S\subset[p]: |S|=s\}$, $\ell(\Z_s)=s$ and $Q=\beta_S$. Moreover, $\X_Z(Q)=X_{*S}\beta_S$.

\item \textit{Multiple linear regression with group sparsity.} It refers to the model $XB$ with $B\in\mathbb{R}^{p\times m}$ being a coefficient matrix with nonzero rows in some subset $S\subset[p]$. It can be represented in a general form similarly as the sparse linear regression with $\ell(\Z_s)=ms$.

\item \textit{Multi-task learning.} Similar to the last example, multi-task learning is the collection of $m$ regression problems. We consider $XB$ for some $B\in\mathbb{R}^{p\times m}$. The $j$th column of $B$ can be represented as $B_{*j}=Q_{*z(j)}$ for some $z\in[k]^m$ and $Q\in\mathbb{R}^{p\times k}$. Thus, it is a special case of our general framework by letting $Z=z$, $\tau=k$, $\T=[m]$, $\Z_k=[k]^m$ and $\ell(\Z_k)=pk$.

\item \textit{Dictionary learning.} Consider the model $\X_Z(Q)=QZ\in\mathbb{R}^{n\times d}$ for some $Z\in\{-1,0,1\}^{p\times d}$ and $Q\in\mathbb{R}^{n\times p}$. Each column of $Z$ is assumed to be sparse. Therefore, dictionary learning can be viewed as sparse regression without knowing the design. It can be written in a general form by letting $\tau=(p,s)$, $\T=\{(p,s)\in[n\wedge d]\times[n]: s\leq p\}$, $\Z_{p,s}=\{Z\in\{-1,0,1\}^{p\times d}: \max_{j\in[d]}|\supp(Z_{*j})|\leq s\}$ and $\ell(\Z_{p,s})=np$.
\end{enumerate}

In several examples above, $Q$ or $\X_Z(Q)$ can be a matrix instead of a vector. Alternative definitions of these examples in the general framework are available by vectorization and Kronecker products. For example, in dictionary learning, the linear operator $\X_Z: Q\mapsto QZ$ is from matrix to matrix. By the formula $\text{vec}(QZ)=(Z^T\otimes I_n)\text{vec}(Q)$, the linear operator $\X_Z$ can be identified with the matrix $Z^T\otimes I_n\in\mathbb{R}^{nd\times np}$, which is also a linear operator from $\mathbb{R}^{np}$ to $\mathbb{R}^{nd}$. Similar rearrangements apply to other examples as well.

In addition to the six examples above, we have four more examples that can be well approximated by the general structured linear models.

\begin{enumerate}
\setcounter{enumi}{6}
\item \textit{Nonparametric graphon estimation.} For an undirected graph, the distribution of its adjacency matrix $\{A_{ij}\}\in\{0,1\}^{n\times n}$ is determined by $A_{ij}|(\xi_i,\xi_j)\sim \text{Bernoulli}\left(f(\xi_i,\xi_j)\right)$, where $\{\xi_i\}$ are latent variables with some joint distribution $\mathbb{P}_{\xi}$. The symmetric nonparametric function $f$ is called graphon. It governs the underlying data generating process of a random graph. When $f$ is assumed to be in a H\"{o}lder class, it can be approximated by the stochastic block model.

\item \textit{Aggregation.} Consider a nonparametric regression problem $Y_i=f(x_i)+W_i$ for $i\in[n]$. Given a collection of functions $\{f_1,...,f_p\}$ and a subset $\Theta\subset\mathbb{R}^p$, the goal of aggregation is to approximate $f$ with some estimator so that the error is comparable to what is given by the best among the class $\left\{f_{\beta}=\sum_{j=1}^p\beta_j f_j:\beta\in\Theta\right\}$. In Section \ref{sec:agg}, we show that the regression function $f$ can be approximated by the general structured linear model.

\item \textit{Linear regression with approximate sparsity.} For the linear regression problem $Y=X\beta+W$, assume that $\beta$ is in a weak $\ell_q$ ball so that it has an approximate sparse pattern. Then, $X\beta$ can be approximated by the structured linear model with an exact sparse pattern.

\item \textit{Wavelet estimation under Besov space.} Consider the Gaussian sequence model $Y_{jk}=\theta_{jk}+n^{-1/2}W_{jk}$ for $k=1,...,2^j$ and $j=0,1,2,...$. The signal $\theta$ belongs to a Besov ball $\Theta_{p,q}^{\alpha}(L)$. Then, we can use the general structured linear model to approximate the signal at each resolution separately. This strategy leads to a minimax optimal procedure for a large collection of Besov balls.
\end{enumerate}

\section{The prior distribution}\label{sec:prior}

In this section, we introduce a prior distribution on the structured linear model (\ref{eq:SLMs}). The prior distribution has a two-step sampling procedure. First, we are going to sample a structure $Z$. Second, given $Z$, we sample the parameter $Q$.
Let us first present the prior distribution on the parameter $Q\in\mathbb{R}^{\ell(\Z_{\tau})}$. We propose an elliptical Laplace distribution with density function proportion to $\exp\left(-\lambda\norm{\X_Z(Q)}\right)$. By direct calculations of its normalizing constant, the density function is
\begin{equation}
f_{\ell(\Z_{\tau}),\X_Z,\lambda}(Q)=\frac{\sqrt{\det(\X_Z^T\X_Z)}}{2}\left(\frac{\lambda}{\sqrt{\pi}}\right)^{\ell(\Z_{\tau})}\frac{\Gamma(\ell(\Z_{\tau})/2)}{\Gamma(\ell(\Z_{\tau}))}\exp\left(-\lambda\norm{\X_Z(Q)}\right).\label{eq:laplace}
\end{equation}
A derivation of the normalizing constant with details is given in Section \ref{sec:ELD} of the supplement. Note that (\ref{eq:laplace}) is well-defined when $\det(\X_Z^T\X_Z)>0$.
Recall that $\X_Z$ is understood as a matrix in $\mathbb{R}^{N\times\ell(\Z_{\tau})}$ whenever a matrix operation is applied. The elliptical Laplace distribution belongs to the elliptical family \citep{fang1990symmetric} with scatter matrix proportional to $(\X_Z^T\X_Z)^{-1}$. Compared with a product measure on $Q$, the density function (\ref{eq:laplace}) involves an extra factor $\frac{\Gamma(\ell(\Z_{\tau})/2)}{\Gamma(\ell(\Z_{\tau}))}$ in the normalizing constant. This factor needs to be corrected in the model selection step.

Let $\epsilon(\Z_{\tau})$ be a function satisfying
\begin{equation}
\epsilon(\Z_{\tau}) \geq \ell(\Z_{\tau})+\log|\Z_{\tau}|.\label{eq:larger}
\end{equation}
The sampling procedure of the prior distribution $\Pi$ on $\X_Z(Q)$ is given by:
\begin{enumerate}
\item Sample $\tau\sim\pi$ from $\T$, where $\pi(\tau)\propto \frac{\Gamma(\ell(\Z_{\tau}))}{\Gamma(\ell(\Z_{\tau})/2)}\exp\left(-D\epsilon(\Z_{\tau})\right)$;
\item Conditioning on $\tau$, sample $Z$ uniformly from the set $\bar{\Z}_{\tau}=\{Z\in\Z_{\tau}: \det(\X_Z^T\X_Z)>0\}$;
\item Conditioning on $(\tau,Z)$, sample $Q\sim f_{\ell(\Z_{\tau}),\X_Z,\lambda}$.
\end{enumerate}
Step 1 weighs the structure index $\tau$ by the function $\epsilon(\Z_{\tau})$ that satisfies (\ref{eq:larger}). For all the examples considered in the paper, $\epsilon(\Z_{\tau})$ is chosen to be at the same order of the model complexity (\ref{eq:complexity}). The quantity $\frac{\Gamma(\ell(\Z_{\tau}))}{\Gamma(\ell(\Z_{\tau})/2)}$ is a correction factor that is imposed to compensate the influence of $\frac{\Gamma(\ell(\Z_{\tau})/2)}{\Gamma(\ell(\Z_{\tau}))}$ in the elliptical Laplace distribution. Without the correction factor, $\exp\left(-D\epsilon(\Z_{\tau})\right)$ is the complexity prior used by \cite{castillo2012needles,castillo2014} in Gaussian sequence model and sparse linear regression. Since the support $\T$ is a finite set, $\pi$ is a valid probability mass function. Step 2 samples a structure $Z$ uniformly in $\bar{\Z}_{\tau}$. It is sufficient to consider such $Z$ that $\det(\X_Z^T\X_Z)>0$ for all the examples considered in this paper. Such restriction leads to a proper density function (\ref{eq:laplace}) and thus Step 3 is well defined.

After defining the prior, we need to specify the likelihood function. The examples in Section \ref{sec:model} have different distributions. For example, the stochastic block model usually assumes a Bernoulli random graph, while sparse linear regression often works with general sub-Gaussian noise distributions. To pursue a unified approach, we propose to use the Gaussian likelihood $Y|(Z,Q)\sim N(\X_Z(Q),I_N)$ throughout the paper. Then, the posterior distribution is
\begin{eqnarray*}
&& \Pi\left(\X_Z(Q)\in U| Y\right) \\
&=& \frac{\sum_{\tau\in\T}e^{-D\epsilon(\Z_{\tau})}\sum_{Z\in\bar{\Z}_{\tau}}\frac{\sqrt{\det(\X_Z^T\X_Z)}}{|\bar{\Z}_{\tau}|}\left(\frac{\lambda}{\sqrt{\pi}}\right)^{\ell(\Z_{\tau})}\int_{\X_Z(Q)\in U} e^{-\frac{1}{2}\norm{Y-\X_Z(Q)}^2-\lambda\norm{\X_Z(Q)}} dQ}{\sum_{\tau\in\T}e^{-D\epsilon(\Z_{\tau})}\sum_{Z\in\bar{\Z}_{\tau}}\frac{\sqrt{\det(\X_Z^T\X_Z)}}{|\bar{\Z}_{\tau}|}\left(\frac{\lambda}{\sqrt{\pi}}\right)^{\ell(\Z_{\tau})}\int e^{-\frac{1}{2}\norm{Y-\X_Z(Q)}^2-\lambda\norm{\X_Z(Q)}} dQ}.
\end{eqnarray*}
Note that in the above formula of posterior distribution, the factor $\frac{\Gamma(\ell(\Z_{\tau})/2)}{\Gamma(\ell(\Z_{\tau}))}$ in the Laplace normalizing constant has been cancelled out by the correction factor $\frac{\Gamma(\ell(\Z_{\tau}))}{\Gamma(\ell(\Z_{\tau})/2)}$ in the model selection prior.

\section{Main results}\label{sec:main}

In this section, we analyze the posterior distribution for the general structured linear model. Though the prior specifies a model $\X_Z(Q)$, we do not need to assume that data is generated from the same model. Instead, we allow data to be generated by an arbitrary signal with sub-Gaussian noise. That is,
$$Y=\theta^*+W,$$
where $W=Y-\theta^*$ is a noise vector with a sub-Gaussian tail satisfying
\begin{equation}
\mathbb{P}\left(\left|\iprod{W}{K}\right|>t\right) \leq e^{-\rho t^2/2}\text{ for all }\norm{K}=1.\label{eq:subG}
\end{equation}
The sub-Gaussianity number $\rho>0$ is assumed to be a constant throughout the paper. It worths noting that $1/\rho$ is a bound on the noise level. We also assume a mild condition on the function $\epsilon(\Z_{\tau})$,
\begin{equation}
\left|\left\{\tau\in\T: t-1<\epsilon(\Z_{\tau})\leq t\right\}\right|\leq t \text{ for all }t\in\mathbb{N}. \label{eq:capacity}
\end{equation}
The condition (\ref{eq:capacity}) is satisfied for all examples considered in the paper. The main result of the paper is stated in the following theorem.
Recall that $\lambda$ and $D$ are parameters of the prior distribution $\Pi$.
\begin{thm} \label{thm:main}
Assume (\ref{eq:larger}), (\ref{eq:subG}) and (\ref{eq:capacity}). Given any $\theta^*\in\mathbb{R}^N$, $\tau^*\in\T$, $Z^*\in\bar{\Z}_{\tau^*}$,  $Q^*\in\mathbb{R}^{\ell(\Z_{\tau^*})}$, any constants $\lambda,\rho>0$ and any sufficiently small constant $\delta\in (0,1)$, there exists some constant $D_{\lambda,\delta,\rho}>0$ only depending on $\lambda,\delta,\rho$, such that
\begin{eqnarray}
\nonumber && \mathbb{E}_{\theta^*}\Pi\left(\epsilon(\Z_{\tau})>(1+\delta_1)\epsilon(\Z_{\tau^*})+\delta_1\norm{\X_{Z^*}(Q^*)-\theta^*}^2\Big|Y\right) \\
\label{eq:main1} &\leq& \exp\left(-C'\left(\epsilon(\Z_{\tau^*})+\norm{\X_{Z^*}(Q^*)-\theta^*}^2\right)\right),
\end{eqnarray}
\begin{eqnarray}
\nonumber && \mathbb{E}_{\theta^*}\Pi\left(\norm{\X_Z(Q)-\theta^*}^2>(1+\delta_2)\norm{\X_{Z^*}(Q^*)-\theta^*}^2+M\epsilon(\Z_{\tau^*})\Big|Y\right) \\
\label{eq:main2} &\leq& \exp\left(-C''\left(\epsilon(\Z_{\tau^*})+\norm{\X_{Z^*}(Q^*)-\theta^*}^2\right)\right),
\end{eqnarray}
and
\begin{eqnarray}
\nonumber && \mathbb{E}_{\theta^*}\|\mathbb{E}_{\Pi}(\X_Z(Q)|Y)-\theta^*\|^2 \\
\label{eq:main33} &\leq& (1+\delta_2)\norm{\X_{Z^*}(Q^*)-\theta^*}^2 +M\epsilon(\Z_{\tau^*}) \\
\nonumber&& + \exp\left(-C'''\left(\epsilon(\Z_{\tau^*})+\norm{\X_{Z^*}(Q^*)-\theta^*}^2\right)\right),
\end{eqnarray}
for any constant $D>D_{\lambda,\delta,\rho}$ with $\delta_1=\delta$, $\delta_2=8\sqrt{14\delta/\rho}$ and some constants $M,C',C'',C'''$ only depending on $\lambda,\delta,\rho,D$.
\end{thm}

Theorem \ref{thm:main} contains three results of an oracle type. The object $\X_{Z^*}(Q^*)$ can be chosen with arbitrary $Q^*$ and $Z^*$, but is usually taken as the oracle model that best approximates the true signal $\theta^*$ in many applications. The first result (\ref{eq:main1}) shows that the model complexity selected by the posterior distribution is not greater than the sum of the complexity of the oracle and a model misspecification term quantified by $\norm{\X_{Z^*}(Q^*)-\theta^*}^2$. The second result (\ref{eq:main2}) is a posterior oracle inequality for the squared error loss $\norm{\X_Z(Q)-\theta^*}^2$. Compared with that of the oracle $\X_{Z^*}(Q^*)$, the squared error loss of $\X_Z(Q)$ has an extra term proportional to $\epsilon(\Z_{\tau^*})$. The third result is an oracle inequality for the posterior mean $\mathbb{E}_{\Pi}(\X_Z(Q)|Y)$. It is worth noting that $\exp\left(-C'''\left(\epsilon(\Z_{\tau^*})+\norm{\X_{Z^*}(Q^*)-\theta^*}^2\right)\right)$ is negligible compared with $(1+\delta_2)\norm{\X_{Z^*}(Q^*)-\theta^*}^2 +M\epsilon(\Z_{\tau^*})$ in all the examples considered in the paper.

When the model is well specified in the sense that $\theta^*=\X_{Z^*}(Q^*)$, Theorem \ref{thm:main} reduces to the following results on posterior contraction.

\begin{corollary} \label{cor:main}
Assume (\ref{eq:larger}), (\ref{eq:subG}) and (\ref{eq:capacity}). For any $\theta^*=\X_{Z^*}(Q^*)$ with any $Z^*\in\bar{\Z}_{\tau^*}$, any $\tau^*\in\T$, any $Q^*\in\mathbb{R}^{\ell(\Z_{\tau^*})}$, any constants $\lambda,\rho>0$ and any sufficiently small constant $\delta\in (0,1)$, there exists some constant $D_{\lambda,\delta,\rho}>0$ only depending on $\lambda,\delta,\rho$, such that
$$\mathbb{E}_{\theta^*}\Pi\left(\epsilon(\Z_{\tau})>(1+\delta)\epsilon(\Z_{\tau^*})\Big|Y\right) \leq \exp\left(-C'\epsilon(\Z_{\tau^*})\right),$$
$$ \mathbb{E}_{\theta^*}\Pi\left(\norm{\X_Z(Q)-\theta^*}^2>M\epsilon(\Z_{\tau^*})\Big|Y\right) \leq \exp\left(-C''\epsilon(\Z_{\tau^*})\right),$$
and
$$\mathbb{E}_{\theta^*}\|\mathbb{E}_{\Pi}(\X_Z(Q)|Y)-\theta^*\|^2\leq M\epsilon(\Z_{\tau^*})+\exp\left(-C'''\epsilon(\Z_{\tau^*})\right)$$
for any constant $D>D_{\lambda,\delta,\rho}$ with some constants $M,C',C'',C'''$ only depending on $\lambda,\delta,\rho,D$.
\end{corollary}

\begin{remark}
The above results hold for all $\epsilon(\Z_{\tau})$ satisfying (\ref{eq:larger}). By choosing $\epsilon(\Z_{\tau})$ at the same order of (\ref{eq:complexity}), we obtain the contraction rate $\ell(\Z_{\tau^*})+\log|\Z_{\tau^*}|$ for the posterior distribution. As we are going to show in the next section, this rate is minimax optimal for all the examples considered in the paper.
From now on, we refer to both (\ref{eq:complexity}) and $\epsilon(\Z_{\tau})$ as the complexity function.
\end{remark}

\begin{remark}\label{rmk:trivial}
By carefully examining the proof, the assumption (\ref{eq:capacity}) can be weakened. In fact, we only require $\left|\left\{\tau\in\T: t-1<\epsilon(\Z_{\tau})\leq t\right\}\right|\leq at^b$ for arbitrary constants $a,b>0$ for the result of Theorem \ref{thm:main} to hold. However, the condition  (\ref{eq:capacity}) is simpler and is sufficient for all the examples considered in the paper. For example, $|\{k\in[n]: t-1<k^2+n\log k\leq t\}|\leq 1$ for stochastic block model,  and $\left|s\in[p]:t-1<2s\log\frac{ep}{s}\leq t\right|\leq 1$ for sparse linear regression.
\end{remark}

\begin{remark}\label{rmk:1+delta}
It is worth noting that the constant $(1+\delta_2)$ in (\ref{eq:main2}) can be arbitrarily close to $1$, as long as $D$ is chosen sufficiently large. Since our procedure involves a model selection step, an oracle inequality with constant exactly $1$ may be impossible, which is suggested by a counter-example in \cite{rigollet2012sparse} for sparse linear regression. 
\end{remark}

\begin{remark}
Note that we do not impose any assumption on the operator $\X_Z(\cdot)$ besides its linearity (\ref{eq:linearity}). In the regression model, this means the results are assumption-free for the design matrix, as those in the frequentist literature \citep{yang1999model,barron1999risk}.
\end{remark}

\section{Applications}\label{sec:app}

\subsection{Stochastic block model}\label{sec:SBM}

The stochastic block model was proposed by \cite{holland1983stochastic} to model random graphs with a community structure. Given a symmetric adjacency matrix $A=A^T\in\{0,1\}^{n\times n}$ that codes an undirected network with no self loop in the sense that $A_{ii}=0$ for all $i\in[n]$, the stochastic block model assumes $\{A_{ij}\}_{i>j}$ are independent Bernoulli random variables with mean $\theta_{ij}=Q_{z(i)z(j)}\in[0,1]$ for some matrix $Q\in[0,1]^{k\times k}$ and some label vector $z\in [k]^n$. In other words, the probability that there is an edge between the $i$th and the $j$th nodes only depends on their community labels $z(i)$ and $z(j)$. Recently, the problem of estimating the success matrix $\theta$  receives some attention. The minimax rate of estimating $\theta$ under the Frobenius norm was established by \cite{gao2014rate}. However, the upper bound in \cite{gao2014rate} was achieved by a procedure assuming the knowledge of the true number of community $k^*$, which is not adaptive. The Bayes framework proposed in this paper provides a natural solution to adaptive estimation for stochastic block model.

Let us write the stochastic block model in a general from as $\theta_{ij}=[\X_Z(Q)]_{ij}=Q_{z(i)z(j)}$ for all $i\neq j$. We do not need to model the diagonal entries because $A_{ii}=0$ for all $i\in[n]$ as convention. Then, $Z=z$, $\tau=k$, $\T=[n]$ and $\Z_k=[k]^n$. Though the true parameter $Q^*$ is symmetric, we do not impose symmetry for the prior distribution. Hence, $\ell(\Z_k)=k^2$ and (\ref{eq:larger}) is satisfied with $\epsilon(\Z_k)=k^2+n\log k$. The general prior distribution $\Pi$ can be specialized to this case as follows,
\begin{enumerate}
\item Sample $k\sim\pi$ from $[n]$, where $\pi(k)\propto\frac{\Gamma(k^2)}{\Gamma(k^2/2)}\exp\left(-D(k^2+n\log k)\right)$;
\item Conditioning on $k$, sample $z$ uniformly from the set $\{z\in [k]^n: \min_{u\in[k]}|\{i\in[n]:z(i)=u\}|>0\}$;
\item Conditioning on $(k,z)$, sample $Q\sim f_{k,z,\lambda}$, where $f_{k,z,\lambda}(Q)\propto e^{-\lambda\sqrt{\sum_{i\neq j}Q_{z(i)z(j)}^2}}$;
\item Set $\theta_{ij}=Q_{z(i)z(j)}$ for all $i\neq j$ and $\theta_{ii}=0$ for all $i\in[n]$.
\end{enumerate}
Note that in Step 2, $\bar{\Z}_k=\{z\in [k]^n: \min_{u\in[k]}|\{i\in[n]:z(i)=u\}|>0\}$. In other words, $\bar{\Z}_k$ is the set of label assignments that induce $k$ clusters. For each $u\in[k]$, $|\{i\in[n]:z(i)=u\}|$ is the size of the $u$th cluster. If for some $u\in[k]$, $|\{i\in[n]:z(i)=u\}|=0$, then there must exists some $k_1<k$ such that $z\in\bar{\Z}_{k_1}$.
Moreover, it is easy to see that for any $z\in\bar{\Z}_k$, $(Q_1)_{z(i)z(j)}=(Q_2)_{z(i)z(j)}$ for all $i\neq j$ implies $Q_1=Q_2$. This indicates that the corresponding linear operator $\X_Z(\cdot)$ is not degenerate.
To help understand the density function $f_{k,z,\lambda}$ in Step 3, consider the case of equal community sizes, i.e., $|\{i\in[n]:z(i)=u\}|=n/k$ for all $u\in[k]$. Then $f_{k,z,\lambda}(Q)\propto e^{-\frac{n\lambda}{k}\fnorm{Q}}$, if we include the diagonal entries and treat $\theta_{ii}$ as $Q_{z(i)z(i)}$.

To study the posterior distribution, we assume that the adjacency matrix is generated by the true mean $\theta_{ij}^*=Q^*_{z^*(i)z^*(j)}=Q^*_{z^*(j)z^*(i)}\in[0,1]$ for $i\neq j$ and $\theta_{ii}^*=0$ for all $i\in[n]$, where $z^*\in\bar{\Z}_{k^*}$ for some $k^*\in[n]$. It can be shown that the noise $W=A-\theta^*$ satisfies (\ref{eq:subG}) for some constant $\rho>0$ by Hoeffding's inequality, and the complexity function $\epsilon(\Z_{\tau})=k^2+n\log k$ satisfies (\ref{eq:capacity}). Hence, Corollary \ref{cor:main} can be specialized for the stochastic block model as follows.
\begin{corollary}\label{cor:SBM}
For any $\theta^*$ and $k^*$ specified above, any constant $\lambda>0$ and any sufficiently small constant $\delta\in(0,1)$, there exists some constant $D_{\lambda,\delta}>0$ only depending on $\lambda,\delta$ such that
$$\mathbb{E}_{\theta^*}\Pi\left(k^2+n\log k>(1+\delta)\left((k^*)^2+n\log k^*\right)\Big|A\right)\leq \exp\left(-C'((k^*)^2+n\log k^*)\right)$$
and
$$\mathbb{E}_{\theta^*}\Pi\left(\fnorm{\theta-\theta^*}^2>M((k^*)^2+n\log k^*)\Big|A\right)\leq \exp\left(-C''((k^*)^2+n\log k^*)\right)$$
for any constant $D>D_{\lambda,\delta}$ with some constants $M,C',C''$ only depending on $\lambda,\delta,D$.
\end{corollary}

A previous result on Bayes estimation for the stochastic block model by \cite{pati2015optimal} assumes the knowledge of $k^*$, and the rate is sub-optimal.
To the best of our knowledge, our result is the first adaptive Bayes estimator for the stochastic block model with a posterior contraction rate $(k^*)^2+n\log k^*$, which is optimal according to \cite{gao2014rate}. When $k^*\leq\sqrt{n\log n}$, the rate is dominated by $n\log k^*$, which grows only logarithmically as $k^*$ grows. When $k^*>\sqrt{n\log n}$, the rate is dominated by $(k^*)^2$, corresponding to the number of parameters. Corollary \ref{cor:main} also implies that the posterior mean achieves the minimax rate $(k^*)^2+n\log k^*$.

While our result uses a prior distribution that does not impose symmetry on the mean matrix $\theta$, it may be more desirable to incorporate symmetry from a practical point of view. This can be achieved within our framework of structured linear models. To be specific, we can consider the object $\X_Z(Q)$ to be a triangle array with entries $\{[\X_Z(Q)]_{ij}: 1\leq i<j\leq n\}$. Then, $Q$ is also a triangle array, but it is of dimension $k(k+1)/2$ and has entries $\{Q_{ij}: 1\leq i\leq j\leq k\}$. The linear operator $\X_Z(\cdot)$ that maps from $Q$ to $\X_Z(Q)$ is specified by $\X_Z(Q)=Q_{z(i)z(j)}$. In other words, the symmetric SBM is also a special case of our structured linear models with $Z=z$, $\tau=k$, $\T=[n]$, $\Z_k=[k]^n$, $\ell(\Z_k)=k(k+1)/2$ and $N=n(n-1)/2$.

To close this section, we also mention an important problem of community detection, which is equivalent to estimating the structure $Z$ in our general framework. The posterior distribution of Bayesian community detection was recently analyzed by \cite{van2016bayesian}.

\subsection{Biclustering}

The biclustering model, originated in \cite{hartigan1972direct}, can be viewed as a precursor and an asymmetric extension of the stochastic block model. The data matrix $Y\in\mathbb{R}^{n\times m}$ is  generated by a signal matrix $\theta=(\theta_{ij})$ with $\theta_{ij}=Q_{z_1(i)z_2(j)}$ for some label vectors $z_1\in[k]^n$ and $z_2\in[l]^m$, i.e., the rows of $\theta$ have $k$ clusters and the columns of $\theta$ have $l$ clusters, and the values of $(\theta_{ij})$ that belong to the same row-cluster and the same column-cluster are identical. The goal is to recover the true signal matrix $\theta^*$ from the observation $Y$.

To put the biclustering model in our general framework, we have $Z=(z_1,z_2)$, $\tau=(k,l)$, $\T=[n]\times [m]$, $\Z_{k,l}=[k]^n\times [l]^m$ and $\ell(\Z_{n,l})=kl$. Moreover, the complexity function is $\epsilon(\Z_{k,l})=kl+k\log n+l\log m$, which satisfies (\ref{eq:larger}) and (\ref{eq:capacity}). The general prior $\Pi$ can be specialized to this case as follows,
\begin{enumerate}
\item Sample $(k,l)\sim\pi$ from $[n]\times[m]$, where $\pi(k,l)\propto \frac{\Gamma(kl)}{\Gamma(kl/2)}\exp\left(-D(kl+n\log k+m\log l)\right)$;
\item Conditioning on $(k,l)$, sample $(z_1,z_2)$ uniformly from $\bar{\Z}_{k,l}$;
\item Conditioning on $(k,l,z_1,z_2)$, sample $Q\sim f_{k,l,z_1,z_2,\lambda}$ with $f_{k,l,z_1,z_2,\lambda}(Q)\propto e^{-\lambda\sqrt{\sum_{ij}Q_{z_1(i)z_2(j)}^2}}$;
\item Set $\theta_{ij}=Q_{z_1(i)z_2(j)}$ for all $(i,j)$.
\end{enumerate}
In Step 2,
$$\bar{\Z}_{k,l}=\left\{(z_1,z_2)\in[k]^n\times [l]^m:\min_{u\in[k]}|\{i\in[n]:z_1(i)=u\}|>0, \min_{v\in[l]}|\{j\in[m]:z_2(j)=v\}|>0\right\}.$$
In other words, for any $(z_1,z_2)\in\bar{\Z}_{k,l}$, $z_1$ and $z_2$ induce row and column clustering structures with numbers of clusters being $k$ and $l$, respectively.

To analyze the posterior distribution, assume $Y=\theta^*+W$, where $\theta_{ij}^*=Q^*_{z_1^*(i)z_2^*(j)}$ for $Q^*\in\mathbb{R}^{k^*\times l^*}$ and $(z_1^*,z_2^*)\in[k^*]^n\times [l^*]^m$, and the noise $W$ is assumed to satisfy (\ref{eq:subG}).
\begin{corollary}\label{cor:biclustering}
For any $\theta^*$ and $(k^*,l^*)$ specified above, any constants $\lambda,\rho>0$ and any sufficiently small constant $\delta\in(0,1)$, there exists some constant $D_{\lambda,\delta,\rho}>0$ only depending on $\lambda,\delta,\rho$ such that
\begin{eqnarray*}
&& \mathbb{E}_{\theta^*}\Pi\left(kl+n\log k+m\log l>(1+\delta)\left(k^*l^*+n\log k^*+m\log l^*\right)\Big|Y\right)\\
&\leq& \exp\left(-C'(k^*l^*+n\log k^*+m\log l^*)\right)
\end{eqnarray*}
and
$$\mathbb{E}_{\theta^*}\Pi\left(\fnorm{\theta-\theta^*}^2>M(k^*l^*+n\log k^*+m\log l^*)\Big|Y\right)\leq \exp\left(-C''(k^*l^*+n\log k^*+m\log l^*)\right)$$
for any constant $D>D_{\lambda,\delta,\rho}$ with  some constants $M,C',C''$ only depending on $\lambda,\delta,\rho,D$.
\end{corollary}
The posterior contraction rate for recovering a signal matrix with a biclustering structure is $k^*l^*+n\log k^*+m\log l^*$, which is minimax optimal according to \cite{gao2014rate}. To the best of our knowledge, this is the first adaptive estimation result for biclustering with an optimal rate.

\subsection{Sparse linear regression} \label{sec:slr}

Consider a regression problem with fixed design $X\beta$, where $X\in\mathbb{R}^{n\times p}$ and $\beta\in\mathbb{R}^{p}$. The regression coefficient is assumed to be sparse so that $\beta^T=(\beta_S^T,0_{S^c}^T)$ for some $S\subset [p]$. Recovering the mean vector $X\beta$ and the regression vector $\beta$ with a sparse prior has been considered in \cite{castillo2014}. However, the results of \cite{castillo2014} imposed a stronger assumption that is used for the Lasso estimator \cite{bickel2009simultaneous}. In this section, we show that the prior distribution that we propose in Section \ref{sec:prior} leads to optimal posterior contraction rates with minimal assumptions.

The sparse linear regression model is a special case of the general structured linear model (\ref{eq:SLMs}) with $Z=S$, $\tau=s$, $\T=[p]$, $\Z_s=\{S\subset[p]:|S|=s\}$, $\ell(\Z_s)=s$ and $Q=\beta_S$. Then, we have the representation $\X_Z(Q)=X_{*S}\beta_S=X\beta$. Since $\log|\Z_s|=\log{p\choose s}\leq s\log\frac{ep}{s}$, the complexity function $\epsilon(\Z_s)=2s\log\frac{ep}{s}$ satisfies the condition (\ref{eq:larger}). It can be shown that $\epsilon(\Z_{\tau})$ satisfies (\ref{eq:capacity}). We specialize the general prior $\Pi$ in Section \ref{sec:prior} as follows,
\begin{enumerate}
\item Sample $s\sim \pi$ from $[p]$, where $\pi(s)\propto \frac{\Gamma(s)}{\Gamma(s/2)}\exp\left(-2Ds\log\frac{ep}{s}\right)$;
\item Conditioning on $s$, sample $S$ uniformly from $\{S\subset[p]: |S|=s, \det(X_{*S}^TX_{*S})>0\}$;
\item Conditioning on $(s,S)$, sample $\beta_S\sim f_{s,S,\lambda}$ with $f_{s,S,\lambda}(\beta_S)\propto e^{-\lambda\norm{X_{*S}\beta_S}}$ and set $\beta_{S^c}=0$.
\end{enumerate}
In Step 1, we set $\epsilon(\Z_s)=2s\log\frac{ep}{s}$ instead of the exact form of $\ell(\Z_{\tau})+\log|\Z_{\tau}|$ in the exponent for simplicity.
In Step 2, we sample $S$ from the set $\bar{\Z}_s=\{S\subset[p]: |S|=s, \det(X_{*S}^TX_{*S})>0\}$ instead of $\Z_s$ such that the density $f_{s,S,\lambda}$ in Step 3 is not degenerate. Since $X_{*S}\in\mathbb{R}^{n\times s}$, when $s>n$, we have $\bar{\Z}_s=\varnothing$. Note that the exponent on the density of $\beta_S$ is $-\lambda\norm{X_{*S}\beta_S}$, different from $-\lambda\norm{\beta_S}_1$ in \cite{castillo2014}. We allow the prior to depend on the design matrix $X$ to obtain an assumption-free optimal posterior  prediction rate. The idea of design-dependent prior was also employed by \cite{martin2014empirical} in an empirical pseudo-Bayes framework. Since $e^{-\lambda\norm{X_{*S}\beta_S}}$ has an exponential tail, it is capable of modeling a large regression coefficient. We expect that an elliptical distribution with heavier tails than Laplace also works here.

The prior distribution involves a correction factor $\frac{\Gamma(s)}{\Gamma(s/2)}$ in the model selection step to compensate the normalizing constant of the elliptical Laplace distribution. Without this factor, $\exp\left(-2Ds\log\frac{ep}{s}\right)$ is the common prior distribution on the model dimension used in \cite{rigollet2011exponential,castillo2012needles,gao2013rate,castillo2014,martin2014empirical}. Since $\exp\left(-2Ds\log\frac{ep}{s}\right)$ is a decreasing function of $s$, it gives less weights for more complex models. However, with the correction factor, $\pi(s)\propto \frac{\Gamma(s)}{\Gamma(s/2)}\exp\left(-2Ds\log\frac{ep}{s}\right)$ is not necessarily a decreasing function of $s$. For a large $D>0$, it can be shown that $\pi(\sqrt{p})<\pi(p)$, which leads to a counter-intuitive prior modeling strategy. Nevertheless, it is worth noting that the $\pi$ in Step 1 is only part of the prior $\Pi$. The elliptical Laplace distribution used later also contributes to the prior modeling on the dimension. The combination of the two gives a correct prior weight on the model dimension.

Let  $Y=X\beta^*+W$ for some $\beta^*$ with support $S^*$ and sparsity $|S^*|=s^*$, where the noise vector $W$ is assumed to be sub-Gaussian (\ref{eq:subG}). Without loss of generality, we may assume $S^*\in \bar{\Z}_{s^*}$. If $X_{*S^*}$ is collinear in the sense that $\det(X_{*S^*}^TX_{*S^*})=0$, there always exists a $\beta_1$ with support $S_1$ and sparsity $s_1=|S_1|$ such that $X\beta^*=X\beta_1$ and $\det(X_{*S_1}^TX_{*S_1})>0$. We may simply redefine $(s^*,S^*)$ by $(s_1,S_1)$.

\begin{corollary}\label{cor:slr-pred}
For any $\beta^*$, $S^*\in\bar{\Z}_{s^*}$ and $s^*$ specified above, any constants $\lambda,\rho>0$ and any sufficiently small constant $\delta\in(0,1)$, there exists some constant $D_{\lambda,\delta,\rho}>0$ only depending on $\lambda,\delta,\rho$ such that
\begin{equation}
\mathbb{E}_{X\beta^*}\Pi\left(s>(1+\delta)s^*\Big|Y\right)\leq \exp\left(-C's^*\log\frac{ep}{s^*}\right)\label{eq:slr-comp}
\end{equation}
and
\begin{equation}
\mathbb{E}_{X\beta^*}\Pi\left(\norm{X\beta-X\beta^*}^2>Ms^*\log\frac{ep}{s^*}\Big|Y\right)\leq\exp\left(-C''s^*\log\frac{ep}{s^*}\right)\label{eq:slr-pred}
\end{equation}
for any constant $D>D_{\lambda,\delta,\rho}$ with some constants $M,C',C''$ only depending on $\lambda,\delta,\rho,D$.
\end{corollary}
The result (\ref{eq:slr-comp}) is a consequence of (\ref{eq:main1}) since $s\log\frac{ep}{s}> (1+\delta_1)s^*\log\frac{ep}{s^*}$ is equivalent to $s>(1+\delta)s^*$. It improves the corresponding bounds in \cite{castillo2012needles,castillo2014} at a constant level. The result (\ref{eq:slr-pred}) achieves the minimax optimal prediction rate with no assumption on the design matrix $X$, which is comparable to the frequentist result in \cite{birge2001gaussian}. A slight improvement of (\ref{eq:slr-pred}) will be discussed in Section \ref{sec:agg}.

Besides the optimal prediction rate, we are ready to obtain optimal estimation rates given (\ref{eq:slr-comp}) and (\ref{eq:slr-pred}). Define
\begin{equation}
\kappa_2=\min_{\{b\neq 0:\norm{b}_0\leq (2+\delta)s^*\}}\frac{\norm{Xb}}{\sqrt{n}\norm{b}}\quad\text{and}\quad\kappa_1=\min_{\{b\neq 0:\norm{b}_0\leq (2+\delta)s^*\}}\frac{\sqrt{s^*}\norm{Xb}}{\sqrt{n}\norm{b}_1}.\label{eq:RE21}
\end{equation}
Note that $\kappa_2$ is the restricted eigenvalue constant \citep{candes2005decoding,bickel2009simultaneous} and $\kappa_1$ is the compatibility constant \citep{buhlmann2011statistics}.
\begin{corollary}\label{cor:slr-esti}
Under the setting of Corollary \ref{cor:slr-pred}, we have
$$\mathbb{E}_{X\beta^*}\Pi\left(\norm{\beta-\beta^*}^2>M\frac{s^*\log\frac{ep}{s^*}}{n\kappa_2^2}\Big|Y\right)\leq 2\exp\left(-(C'+ C'')s^*\log\frac{ep}{s^*}\right)$$
and
$$\mathbb{E}_{X\beta^*}\Pi\left(\norm{\beta-\beta^*}_1^2>M\frac{(s^*)^2\log\frac{ep}{s^*}}{n\kappa_1^2}\Big|Y\right)\leq 2\exp\left(-(C'+ C'')s^*\log\frac{ep}{s^*}\right)$$
for the same constants $M,C',C''$ in Corollary \ref{cor:slr-pred}.
\end{corollary}
Compared with the minimax rates \citep{donoho1994minimax,verzelen2012minimax}, Corollary \ref{cor:slr-esti} obtains optimal estimation rates for both $\ell_2$ and $\ell_1$ loss functions. Moreover, the dependence on the quantities $\kappa_2$ and $\kappa_1$ are optimal \citep{raskutti2011minimax}, compared with the Lasso estimator and the spike and slab prior \citep{castillo2014}. When $\kappa\asymp\kappa_1\asymp\kappa_2$, the rates of the Lasso estimator are $\frac{s^*\log p}{n\kappa^4}$ and $\frac{(s^*)^2\log p}{n\kappa^4}$ for the loss $\norm{\cdot}^2$ \citep{bickel2009simultaneous} and the loss $\norm{\cdot}_1^2$ \citep{van2009conditions}, respectively, and the rates of the spike and slab prior are $\frac{s^*\log\frac{ep}{s^*}}{n\kappa^6}$ and $\frac{(s^*)^2\log\frac{ep}{s^*}}{n\kappa^8}$ for the loss $\norm{\cdot}^2$ and $\norm{\cdot}^2_1$ \citep{castillo2014}, respectively.

The results on $\ell_{\infty}$ convergence and model selection consistency for sparse linear regression are not implied by the general theory. We are going to treat it separately in Section \ref{sec:more}.

To close this section, we briefly discuss the computational issue of the proposed prior distribution. A recent theoretical result by \cite{yang2016computational} shows that the mixing-time of a simple MCMC algorithm is polynomial in the setting of Bayesian sparse linear regression. They also use a two-step model selection prior, but the distribution on model parameters is $e^{-\lambda \|X_{*S}\beta_S\|^2}$, compared with our $e^{-\lambda \|X_{*S}\beta_S\|}$. Given the similarity between the two prior distributions, it is conceivable that similar results in \cite{yang2016computational} can also be established in our setting. More interestingly, whether a general theory of computation can be established under our framework of structured linear models will be an important topic to study in the future.

\subsection{Multiple linear regression with group sparsity}\label{sec:group}

Let us consider a multiple regression setting $XB$ for $X\in\mathbb{R}^{n\times p}$ and $B\in\mathbb{R}^{p\times m}$. The matrix $B$ collects regression coefficients from $m$ regression problems. We assume the $m$ regression coefficients share the same support. There is some $S\subset[p]$ such that $B_{S^c*}=0$, i.e., $S$ is the nonzero rows of $B$. The concept of group sparsity was proposed by \cite{bakin1999adaptive,yuan2006model}, and frequentist statistical properties were analyzed by \cite{lounici2011oracle}.

To put the problem in the general framework, we have $Z=S$, $\tau=s$, $\T=[p]$, $\Z=\{S\subset[p]:|S|=s\}$, $\ell(\Z_s)=ms$ and $Q=B_{S*}$. Then,  $\X_Z(Q)=X_{*S}B_{S*}=XB$. The choice $\epsilon(\Z_s)= s\left(m+\log\frac{ep}{s}\right)$  satisfies the conditions (\ref{eq:larger}) and (\ref{eq:capacity}). The prior distribution $\Pi$ is similar to that used in Section \ref{sec:slr},
\begin{enumerate}
\item Sample $s\sim \pi$ from $[p]$, where $\pi(s)\propto \frac{\Gamma(s)}{\Gamma(s/2)}\exp\left(-Ds\left(m+\log\frac{ep}{s}\right)\right)$;
\item Conditioning on $s$, sample $S$ uniformly from $\bar{\Z}_s=\{S\subset[p]: |S|=s, \det(X_{*S}^TX_{*S})>0\}$;
\item Conditioning on $(s,S)$, sample $B_{S*}\sim f_{s,S,\lambda}$ with $f_{s,S,\lambda}(B_{S*})\propto e^{-\lambda\fnorm{X_{*S}B_{S*}}}$ and set $B_{S^c*}=0$.
\end{enumerate}
Note that we sample $S$ from $\bar{\Z}_s$ in Step 2 as for sparse linear regression. Assume the data is generated by $Y=XB^*+W$ for some matrix $B^*$ with support $S^*$ and sparsity $s^*$. Without loss of generality, we assume $S^*\in\bar{\Z}_{s^*}$. The noise matrix $W$ is assumed to be the sub-Gaussian in the sense of (\ref{eq:subG}).

\begin{corollary}\label{cor:gslr-pred}
For any $B^*$, $S^*\in\bar{\Z}_{s^*}$ and $s^*$ specified above, any constants $\lambda,\rho>0$ and any sufficiently small constant $\delta\in(0,1)$, there exists some constant $D_{\lambda,\delta,\rho}>0$ only depending on $\lambda,\delta,\rho$ such that
$$
\mathbb{E}_{XB^*}\Pi\left(s>(1+\delta)s^*\Big|Y\right)\leq \exp\left(-C's^*\left(m+\log\frac{ep}{s^*}\right)\right)$$
and
$$
\mathbb{E}_{XB^*}\Pi\left(\fnorm{XB-XB^*}^2>Ms^*\left(m+\log\frac{ep}{s^*}\right)\Big|Y\right)\leq\exp\left(-C''s^*\left(m+\log\frac{ep}{s^*}\right)\right)$$
for any constant $D>D_{\lambda,\delta,\rho}$ with some constants $M,C',C''$ only depending on $\lambda,\delta,\rho,D$.
\end{corollary}

The posterior contraction rate for the prediction loss is $s^*\left(m+\log\frac{ep}{s^*}\right)$, which is minimax optimal according to \cite{lounici2011oracle,ma2013volume}. Posterior contraction for various estimation loss functions can also be derived in a similar way as in Section \ref{sec:slr}.

\subsection{Multi-task learning}\label{sec:mtl}

Multi-task learning is another name for multiple linear regression in the form of $XB$ with $X\in\mathbb{R}^{n\times p}$ and $B\in\mathbb{R}^{p\times m}$. Compared with $m$ independent linear regression problems, a typical multi-task learning setting assumes some dependent structure among the columns of the coefficient matrix $B$. The group sparsity assumption considered in Section \ref{sec:group} is an example where the columns of $B$ share the same support.

In this section, we consider another special but important class of multi-task learning problems. We assume a clustering structure among the columns of $B$, i.e., $B_{*j}=Q_{*z(j)}$ for some $z\in[k]^m$ and $Q\in\mathbb{R}^{p\times k}$. In other words, the $m$ regression coefficient vectors are allowed to choose from $k$ possibilities. When the design $X$ is an identity matrix, it reduces to an ordinary clustering problem.

Let us write the multi-task learning problem in the general framework. This can be done by letting $Z=z$, $\tau=k$, $\T=[m]$, $\Z_k=[k]^m$ and $\ell(\Z_k)=pk$. Moreover, we have the representation $[\X_Z(Q)]_{*j}=XQ_{*z(j)}$. The complexity function $\epsilon(\Z_{\tau})=pk+m\log k$ satisfies the conditions (\ref{eq:larger}) and (\ref{eq:capacity}).  We consider a full rank design matrix with $\det(X^TX)>0$. The general prior distribution $\Pi$ in Section \ref{sec:prior} can be specialized to this case,
\begin{enumerate}
\item Sample $k\sim \pi$ from $[p]$, where $\pi(k)\propto\frac{\Gamma(pk)}{\Gamma(pk/2)}\exp\left(-D(pk+m\log k)\right)$;
\item Conditioning on $k$, sample $z$ uniformly from the set $\{z\in [k]^m: \min_{u\in[k]}|\{j\in[m]:z(j)=u\}|>0\}$;
\item Conditioning on $(k,z)$, sample $Q\sim f_{k,z,\lambda}$ with $f_{k,z,\lambda}(Q)\propto e^{-\lambda\sqrt{\sum_j\norm{XQ_{z(j)*}}^2}}$;
\item Set $B_{*j}=Q_{*z(j)}$ for all $j\in[m]$.
\end{enumerate}
Note that in Step 2, we have $\bar{\Z}_k=\{z\in [k]^m: \min_{u\in[k]}|\{j\in[m]:z(j)=u\}|>0\}$, which is due to $\det(X^TX)>0$. The full rankness of the design matrix implicitly implies $p\leq n$. In fact, there is no loss to assume $\det(X^TX)>0$, because whenever $\det(X^TX)=0$, one can simply use a subset of the variables that are linearly independent without affecting the prediction error.

We assume that the data is generated as $Y=XB^*+W$ for some matrix $B^*$ satisfying $B^*_{*j}=Q^*_{*z^*(j)}$ with some $Q^*$ and $z^*\in[k^*]^m$. The noise matrix is assumed to satisfy (\ref{eq:subG}).
\begin{corollary}\label{cor:mtl}
For any $B^*$ and $k^*$ specified above, any constants $\lambda,\rho>0$ and any sufficiently small constant $\delta\in(0,1)$, there exists some constant $D_{\lambda,\delta,\rho}>0$ only depending on $\lambda,\delta,\rho$ such that
$$
\mathbb{E}_{XB^*}\Pi\left(pk+m\log k>(1+\delta)(pk^*+m\log k^*)\Big|Y\right)\leq \exp\left(-C'(pk^*+m\log k^*)\right)$$
and
$$
\mathbb{E}_{XB^*}\Pi\left(\fnorm{XB-XB^*}^2>M(pk^*+m\log k^*)\Big|Y\right)\leq\exp\left(-C''(pk^*+m\log k^*)\right)$$
for any constant $D>D_{\lambda,\delta,\rho}$ with some constants $M,C',C''$ only depending on $\lambda,\delta,\rho,D$.
\end{corollary}

The posterior contraction rate for multi-task learning is $pk^*+m\log k^*$.  According to \cite{lu2015}, the rate $pk^*+m\log k^*$ is minimax optimal when $pk^*+m\log k^*\leq pm$. The minimax rate for the case $pk^*+m\log k^*>pm$ is simply $pm$, the dimension of $B$. In that case, even the ordinary least-squares estimator $\hat{B}=\argmin_{B}\fnorm{Y-XB}^2$ can achieve the rate.

\subsection{Dictionary learning}\label{sec:dicL}

Dictionary learning  can be viewed as a linear regression problem without knowing the design matrix. Mathematically, the signal matrix $\theta\in\mathbb{R}^{n\times d}$ can be represented as $\theta=QZ$ for some $Q\in\mathbb{R}^{n\times p}$ and $Z\in\mathbb{R}^{p\times d}$. Both the dictionary $Q$ and the coefficient matrix $Z$ are unknown. A common assumption is that each column of $Z$ is sparse, and the goal is to learn the latent sparse representation of the signal. The problem is also referred to as sparse coding \citep{olshausen1996emergence}. Recently, the minimax rate of dictionary learning has been established by \cite{lu2015} for estimating the true signal matrix $\theta^*$. In this section, we provide a Bayes solution to the adaptive estimation problem of dictionary learning. Following \cite{agarwal2013exact}, we consider a discrete version of the problem. Namely, $Z\in\{-1,0,1\}^{p\times d}$. Then, the problem can be represented in a general form by letting $\tau=(p,s)$, $\T=\left\{(p,s)\in[n\wedge d]\times[n]: s\leq p\right\}$, $\Z_{p,s}=\{Z\in\{-1,0,1\}^{p\times d}:\max_{j\in[d]}|\supp(Z_{*j})|\leq s\}$ and $\ell(\Z_{p,s})=np$. Moreover, we have the representation $\X_Z(Q)=QZ$. The complexity function is $\ell(\Z_{p,s})+\log|\Z_{p,s}|=np+d\left(\log{p\choose s}+3\log s\right)$. With $\epsilon(\Z_{p,s})=3\left(np+ds\log\frac{ep}{s}\right)$, (\ref{eq:larger}) and (\ref{eq:capacity}) are satisfied. The general prior distribution $\Pi$ can be specialized into the following sampling procedures:
\begin{enumerate}
\item Sample $(p,s)\sim \pi$ from $\T$ with $\pi(p,s)\propto \frac{\Gamma(np)}{\Gamma(np/2)}\exp\left(-3D\left(np+ds\log\frac{ep}{s}\right)\right)$;
\item Given $(p,s)$, sample $Z$ uniformly from $\bar{\Z}_{p,s}=\left\{Z\in\Z_{p,s}: \det(ZZ^T)>0\right\}$;
\item Given $(p,s,Z)$, sample $Q\sim f_{p,s,Z,\lambda}$ with $f_{p,s,Z,\lambda}(Q)\propto e^{-\lambda\fnorm{QZ}}$;
\item Set $\theta=QZ$.
\end{enumerate}
Note that we have used $\epsilon(\Z_{p,s})=3\left(np+ds\log\frac{ep}{s}\right)$ instead of the exact $\ell(\Z_{\tau})+\log|\Z_{\tau}|$ in Step 1 for simplicity.

We assume that the data is generated by $Y=\theta^*+W$ for some noise matrix $W$ satisfying (\ref{eq:subG}) and $\theta^*=Q^*Z^*$. Without loss of generality, we assume the matrix $Z^*$ belongs to the set $\bar{\Z}_{p^*,s^*}$. If $\det(Z^*(Z^*)^T)=0$, there must exist some $Q_1\in\mathbb{R}^{n\times p_1}$ and $Z_1\in\bar{\Z}_{p_1,s_1}$ such that $\theta^*=Q^*Z^*=Q_1Z_1$.

\begin{corollary}\label{cor:dic}
For any $\theta^*=Q^*Z^*$ with $Z^*\in\bar{\Z}_{p^*,s^*}$ specified above, any constants $\lambda,\rho>0$ and any sufficiently small constant $\delta\in(0,1)$, there exists some constant $D_{\lambda,\delta,\rho}>0$ only depending on $\lambda,\delta,\rho$ such that
$$
\mathbb{E}_{\theta^*}\Pi\left(np+ds\log\frac{ep}{s}>(1+\delta)\left(np^*+ds^*\log\frac{ep^*}{s^*}\right)\Big|Y\right)\leq \exp\left(-C'\left(np^*+ds^*\log\frac{ep^*}{s^*}\right)\right)$$
and
$$
\mathbb{E}_{\theta^*}\Pi\left(\fnorm{\theta-\theta^*}^2>M\left(np^*+ds^*\log\frac{ep^*}{s^*}\right)\Big|Y\right)\leq\exp\left(-C''\left(np^*+ds^*\log\frac{ep^*}{s^*}\right)\right)$$
for any constant $D>D_{\lambda,\delta,\rho}$ with some constants $M,C',C''$ only depending on $\lambda,\delta,\rho,D$.
\end{corollary}
The rate we have obtained from (\ref{cor:dic}) is $np^*+ds^*\log\frac{ep^*}{s^*}$, which is minimax optimal when $np^*+ds^*\log\frac{ep^*}{s^*}\leq nd$ according to \cite{lu2015}. When $np^*+ds^*\log\frac{ep^*}{s^*}> nd$, the minimax rate is just $nd$, the dimension of $\theta$. It can be achieved by the naive estimator $\hat{\theta}=Y$, and thus this is not an interesting case to us. The term $ds^*\log\frac{ep^*}{s^*}$ in the rate is the error for recovering $d$ sparse regression coefficient vectors, and $np^*$ is the price to pay for not knowing the design matrix $Q^*$. The result can be extended to the case where the entries of $Z^*$ are allowed to take values in an arbitrary discrete set with finite cardinality. To the best of our knowledge, this is the first adaptive estimation result for dictionary learning with an optimal prediction rate.

\subsection{Nonparametric graphon estimation}\label{sec:graphon}

Consider a random graph with adjacency matrix $\{A_{ij}\}\in\{0,1\}^{n\times n}$, whose sampling procedure is determined by
\begin{equation}
(\xi_1,...,\xi_n)\sim\mathbb{P}_{\xi},\quad A_{ij}|(\xi_i,\xi_j)\sim \text{Bernoulli}(\theta^*_{ij}),\quad\text{where }\theta^*_{ij}=f^*(\xi_i,\xi_j).\label{eq:graphon-gen}
\end{equation}
For $i\in[n]$, $A_{ii}=\theta_{ii}^*=0$. Conditioning on $(\xi_1,...,\xi_n)$, $A_{ij}=A_{ji}$ is independent across $i>j$. The function $f^*$ on $[0,1]^2$, which is assumed to be symmetric, is called graphon. The concept of graphon is originated from graph limit theory \citep{hoover79,lovasz06,diaconis07,lovasz12} and the studies of exchangeable arrays \citep{aldous81,kallenberg89}. It is the underlying nonparametric object that generates the random graph.

Let us proceed to specify the function class of graphons. Define the derivative operator by
$$\nabla_{jk}f(x,y)=\frac{\partial^{j+k}}{(\partial x)^j(\partial y)^k}f(x,y),$$
and we adopt the convention $\nabla_{00}f(x,y)=f(x,y)$.
The H\"{o}lder norm is defined as
$$||f||_{\mathcal{H}_{\alpha}}=\max_{j+k\leq\floor{\alpha}}\sup_{x,y\in\mathcal{D}}\left|\nabla_{jk}f(x,y)\right|+\max_{j+k=\floor{\alpha}}\sup_{(x,y)\neq (x',y')\in\mathcal{D}}\frac{\left|\nabla_{jk}f(x,y)-\nabla_{jk}f(x',y')\right|}{||(x-x',y-y')||^{\alpha-\floor{\alpha}}},$$
where $\mathcal{D}=\{(x,y)\in[0,1]^2:x\geq y\}$. Then, the graphon class with H\"{o}lder smoothness $\alpha$ is defined by
$$\mathcal{F}_{\alpha}(L)=\left\{0\leq f\leq 1: \norm{f}_{\mathcal{H}_{\alpha}}\leq L, f(x,y)=f(y,x)\text{ for all }x\in\mathcal{D}\right\},$$
where $L>0$ is the radius of the class, which is assumed to be a constant. Recently, a minimax optimal estimator of  $f^*$ was proposed by \cite{gao2014rate} given the knowledge of $\alpha$. In this section, we solve the adaptive graphon estimation problem via a Bayes procedure.

As shown in \cite{gao2014rate}, it is sufficient to approximate a graphon with H\"{o}lder smoothness by a blockwise constant function. In the random graph setting, a blockwise constant function is the stochastic block model. Therefore, we apply the prior distribution in Section \ref{sec:SBM} by equating $f(\xi_i,\xi_j)=\theta_{ij}$. The oracle inequality in Theorem \ref{thm:main} gives the desired bias-variance tradeoff of the problem.

\begin{corollary}\label{cor:graphon}
Consider the prior distribution specified in Section \ref{sec:SBM}.
For the class $\mathcal{F}_{\alpha}(L)$ with $\alpha,L>0$ defined above and any constant $\lambda>0$, there exists some constant $D_{\lambda}>0$ only depending on $\lambda$ such that
\begin{eqnarray*}
&& \sup_{f^*\in\mathcal{F}_{\alpha}(L)}\sup_{\mathbb{P}_{\xi}}\mathbb{E}_{f^*}\Pi\left(\frac{1}{n^2}\sum_{i,j\in[n]}\left(f(\xi_i,\xi_j)-f^*(\xi_i,\xi_j)\right)^2>M\left(n^{-\frac{2\alpha}{\alpha+1}}+\frac{\log n}{n}\right)\Big| A\right) \\
&\leq& \exp\left(-C'\left(n^{\frac{1}{\alpha+1}}+n\log n\right)\right)
\end{eqnarray*}
for any constant $D>D_{\lambda}$ with some constants $M,C'$ only depending on $\lambda,D,L$.
\end{corollary}
\begin{remark}
The expectation in Corollary \ref{cor:graphon} is associated with the joint distribution (\ref{eq:graphon-gen}) over both $\{A_{ij}\}$ and $\{\xi_i\}$. Moreover, we do not need any assumption on the distribution on $\{\xi_i\}$, and the result of Corollary \ref{cor:graphon} holds uniformly over all $\mathbb{P}_{\xi}$.
\end{remark}

The posterior contraction rate we have obtained for graphon estimation is $n^{-\frac{2\alpha}{\alpha+1}}+\frac{\log n}{n}$, which is minimax optimal for the worst-case design according to \cite{gao2014rate}. When $\alpha\in (0,1)$, the rate is dominated by $n^{-\frac{2\alpha}{\alpha+1}}$, which is the typical two-dimensional nonparametric regression rate. When $\alpha\geq 1$, the rate becomes $\frac{\log n}{n}$, which does not depend on $\alpha$ anymore. The key difference between graphon estimation and nonparametric regression lies in the knowledge of the design sequence $\{\xi_i\}$. A nonparametric regression problem observes the pair $\{(\xi_i,\xi_j),A_{ij}\}$, while graphon estimation only observes the adjacency matrix $\{A_{ij}\}$, resulting in an extra term $\frac{\log n}{n}$ in the rate. To the best of our knowledge, Corollary \ref{cor:graphon} is the first adaptive estimation result on graphon estimation with an optimal convergence rate.

\subsection{Aggregation} \label{sec:agg}

Aggregation in nonparametric regression has been considered by \cite{yang2000combining,nemirovski2000topics,tsybakov2003optimal,catoni2004statistical,yang2004aggregating,leung2006information} among others. Let us start with the nonparametric regression setting with fixed design. The data is generated by
\begin{equation}
Y_i=f^*(x_i)+W_i,\quad i=1,...,n,\label{eq:Y-agg}
\end{equation}
where the noise vector $W=\{W_i\}$ is assumed to satisfy (\ref{eq:subG}). The goal of nonparametric regression is to estimate the true regression function $f^*$ by some estimator $\hat{f}$ under the loss
$$\norm{\hat{f}-f}_n^2=\frac{1}{n}\sum_{i=1}^n\left(\hat{f}(x_i)-f^*(x_i)\right)^2,$$
where $\norm{\cdot}_n$ stands for the empirical $\ell_2$ norm. Assume we are given a collection of functions $\{f_1,...,f_p\}$, called the dictionary. Given a subset $\Theta\subset\mathbb{R}^p$, for $\beta\in\Theta$, define $f_{\beta}=\sum_{j=1}^p\beta_jf_j$. The goal of aggregation is to find an estimator $\hat{f}$ such that its error $\norm{\hat{f}-f^*}_n^2$ is comparable to that given by the best among the class $\{f_{\beta}:\beta\in\Theta\}$. To be specific, one seeks an estimator $\hat{f}$ to satisfy the following oracle inequality,
\begin{equation}
\norm{\hat{f}-f^*}^2_n\leq (1+\delta)\inf_{\beta\in\Theta}\norm{f_{\beta}-f^*}_n^2+\Delta_{n,p}(\Theta)\label{eq:oracle-agg}
\end{equation}
with high probability for some arbitrarily small constant $\delta\in(0,1)$ and some optimal rate function $\Delta_{n,p}(\Theta)$ determined by the class $\Theta$. The right hand side of (\ref{eq:oracle-agg}) is also called the index of resolvability of $f^*$ \citep{barron1991minimum,yang1999model}. Various types of aggregation problems include linear, convex, model selection aggregation, etc., which are determined by the choice of the class $\Theta$. In this section, we provide a single Bayes solution to various types of aggregation problems simultaneously and establish the oracle inequality (\ref{eq:oracle-agg}) under the posterior distribution.

Since the vector $f_{\beta}=(f_{\beta}(x_1),...,f_{\beta}(x_n))$ can be represented as $X\beta$ with $X_{ij}=f_j(x_i)$ for all $(i,j)\in[n]\times[p]$, the aggregation problem can be recast as a linear regression problem. Define $r=\text{rank}(X)$. Without loss of generality, we assume the first $r$ columns of $X$ span the column space of $X$, i.e., $\text{span}(\{X_{*j}\}_{j\in[r]})=\text{span}(\{X_{*j}\}_{j\in[p]})$. We are going to use a modified version of the prior distribution defined in Section \ref{sec:slr}:
\begin{enumerate}
\item Sample $s\sim \pi$ from $[r]$, where $\pi(s)=\mathcal{N}\frac{\Gamma(s)}{\Gamma(s/2)}\exp\left(-Ds\log\frac{ep}{s}\right)$ for $s<r$ and $\pi(r)=\mathcal{N}\frac{\Gamma(r)}{\Gamma(r/2)}\exp(-Dr)$ with some normalizing constant $\mathcal{N}$;
\item Conditioning on $s$, sample $S$ uniformly from $\bar{\Z}_s=\{S\subset[p]: |S|=s, \det(X_{*S}^TX_{*S})>0\}$ if $s<r$ and set $S=[r]$ if $s=r$;
\item Conditioning on $(s,S)$, sample $\beta_S\sim f_{s,S,\lambda}$ with $f_{s,S,\lambda}(\beta_S)\propto e^{-\lambda\norm{X_{*S}\beta_S}}$ and set $\beta_{S^c}=0$.
\end{enumerate}
The prior $\Pi$ is similar to the exponential weights used for sparsity pattern aggregation by \cite{rigollet2011exponential,rigollet2012sparse}. Compared with the prior in Section \ref{sec:slr}, it has a modified weight on the model $S=[r]$, which captures the intrinsic dimension of the matrix $X$. Assuming the data generating process (\ref{eq:Y-agg}), we have the following result implied by Theorem \ref{thm:main}.

\begin{corollary}\label{cor:general-agg}
For any $\beta^*$ with support $S^*\in\bar{\Z}_{s^*}$ and sparsity $s^*=|S^*|\leq r$, any $f^*$, any constants $\lambda,\rho>0$ and any sufficiently small constant $\delta\in(0,1)$, there exists some constant $D_{\lambda,\delta,\rho}$ only depending on $\lambda,\delta,\rho$ such that
\begin{eqnarray*}
&& \mathbb{E}_{f^*}\Pi\left(\norm{f_{\beta}-f^*}_n^2> (1+\delta)\norm{f_{\beta^*}-f^*}_n^2+M\left(\frac{r}{n}\wedge\frac{s^*\log(ep/s^*)}{n}\right)\Big|Y\right) \\
&\leq& \exp\left(-C'\left(n\norm{f_{\beta}-f^*}_n^2+r\wedge s^*\log\frac{ep}{s^*}\right)\right)
\end{eqnarray*}
for any constant $D>D_{\lambda,\delta,\rho}$ with some constants $M,C'$ only depending on $\lambda,\delta,\rho,D$.
\end{corollary}
Since $\text{rank}(X)=r$, it is sufficient to establish the posterior oracle inequality for all $\beta^*$ with sparsity $s^*\leq r$.
Due to the modified prior weight on the model $S=[r]$, Corollary \ref{cor:general-agg} has a better convergence rate than Corollary \ref{cor:slr-pred}. The corresponding frequentist results \citep{rigollet2011exponential,rigollet2012sparse} have leading constant $1$ instead of the $(1+\delta)$ in Corollary \ref{cor:general-agg}. Since our prior and postetior have a subset selection step, the result in \cite{rigollet2012sparse} suggests that the extra constant $\delta$ may be necessary.

Let us specialize Corollary \ref{cor:general-agg} to various types of aggregation problems. Following the notation in \cite{tsybakov2014aggregation}, define the simplex $\Lambda^p=\{\beta\in\mathbb{R}^p:\sum_j\beta_j=1,\beta_j\geq 0\}$ and the $\ell_0$ ball $\mathcal{B}_0(s^*)=\{\beta\in\mathbb{R}^p:|\supp(\beta)|\leq s^*\}$. Then, we consider model selection aggregation $\Theta_{\textsf{(MS)}}=\mathcal{B}_0(1)\cap\Lambda^p$, convex aggregation $\Theta_{\textsf{(C)}}=\Lambda^p$, linear aggregation $\Theta_{\textsf{(L)}}=\mathbb{R}^p$, sparse aggregation $\Theta_{\textsf{($\textsf{L}_s$)}}=\mathcal{B}_0(s^*)$ and sparse convex aggregation $\Theta_{\textsf{($\textsf{C}_s$)}}=\mathcal{B}_0(s^*)\cap\Lambda^p$. For these aggregation problems, define the rate function
$$\Delta_{n,p}(\Theta)=\begin{cases}
\frac{\log p}{n}, & \Theta=\Theta_{\textsf{(MS)}};\\
\sqrt{\frac{1}{n}\log\left(1+\frac{p}{\sqrt{n}}\right)} & \Theta=\Theta_{\textsf{(C)}};\\
\frac{r}{n}, & \Theta=\Theta_{\textsf{(L)}};\\
\frac{s*\log\frac{ep}{s^*}}{n}, & \Theta=\Theta_{\textsf{($\textsf{L}_s$)}};\\
\sqrt{\frac{1}{n}\log\left(1+\frac{p}{\sqrt{n}}\right)}\wedge\frac{s*\log\frac{ep}{s^*}}{n}, & \Theta=\Theta_{\textsf{($\textsf{C}_s$)}}.
\end{cases}$$
\begin{corollary}\label{cor:universal}
Assume $\max_{j\in[p]}\norm{f_j}_n\leq 1$.
For any $f^*$, any $\Theta\in\left\{\Theta_{\textsf{(MS)}},\Theta_{\textsf{(C)}},\Theta_{\textsf{(L)}},\Theta_{\textsf{($\textsf{L}_s$)}},\Theta_{\textsf{($\textsf{C}_s$)}}\right\}$, any constants $\lambda,\rho>0$ and any sufficiently small constant $\delta\in(0,1)$, there exists some constant $D_{\lambda,\delta,\rho}$ only depending on $\lambda,\delta,\rho$ such that
\begin{eqnarray*}
&& \mathbb{E}_{f^*}\Pi\left(\norm{f_{\beta}-f^*}_n^2> (1+\delta)\inf_{\beta\in\Theta}\norm{f_{\beta}-f^*}_n^2+M\left(\Delta_{n,p}(\Theta)\wedge\frac{r}{n}\right)\Big|Y\right) \\
&\leq& \exp\left(-C'n\left(\inf_{\beta\in\Theta}\norm{f_{\beta}-f^*}_n^2+\Delta_{n,p}(\Theta)\wedge\frac{r}{n}\right)\right)
\end{eqnarray*}
for any constant $D>D_{\lambda,\delta,\rho}$ with some constants $M,C'$ only depending on $\lambda,\delta,\rho,D$.
\end{corollary}
Corollary \ref{cor:universal} provides a universal aggregation result with a single posterior distribution. The rate is minimax optimal according to \cite{rigollet2011exponential,wang2011adaptive}. Bayes aggregation was recently studied by \cite{yang2014minimax} under the model misspecification framework \citep{kleijn2006misspecification}. Corollary \ref{cor:universal} is a stronger result of posterior oracle inequality under weaker assumptions compared with that of \cite{yang2014minimax}. Other types of aggregation results such as $\ell_q$ aggregation can also be derived directly from Corollary \ref{cor:general-agg}.

\subsection{Linear regression under weak $\ell_q$ ball}\label{sec:weak-ball}

Section \ref{sec:slr} studied high dimensional linear regression under exact sparsity. In this section, we assume that regression coefficients are approximately sparse. Theorem \ref{thm:main} allows us to derive optimal posterior rates of contraction via a bias variance tradeoff argument. Assume the data is generated by $Y=X\beta^*+W\in\mathbb{R}^p$ with some design $X\in\mathbb{R}^{n\times p}$ and some sub-Gaussian noise vector $W$ satisfying (\ref{eq:subG}). We assume $\beta^*$ is approximately sparse,
$$\beta^*\in\mathcal{B}_q(k)=\left\{\beta\in\mathbb{R}^p: \max_{j\in[p]}j|\beta|_{(j)}^q\leq k\right\}$$
with some $q\in(0,1]$, where we order the absolute values of the entries of $\beta$ by $|\beta|_{(1)}\geq |\beta|_{(2)}\geq ... \geq |\beta|_{(p)}$, i.e., $\beta^*$ is in a weak $\ell_q$ ball with radius at most $k$. For $q=0$,
$$\mathcal{B}_0(k)=\left\{\beta\in\mathbb{R}^p:\sum_{j=1}^p\mathbb{I}\{\beta_j\neq 0\}\leq k\right\},$$
which is reduced to the case of exact sparsity.
To facilitate the presentation, we define the effective sparsity by $s^*=\ceil{x^*}$, where
$$x^*=\max\left\{0\leq x\leq p: x\leq k\left(\frac{n}{\log(ep/x)}\right)^{q/2}\right\}.$$
The effective sparsity $s^*$ is a function of $q,k,p,n$.
In the exact sparse case where $q=0$, we have $s^*=k$. For the prior distribution specified in Section \ref{sec:slr}, we have the following result.
\begin{corollary}\label{cor:slrq-pred}
Assume $\max_{j\in[p]}n^{-1/2}\norm{X_{*j}}\leq L$ for some constant $L>0$. For any $q\in[0,1]$, $k$ and $s^*$ specified above and any constants $\lambda,\rho>0$, there exists some constant $D_{\lambda,\rho}>0$ only depending on on $\lambda,\rho$ such that
$$\sup_{\beta^*\in\mathcal{B}_q(k)}\mathbb{E}_{X\beta^*}\Pi\left(\norm{X\beta-X\beta^*}^2>Ms^*\log\frac{ep}{s^*}\Big|Y\right)\leq\exp\left(-C's^*\log\frac{ep}{s^*}\right)$$
for any constant $D>D_{\lambda,\rho}$ with some constants $M,C'$ only depending on $\lambda,\rho,D,L$.
\end{corollary}
With $s^*$ being the effective sparsity, the posterior rate of contraction has the same form as that of Corollary \ref{cor:slr-pred}. The rate is known to be minimax optimal \citep{donoho1994minimax,raskutti2011minimax}. In the special case when $k\leq p^{1-\eta}\left(\frac{\log p}{n}\right)^{q/2}$ for some constant $\eta\in(0,1)$, the rate has an explicit formula in terms of $k$, which is $\frac{s^*\log(ep/s^*)}{n}\asymp k\left(\frac{\log p}{n}\right)^{1-q/2}$. When $X$ is an identity matrix, Corollary \ref{cor:slrq-pred} reduces to the results for sparse Gaussian sequence model in \cite{castillo2012needles}. Besides the prediction error, estimation error under approximate sparsity can be derived in the same way as Corollary \ref{cor:slr-esti}. Finally, we remark that in practice, the assumption $\max_{j\in[p]}n^{-1/2}\norm{X_{*j}}\leq L$ can be met by column normalization of the design matrix.

\subsection{Wavelet estimation in Besov space}\label{sec:wavelet}

In this section, we apply the general prior distribution in Section \ref{sec:prior} to establish optimal Bayes wavelet estimation under Besov space. Assume the data is generated as
\begin{equation}
Y_{jk}=\theta^*_{jk}+\frac{1}{\sqrt{n}}W_{jk},\quad k=1,...,2^j;\quad j=0,1,2,...,\label{eq:sequence}
\end{equation}
where $\{W_{jk}\}$ are i.i.d. $N(0,1)$ variables. It is well known that the sequence model is equivalent to Gaussian white noise model \cite{johnstone2011gaussian}, and it is closely related to nonparametric regression and density estimation \citep{brown1996asymptotic,nussbaum1996asymptotic}. Under a wavelet basis, $\{\theta_{jk}\}$ are understood as wavelet coefficients. We assume the true signal $\theta^*=\{\theta_{jk}^*\}$ belongs to the Besov ball defined by
\begin{equation}
\Theta_{p,q}^{\alpha}(L) = \left\{\theta: \sum_j2^{ajq}\norm{\theta_{j*}}_p^q\leq L^q\right\}\label{eq:besov}
\end{equation}
for some $p,q,\alpha,L>0$ and $a=\alpha+\frac{1}{2}-\frac{1}{p}$. The Besov ball (\ref{eq:besov}) naturally induces a multi-resolution structure of the signal. This inspires us to use a sparse prior distribution independently at each resolution level. That is, we consider a prior distribution $\Pi$ on $\theta$ satisfying
$$\Pi(d\theta)=\prod_j\Pi_j(d\theta_{j*}).$$
The prior distribution $\Pi_j$ on the $j$th level for $j< \log_2n$ is specified as follows:
\begin{enumerate}
\item Sample $s_j\sim\pi$ from $[2^j]$, where $\pi(s_j)\propto \frac{\Gamma(s_j)}{\Gamma(s_j/2)}\exp\left(-Ds_j\log\frac{e2^j}{s_j}\right)$;
\item Conditioning on $s_j$, sample $S_j$ uniformly from $\{S_j\subset[2^j]: |S_j|=s_j\}$;
\item Conditioning on $(s_j,S_j)$, sample $\theta_{jS_j}\sim f_{s_j,S_j,\lambda}$ with $f_{s_j,S_j,\lambda}(\theta_{jS_j})\propto e^{-\lambda \sqrt{n}\norm{\theta_{jS_j}}}$ and set $\theta_{jS_j^c}=0$.
\end{enumerate}
For $j\geq \log_2n$, let $\Pi_j(\theta_{j*}=0)=1$. Using Theorem \ref{thm:main} at each resolution level, we are able to establish the posterior contraction rate in the following corollary.
\begin{corollary}\label{cor:wavelet}
For any costants $p,q,\alpha$ satisfying $0< p,q\leq \infty$, $L>0$ and $\alpha\geq\frac{1}{p}$ and any constant $\lambda>0$, there exists some constant $D_{\lambda}$ only depending on $\lambda$ such that
$$\sup_{\theta^*\in\Theta_{p,q}^{\alpha}(L)}\mathbb{E}_{\theta^*}\Pi\left(\norm{\theta-\theta^*}^2>Mn^{-\frac{2\alpha}{2\alpha+1}}\Big|Y\right)\leq\exp\left(-C'n^{\frac{1}{2\alpha+1}}/\log n\right)$$
for any $D>D_{\lambda}$ with some constants $M,C'$ only depending on $\lambda,D,\alpha,p,L$.
\end{corollary}
The result of Corollary \ref{cor:wavelet} can be regarded as a Bayes version of Theorem 12.1 of \cite{johnstone2011gaussian} under the same condition. The rate $n^{-\frac{2\alpha}{2\alpha+1}}$ is minimax optimal over the class $\Theta_{p,q}^{\alpha}(L)$. Posterior contraction for (\ref{eq:sequence}) over the class $\Theta_{p,q}^{\alpha}(L)$ has been investigated by \cite{van2008rates,rivoirard2012posterior,gao2013adaptive,hoffmann2013adaptive} only for a restricted configuration of $(p,q,\alpha)$. In comparison, Corollary \ref{cor:wavelet} obtains adaptive optimal posterior contraction rates to all possible combinations of $(p,q,\alpha)$ considered in the frequentist literature \citep{johnstone2011gaussian}.

When $p=q=2$, the class $\Theta_{p,q}^{\alpha}(L)$ is equivalent to a Sobolev ball. It is worth noting that in this case the prior distribution can be greatly simplified. Let us recast (\ref{eq:sequence}) into the sequence model with single index. That is, consider data generated by
$$Y_j=\theta^*_j+\frac{1}{\sqrt{n}}W_j,\quad j=1,2,3,...,$$
with $\{W_j\}$ being i.i.d. $N(0,1)$ variables. Assume the true signal $\theta^*=\{\theta_j^*\}$ belongs to the Sobolev ball defined by
$$\mathcal{S}_{\alpha}(L)=\left\{\theta: \sum_ja_j^2\theta_j^2\leq L^2\right\},$$
for some sequence $a_j\asymp j^{\alpha}$.
We use the following version of the general prior $\Pi$ in Section \ref{sec:prior}.
\begin{enumerate}
\item Sample $k\sim \pi$ from $[n]$, where $\pi(k)\propto\frac{\Gamma(k)}{\Gamma(k/2)}\exp\left(-Dk\right)$;
\item Conditioning on $k$, sample $\theta_{[k]}=(\theta_1,...,\theta_k)\sim f_{k,\lambda}$ with $f_{k,\lambda}(\theta_{[k]})\propto e^{-\lambda\sqrt{n}\norm{\theta_{[k]}}}$ and set $\theta_j=0$ for all $j>k$.
\end{enumerate}
Note that the prior distribution has a missing step compared with the general prior in Section \ref{sec:prior}, since $\Z_k=\{[k]\}$ is a singleton set and we do not need to perform a further model selection. Specializing Theorem \ref{thm:main} to this case, we obtain the following result.
\begin{corollary}\label{cor:sobolev}
For any constants $\alpha,L>0$ and any constant $\lambda>0$, there exists some constant $D_{\lambda}$ only depending on $\lambda$ such that
$$\sup_{\theta^*\in\mathcal{S}_{\alpha}(L)}\mathbb{E}_{\theta^*}\Pi\left(\norm{\theta-\theta^*}^2>Mn^{-\frac{2\alpha}{2\alpha+1}}\Big|Y\right)\leq\exp\left(-C'n^{\frac{1}{2\alpha+1}}\right)$$
for any $D>D_{\lambda}$ with some constants $M,C'$ only depending on $\lambda,D,\alpha,L$.
\end{corollary}
Thus, we have obtained rate-optimal adaptive posterior contraction over the Sobolev ball through a very simple prior distribution. 

To close this section, we remark that the prior distributions used in this section depend on $n$. This is a consequence of writing the Gaussian sequence model in the form of structured linear models. For adaptive priors that do not depend on $n$ but still achieve optimal posterior contraction rates, we refer the readers to \cite{gao2013adaptive}.

\section{More results on sparse linear regression}\label{sec:more}

In this section, we provide some further results on posterior contraction rates for linear regression under the $\ell_{\infty}$ norm $\norm{\cdot}_{\infty}$. First, let us consider the sparse linear regression setting $Y=X\beta+W$ in Section \ref{sec:slr}. Convergence under the $\ell_{\infty}$ norm requires stronger assumptions than convergence under the $\ell_2$ norm. Following \cite{donoho2006stable,lounici2008sup}, we assume the mutual coherence condition:
\begin{equation}
n^{-1}X_{*j}^TX_{*j}=1\text{ for all }j\in[p]\quad\text{and}\quad \max_{j\neq k}n^{-1}X_{*j}^TX_{*k}\leq \tau.\label{eq:coherence}
\end{equation}
Assuming that data is generated by $Y=X\beta^*+W$ for some regression coefficient $\beta^*$ with sparsity $s^*$ and some noise vector $W$ satisfying (\ref{eq:subG}), the posterior contraction under the $\ell_{\infty}$ norm for the prior distribution specified in Section \ref{sec:slr} is given in the following theorem.
\begin{thm}\label{thm:sup}
For any $\tau>0$ and any $\beta^*$ with sparsity $s^*$ satisfying $\tau s^*\leq 1/9$ and any constants $\lambda,\rho>0$, there exists some constant $D_{\lambda,\rho}>0$ only depending on $\lambda,\rho$ such that
$$\mathbb{E}_{X\beta^*}\Pi\left(\norm{\beta-\beta^*}_{\infty}>M\sqrt{\frac{\log p}{n}}\Big|Y\right)\leq p^{-C'}$$
for any constant $D>D_{\lambda,\rho}$ with some constants $M,C'$ only depending on $\lambda,\rho,D$.
\end{thm}
The result of convergence under the $\ell_{\infty}$ norm is obtained under the assumption $\tau s^*\leq 1/9$. Such assumption was also made in \cite{donoho2006stable,bunea2008consistent,lounici2008sup,castillo2014}. It implies the restricted eigenvalue $\kappa_2$ defined in (\ref{eq:RE21}) to be bounded away from $0$ \citep{zhang2008sparsity}.
The convergence rate $\sqrt{\frac{\log p}{n}}$ is optimal under the $\ell_{\infty}$ norm. Moreover, with a standard minimal signal strength assumption, Theorem \ref{thm:sup} immediately implies model selection consistency under the posterior distribution.

While the optimal convergence result for $\ell_{\infty}$ norm is well known in the frequentist literature for sparse linear regression, an analogous result for regression with group sparsity is not stated in literature. We provide a Bayes solution to this problem. For simplicity of presentation, we consider the case of identity design $Y=B+W\in\mathbb{R}^{p\times m}$. The result for the case of a more general design can be derived in a similar way. For any subset $T\subset[p]\times[m]$, let $r(T)=\{i\in[p]:(\{i\}\times[m])\cap T\neq \varnothing\}$ denote the the rows selected by the set $T$. The prior $\Pi$ we use is defined through the following sampling procedure:
\begin{enumerate}
\item Sample $T\sim \pi$ in $\{T: T\subset[p]\times[m]\}$ with
\begin{equation}
\pi(T)\propto\frac{\Gamma(|T|)}{\Gamma(|T|/2)}\exp\left(-D\left(m|r(T)|+|r(T)|\log\frac{ep}{|r(T)|}+|T|\log\frac{em|r(T)|}{|T|}\right)\right);\label{eq:2l}
\end{equation}
\item Conditioning on $T$, sample $B_T\sim f_{T,\lambda}$ with $f_{T,\lambda}(B_T)\propto e^{-\lambda\sqrt{\sum_{(i,j)\in T}B_{ij}^2}}$ and set $B_{T^c}=0$.
\end{enumerate}
Compared with the prior distribution specified in Section \ref{sec:group}, the model selection step for the above prior has a two-level structure. Apart from the correction factor $\frac{\Gamma(|T|)}{\Gamma(|T|/2)}$, the probability mass (\ref{eq:2l}) can be viewed as the product of $e^{-D|S|\left(m+\log\frac{ep}{|S|}\right)}$ and $e^{-D|T|\log\frac{em|S|}{|T|}}$ with $S=r(T)$ denoting the row support. Therefore, (\ref{eq:2l}) can be understood as first picking a row support $S$, and then further selecting a finer support from $S\times[m]$. In comparison, the prior specified in Section \ref{sec:group} does not have the second step. While it only produces $B$ with support in the form of $S\times[m]$ for some $S$, (\ref{eq:2l}) can give an arbitrary support $T$, which is critical to obtain optimal convergence rate under the $\ell_{\infty}$ loss. Assume that the data is generated from $Y=B^*+W$ for some $B^*$ with row support $S^*$ and noise matrix $W$ satisfying (\ref{eq:subG}). The posterior contraction rate is given in the following theorem.

\begin{thm}\label{thm:sup-g}
For any $B^*$ with row support $S^*$ and sparsity $s^*=|S^*|$, any arbitrarily small constant $\delta>0$ and any constants $\lambda,\rho>0$, there exists some constant $D_{\lambda,\delta,\rho}>0$ only depending on $\lambda,\delta,\rho$ such that
\begin{equation}
\mathbb{E}_{B^*}\Pi\left(|r(T)|>(1+\delta)s^*\Big|Y\right)\leq \exp\left(-C's^*\left(m+\log\frac{ep}{s^*}\right)\right),\label{eq:sup-g-dim}
\end{equation}
\begin{equation}
\mathbb{E}_{B^*}\Pi\left(\fnorm{B-B^*}^2>Ms^*\left(m+\log\frac{ep}{s^*}\right)\Big|Y\right)\leq \exp\left(-C''s^*\left(m+\log\frac{ep}{s^*}\right)\right)\label{eq:sup-g-2}
\end{equation}
and
\begin{equation}
\mathbb{E}_{B^*}\Pi\left(\norm{B-B^*}_{\infty}>M\sqrt{\log(p+m)}\Big|Y\right)\leq (pm)^{-C'''}\label{eq:sup-g-inf}
\end{equation}
for any constant $D>D_{\lambda,\delta,\rho}$ with some constants $M,C',C'',C'''$ only depending on $\lambda,\delta,\rho,D$.
\end{thm}
To the best our knowledge, this is the first procedure that achieves the optimal rates simultaneously for both $\ell_2$ and $\ell_{\infty}$ losses in a group sparse signal recovery problem. The $e^{-D|S|\left(m+\log\frac{ep}{|S|}\right)}$ part in (\ref{eq:2l}) preserves the group sparse structure and results in the optimal $\ell_2$ result (\ref{eq:sup-g-2}). The $e^{-D|T|\log\frac{em|S|}{|T|}}$ part in (\ref{eq:2l}) does a further model selection in a finer resolution, thus giving optimal rate for each coordinate in (\ref{eq:sup-g-inf}). The subtlety of the simultaneous adaptation under both global and local loss functions is not reflected in an ordinary sparsity setting. When $m=1$, group sparsity reduces to ordinary sparsity and the two-level model selection prior $\Pi$ is equivalent to the prior in Section \ref{sec:slr}, so that a one-level model selection would be sufficient for the task.

\section{Proof of Theorem \ref{thm:main}}\label{sec:proof}

Let us first introduce some notation and give the outline of the proof. Define the following two sets,
$$\mathcal{A}(t)=\left\{\epsilon(\Z_{\tau})>(1+\delta_1)\epsilon(\Z_{\tau^*})+\delta_1\norm{\X_{\Z^*}(Q^*)-\theta^*}^2+ct\right\},$$
$$U(t)=\left\{\norm{\X_Z(Q)-\theta^*}^2>(1+\delta_2)\norm{\X_{Z^*}(Q^*)-\theta^*}^2+M(\epsilon(\Z_{\tau^*})+t)\right\}.$$
We will specify the numbers $\delta_1,\delta_2,c$ later. The goal of the proof is to derive bounds for both $\mathbb{E}\Pi(\tau\in\mathcal{A}(t)|Y)$ and $\mathbb{E}\Pi(\X_Z(Q)\in U(t)|Y)$ with any $t\geq 0$. Then, the conclusions (\ref{eq:main1}) and (\ref{eq:main2}) are deduced by setting $t=0$. The conclusion (\ref{eq:main33}) is then obtained by integrating out the tail bound of $\mathbb{E}\Pi(\X_Z(Q)\in U(t)|Y)$ over $t\geq 0$.

Using the fact that
$$\frac{e^{-\frac{1}{2}\norm{Y-\X_Z(Q)}^2}}{e^{-\frac{1}{2}\norm{Y-\X_{Z^*}(Q^*)}^2}}=e^{-\frac{1}{2}\norm{\X_Z(Q)-\X_{Z^*}(Q^*)}^2+\iprod{Y-\X_{Z^*}(Q^*)}{\X_Z(Q)-\X_{Z^*}(Q^*)}},$$
we can rewrite the posterior distribution as
\begin{equation}
\Pi\left(\X_Z(Q)\in U(t)|Y\right) = \frac{\sum_{\tau\in\T}\exp(-D\epsilon(\Z_{\tau}))\frac{1}{|\bar{\Z}_{\tau}|}\sum_{Z\in\bar{\Z}_{\tau}}R(Z,U(t))}{\sum_{\tau\in\T}\exp(-D\epsilon(\Z_{\tau}))\frac{1}{|\bar{\Z}_{\tau}|}\sum_{Z\in\bar{\Z}_{\tau}}R(Z)}, \label{eq:form-pred}
\end{equation}
where $R(Z,U(t))$ is defined by
\begin{eqnarray*}
&& \sqrt{\det(\X_Z^T\X_Z)}\left(\frac{\lambda}{\sqrt{\pi}}\right)^{\ell(\Z_{\tau})} \\
&& \times \int_{\X_Z(Q)\in U(t)} e^{-\frac{1}{2}\norm{\X_Z(Q)-\X_{Z^*}(Q^*)}^2+\iprod{Y-\X_{Z^*}(Q^*)}{\X_Z(Q)-\X_{Z^*}(Q^*)}-\lambda\norm{\X_Z(Q)}} dQ,
\end{eqnarray*}
and $R(Z)=R(Z,\mathbb{R}^N)$. Moreover, for a class of structure indexes $\mathcal{A}(t)\subset\T$, its posterior distribution can be written as
\begin{equation}
\Pi\left(\tau\in\mathcal{A}(t)|Y\right) = \frac{\sum_{\tau\in\mathcal{A}(t)}\exp(-D\epsilon(\Z_{\tau}))\frac{1}{|\bar{\Z}_{\tau}|}\sum_{Z\in\bar{\Z}_{\tau}}R(Z)}{\sum_{\tau\in\T}\exp(-D\epsilon(\Z_{\tau}))\frac{1}{|\bar{\Z}_{\tau}|}\sum_{Z\in\bar{\Z}_{\tau}}R(Z)}. \label{eq:form-dim}
\end{equation}
We are going to work with the formulas (\ref{eq:form-dim}) and (\ref{eq:form-pred}) to prove (\ref{eq:main1}) and (\ref{eq:main2}), respectively. The main strategy is to lower bound $R(Z^*)$ in the denominator and upper bound $R(Z)$ or $R(Z,U(t))$ in the numerator given some events holding with high probability. For each $Z\in\bar{\Z}_{\tau}$ and $t\geq 0$, consider the following events
\begin{eqnarray*}
E_Z(t) &=& \left\{\left|\iprod{W}{\X_Z(Q)-\X_{Z^*}(Q^*)}\right|\leq \sqrt{\epsilon^*(\Z_{\tau})+t}\norm{\X_Z(Q)-\X_{Z^*}(Q^*)}\text{ for all }Q\in\mathbb{R}^{\ell(\Z_{\tau})}\right\}, \\
F_Z(t) &=& \left\{\left|\iprod{W}{\X_Z(Q)-\X_{Z^*}(Q^*)}\right|\leq \sqrt{\epsilon^*(\Z_{\tau^*})+t}\norm{\X_Z(Q)-\X_{Z^*}(Q^*)}\text{ for all }Q\in\mathbb{R}^{\ell(\Z_{\tau})}\right\},
\end{eqnarray*}
where $\epsilon^*(\Z_{\tau})=C_1\epsilon(\Z_{\tau})+C_2\norm{\X_{Z^*}(Q^*)-\theta^*}^2$ and $\epsilon^*(\Z_{\tau^*})=C_1\epsilon(\Z_{\tau^*})+C_2\norm{\X_{Z^*}(Q^*)-\theta^*}^2$ for some constants $C_1,C_2$ to be specified later. The next lemma shows that both events hold with high probability.
\begin{lemma}\label{lem:noise}
For any constants $C_1>1$, $C_2>0$ and $t\geq 0$, the conditions (\ref{eq:larger}) and (\ref{eq:subG}) imply
\begin{eqnarray*}
\mathbb{P}(E_Z(t)^c) &\leq& 2\exp\left(-(\rho C_1/16-5)\epsilon(\Z_{\tau})-\rho C_2\norm{\X_{Z^*}(Q^*)-\theta^*}^2/16-\rho t/16\right), \\
\mathbb{P}(F_Z(t)^c) &\leq& 2\exp\left(5\ell(\Z_{\tau})-\rho C_1\epsilon(\Z_{\tau^*})/16-\rho C_2\norm{\X_{Z^*}(Q^*)-\theta^*}^2/16-\rho t/16\right).
\end{eqnarray*}
\end{lemma}
We need a lemma to characterize the growing rate of $\epsilon(\Z_{\tau})$.
\begin{lemma}\label{lem:growth}
For any $\beta\geq 2$ and $\alpha\geq 1$, the condition (\ref{eq:capacity}) implies
\begin{eqnarray*}
\sum_{\{\tau\in\T:\epsilon(\Z_{\tau})\leq\alpha\}}\exp\left(\beta\epsilon(\Z_{\tau})\right) &\leq& 4\ceil{\alpha}\exp(\beta\ceil{\alpha}); \\
\sum_{\{\tau\in\T:\epsilon(\Z_{\tau})>\alpha\}}\exp\left(-\beta\epsilon(\Z_{\tau})\right) &\leq& 4\alpha\exp\left(-\beta\floor{\alpha}\right); \\
\sum_{\{\tau\in\T:\epsilon(\Z_{\tau})\leq\alpha\}} \exp\left(-\beta\epsilon(\Z_{\tau})\right) &\leq& 6.
\end{eqnarray*}
\end{lemma}
The proofs of Lemma \ref{lem:noise} and Lemma \ref{lem:growth} are given in Section \ref{sec:pftech} of the supplement \citep{gao16}.

\paragraph{Lower bounding $R(Z^*)$.}
We first introduce some extra notation. For the matrix $\X_{Z^*}\in\mathbb{R}^{N\times\ell(\Z_{\tau^*})}$, its singular value decomposition is $\X_{Z^*}=\mathcal{U}\Lambda\mathcal{V}^T$, with $\mathcal{U}\in\mathbb{R}^{N\times \ell(\Z_{\tau^*})}$ and $\mathcal{V}\in\mathbb{R}^{\ell(\Z_{\tau^*})\times \ell(\Z_{\tau^*})}$ being orthonormal matrices, and $\Lambda$ is an $\ell(\Z_{\tau^*})\times \ell(\Z_{\tau^*})$ diagonal matrix with positive entries on the diagonal.

For $Z^*\in\bar{\Z}_{\tau^*}$ with any $\tau^*\in\T$, we lower bound $R(Z^*)$ by
\begin{eqnarray}
\nonumber && \left(\frac{\sqrt{\pi}}{\lambda}\right)^{\ell(\Z_{\tau^*})}R(Z^*) \\
\nonumber  &=& \sqrt{\det(\X_{Z^*}^T\X_{Z^*})}\int e^{-\frac{1}{2}\norm{\X_{Z^*}(Q)-\X_{Z^*}(Q^*)}^2+\iprod{Y-\X_{Z^*}(Q^*)}{\X_{Z^*}(Q)-\X_{Z^*}(Q^*)}-\lambda\norm{\X_{Z^*}(Q)}} dQ \\
\label{eq:denom1} &=& \sqrt{\det(\X_{Z^*}^T\X_{Z^*})}\int e^{-\frac{1}{2}\norm{\X_{Z^*}(Q)}^2+\iprod{Y-\X_{Z^*}(Q^*)}{\X_{Z^*}(Q)}-\lambda\norm{\X_{Z^*}(Q)+\X_{Z^*}(Q^*)}} dQ \\
\label{eq:denom2} &\geq& e^{-\lambda\norm{\X_{Z^*}(Q^*)}}\sqrt{\det(\X_{Z^*}^T\X_{Z^*})}\int e^{-\frac{1}{2}\norm{\X_{Z^*}(Q)}^2+\iprod{Y-\X_{Z^*}(Q^*)}{\X_{Z^*}(Q)}-\lambda\norm{\X_{Z^*}(Q)}} dQ \\
\label{eq:denom2.5} &=& e^{-\lambda\norm{\X_{Z^*}(Q^*)}}\sqrt{\det(\mathcal{V}\Lambda^2\mathcal{V}^T)}\int e^{-\frac{1}{2}\norm{\Lambda \mathcal{V}^TQ}^2+\iprod{\mathcal{U}^T(Y-\X_{Z^*}(Q^*))}{\Lambda\mathcal{V}^TQ}-\lambda\norm{\Lambda \mathcal{V}^TQ}}dQ  \\
\label{eq:denom3} &=& e^{-\lambda\norm{\X_{Z^*}(Q^*)}}\int e^{-\frac{1}{2}\norm{b}^2+\iprod{\mathcal{U}^T(Y-\X_{Z^*}(Q^*))}{b}-\lambda\norm{b}} db \\
\label{eq:denom4} &\geq& e^{-\lambda\norm{\X_{Z^*}(Q^*)}}\int e^{-\frac{1}{2}\norm{b}^2-\lambda\norm{b}} db \\
\nonumber && \times \exp\left(\int\iprod{\mathcal{U}^T(Y-\X_{Z^*}(Q^*))}{b}\frac{e^{-\frac{1}{2}\norm{b}^2-\lambda\norm{b}}}{\int e^{-\frac{1}{2}\norm{b}^2-\lambda\norm{b}} db}db\right) \\
\label{eq:denom5} &=& e^{-\lambda\norm{\X_{Z^*}(Q^*)}}\int e^{-\frac{1}{2}\norm{b}^2-\lambda\norm{b}} db.
\end{eqnarray}
The equalities (\ref{eq:denom1}) and (\ref{eq:denom3}) are due to changes of variables and the linearity (\ref{eq:linearity}), and we use the orthonormal property of $\mathcal{U}$ to get $\det(\X_{Z^*}^T\X_{Z^*})=\det(\mathcal{V}\Lambda^2\mathcal{V}^T)$ and $\norm{\X_{Z^*}(Q)}=\norm{\Lambda \mathcal{V}^TQ}$. We use triangle inequality and Jensen's inequality to derive (\ref{eq:denom2}) and (\ref{eq:denom4}), respectively. The last equality (\ref{eq:denom5}) uses the fact that the distribution $\frac{e^{-\frac{1}{2}\norm{b}^2-\lambda\norm{b}}}{\int e^{-\frac{1}{2}\norm{b}^2-\lambda\norm{b}} db}$ is spherically symmetric so that its mean is zero. Let us continue to lower bound the integral $\int e^{-\frac{1}{2}\norm{b}^2-\lambda\norm{b}} db$ by
\begin{eqnarray*}
\int e^{-\frac{1}{2}\norm{b}^2-\lambda\norm{b}} db &=& \frac{2\pi^{\ell(\Z_{\tau^*})/2}}{\Gamma(\ell(\Z_{\tau^*})/2)}\int_0^{\infty}r^{\ell(\Z_{\tau^*})-1}e^{-\frac{1}{2}r^2-\lambda r}dr \\
&\geq& \frac{2\pi^{\ell(\Z_{\tau^*})/2}}{\Gamma(\ell(\Z_{\tau^*})/2)}e^{-\frac{1}{2}\ell(\Z_{\tau^*})-\lambda\sqrt{\ell(\Z_{\tau^*})}}\int_0^{\sqrt{\ell(\Z_{\tau^*})}}r^{\ell(\Z_{\tau^*})-1}dr \\
&=& \frac{2\pi^{\ell(\Z_{\tau^*})/2}}{\ell(\Z_{\tau^*})}\frac{[\ell(\Z_{\tau^*})]^{\ell(\Z_{\tau^*})/2}}{\Gamma(\ell(\Z_{\tau^*})/2)}e^{-\frac{1}{2}\ell(\Z_{\tau^*})-\lambda\sqrt{\ell(\Z_{\tau^*})}} \\
&\geq& \frac{2(2\pi)^{\ell(\Z_{\tau^*})/2}}{\ell(\Z_{\tau^*})}e^{-\frac{1}{2}\ell(\Z_{\tau^*})-\lambda\sqrt{\ell(\Z_{\tau^*})}}.
\end{eqnarray*}
Combining the above lower bound with (\ref{eq:denom5}), we reach the conclusion
\begin{eqnarray}
\nonumber R(Z^*) &\geq& e^{-\lambda\norm{\X_{Z^*}(Q^*)}}\exp\left(-\frac{1}{2}\ell(\Z_{\tau^*})-\lambda\sqrt{\ell(\Z_{\tau^*})}+\ell(\Z_{\tau^*})\log\lambda-\log \ell(\Z_{\tau^*})\right)  \\
 \label{eq:ref3-confuse}&\geq& e^{-\lambda\norm{\X_{Z^*}(Q^*)}}\exp\left(-\ell(\Z_{\tau^*})-\lambda\sqrt{\ell(\Z_{\tau^*})}+\ell(\Z_{\tau^*})\log\lambda\right) \\
 \label{eq:denom-lower}&\geq& e^{-\lambda\norm{\X_{Z^*}(Q^*)}-(1+\lambda+\lambda^{-1})\ell(\Z_{\tau^*})}.
\end{eqnarray}
The inequality (\ref{eq:ref3-confuse}) is by $-\log \ell(\Z_{\tau^*})\geq -\frac{1}{2}\ell(\Z_{\tau^*})$ given the fact that $\ell(\Z_{\tau^*})$ is an integer. To obtain (\ref{eq:denom-lower}), we discuss two cases. When $\lambda \geq 1$,
$$-\lambda\sqrt{\ell(\Z_{\tau^*})}+\ell(\Z_{\tau^*})\log\lambda\geq -\lambda\sqrt{\ell(\Z_{\tau^*})}\geq -\lambda\ell(\Z_{\tau^*})\geq -(\lambda+\lambda^{-1})\ell(\Z_{\tau^*}).$$
When $\lambda <1$,
$$-\lambda\sqrt{\ell(\Z_{\tau^*})}+\ell(\Z_{\tau^*})\log\lambda\geq -\ell(\Z_{\tau^*})+\ell(\Z_{\tau^*})\log\lambda\geq -\lambda^{-1}\ell(\Z_{\tau^*})\geq -(\lambda+\lambda^{-1})\ell(\Z_{\tau^*}).$$
Note that (\ref{eq:denom-lower}) is a deterministic lower bound for the denominator $R(Z^*)$. The arguments we have used to derive (\ref{eq:denom-lower}) are greatly inspired by the corresponding ones in \cite{castillo2012needles,castillo2014}.

\paragraph{Upper bounding $R(Z)\mathbb{I}_{E_Z(t)}$.}
To facilitate the analysis, we introduce the object
\begin{equation}
\bar{Q}_Z=\argmin_{Q\in\mathbb{R}^{\ell(\Z_{\tau})}}\norm{\X_Z(Q)-\X_{Z^*}(Q^*)}^2.\label{eq:defQZ}
\end{equation}
The property of least squares implies the following Pythagorean identity,
\begin{equation}
\norm{\X_Z(Q)-\X_{Z^*}(Q^*)}^2 = \norm{\X_Z(Q)-\X_{Z}(\bar{Q}_Z)}^2+\norm{\X_Z(\bar{Q}_Z)-\X_{Z^*}(Q^*)}^2.\label{eq:pyth}
\end{equation}
We first analyze the exponent in the definition of $R(Z)$ on the event $E_Z(t)$ by
\begin{eqnarray}
\nonumber && -\frac{1}{2}\norm{\X_Z(Q)-\X_{Z^*}(Q^*)}^2+\iprod{Y-\X_{Z^*}(Q^*)}{\X_Z(Q)-\X_{Z^*}(Q^*)}-\lambda\norm{\X_Z(Q)} \\
\nonumber &=& -\frac{1}{2}\norm{\X_Z(Q)-\X_{Z^*}(Q^*)}^2+\iprod{W}{\X_Z(Q)-\X_{Z^*}(Q^*)} \\
\nonumber && +\iprod{\theta^*-\X_{Z^*}(Q^*)}{\X_Z(Q)-\X_{Z^*}(Q^*)}-\lambda\norm{\X_Z(Q)} \\
\label{eq:num1} &\leq& -\frac{1}{2}\norm{\X_Z(Q)-\X_{Z^*}(Q^*)}^2+(\sqrt{\epsilon^*(\Z_{\tau})+t}+\lambda)\norm{\X_Z(Q)-\X_{Z^*}(Q^*)} \\
\nonumber && +\norm{\theta^*-\X_{Z^*}(Q^*)}\norm{\X_Z(Q)-\X_{Z^*}(Q^*)}\\
\nonumber && -\lambda\norm{\X_Z(Q)}-\lambda\norm{\X_Z(Q)-\X_{Z^*}(Q^*)} \\
\label{eq:num2} &\leq& 2\left(\sqrt{\epsilon^*(\Z_{\tau})+t}+\lambda\right)^2-\left(\frac{1}{2}-\frac{1}{8}\right)\norm{\X_Z(Q)-\X_{Z^*}(Q^*)}^2 \\
\nonumber && +2\norm{\theta^*-\X_{Z^*}(Q^*)}^2+\frac{1}{8}\norm{\X_Z(Q)-\X_{Z^*}(Q^*)}^2 -\lambda\norm{\X_{Z^*}(Q^*)} \\
\label{eq:num3} &\leq& (4+2/C_2)\epsilon^*(\Z_{\tau})+4t+4\lambda^2-\frac{1}{4}\norm{\X_Z(Q)-\X_{Z^*}(Q^*)}^2 -\lambda\norm{\X_{Z^*}(Q^*)} \\
\label{eq:num4} &\leq& (4+2/C_2)\epsilon^*(\Z_{\tau})+4t+4\lambda^2-\frac{1}{4}\norm{\X_Z(Q)-\X_Z(\bar{Q}_Z)}^2-\lambda\norm{\X_{Z^*}(Q^*)}.
\end{eqnarray}
We have used Cauchy-Schwarz inequality and the event $E_Z(t)$ to get (\ref{eq:num1}). The inequality (\ref{eq:num2}) is due to the fact $ab\leq 2a^2+b^2/8$ for all $a,b\geq 0$ and triangle inequality. By rearrangement and the fact $C_2\norm{\theta^*-\X_{Z^*}(Q^*)}^2\leq \epsilon^*(\Z_{\tau})$, we obtain (\ref{eq:num3}). Finally, the inequality (\ref{eq:num4}) is due to the identity (\ref{eq:pyth}). The above upper bound implies
\begin{eqnarray}
\nonumber R(Z)\mathbb{I}_{E_Z(t)} &\leq& \left(\frac{\lambda}{\sqrt{\pi}}\right)^{\ell(\Z_{\tau})} e^{(4+2/C_2)\epsilon^*(\Z_{\tau})+4t+4\lambda^2-\lambda\norm{\X_{Z^*}(Q^*)}}\\
\nonumber &&\times \sqrt{\det(\X_Z^T\X_Z)}\int e^{-\frac{1}{4}\norm{\X_Z(Q)-\X_Z(\bar{Q}_Z)}^2} dQ \\
\label{eq:porsche} &=& \left(\frac{\lambda}{\sqrt{\pi}}\right)^{\ell(\Z_{\tau})} e^{(4+2/C_2)\epsilon^*(\Z_{\tau})+4t+4\lambda^2-\lambda\norm{\X_{Z^*}(Q^*)}}\int e^{-\frac{1}{4}\norm{b}^2}db \\
\nonumber &=& \left(2\lambda\right)^{\ell(\Z_{\tau})} e^{(4+2/C_2)\epsilon^*(\Z_{\tau})+4t+4\lambda^2-\lambda\norm{\X_{Z^*}(Q^*)}}.
\end{eqnarray}
The change of variable in (\ref{eq:porsche}) uses the same argument in (\ref{eq:denom1}) and (\ref{eq:denom3}).
Using the fact that $\ell(\Z_{\tau})\leq \epsilon^*(\Z_{\tau})$ by (\ref{eq:larger}), we reach the conclusion
\begin{equation}
R(Z)\mathbb{I}_{E_Z(t)}\leq e^{(4+2/C_2+|\log(2\lambda)|)\epsilon^*(\Z_{\tau})+4t+4\lambda^2-\lambda\norm{\X_{Z^*}(Q^*)}}.\label{eq:num-upper}
\end{equation}

\paragraph{Upper bounding $R(Z,U(t))\mathbb{I}_{F_Z(t)}$.}
We require $\delta_2\in(0,1/4)$ throughout the proof.
Let $\xi\in (0,1/4)$ be a constant to be specified later.
When both $F_Z(t)$ and $U(t)$ hold, the exponent in the definition of $R(Z,U(t))$ is bounded by
\begin{eqnarray}
\nonumber && -\frac{1}{2}\norm{\X_Z(Q)-\X_{Z^*}(Q^*)}^2+\iprod{Y-\X_{Z^*}(Q^*)}{\X_Z(Q)-\X_{Z^*}(Q^*)}-\lambda\norm{\X_Z(Q)} \\
\nonumber &=& -\frac{1}{2}\xi\norm{\X_Z(Q)-\X_{Z^*}(Q^*)}^2 + \iprod{W}{\X_Z(Q)-\X_{Z^*}(Q^*)} + \iprod{\theta^*-\X_{Z^*}(Q^*)}{\X_Z(Q)-\X_{Z^*}(Q^*)} \\
\nonumber && -\frac{1}{2}(1-\xi)\norm{\X_Z(Q)-\X_{Z^*}(Q^*)}^2 - \lambda\norm{\X_Z(Q)} \\
\label{eq:numU1} &\leq& -\frac{1}{2}\xi\norm{\X_Z(Q)-\X_{Z^*}(Q^*)}^2 + (\sqrt{\epsilon^*(\Z_{\tau^*})+t}+\lambda)\norm{\X_Z(Q)-\X_{Z^*}(Q^*)} \\
\nonumber && +\iprod{\theta^*-\X_{Z^*}(Q^*)}{\X_Z(Q)-\X_{Z^*}(Q^*)}-\frac{1}{2}(1-\xi)\norm{\X_Z(Q)-\X_{Z^*}(Q^*)}^2 \\
\nonumber && -\lambda\norm{\X_Z(Q)-\X_{Z^*}(Q^*)} - \lambda\norm{\X_Z(Q)} \\
\label{eq:numU2} &\leq& \xi^{-1}\left(\sqrt{\epsilon^*(\Z_{\tau^*})+t}+\lambda\right)^2-\frac{1}{4}\xi\norm{\X_Z(Q)-\X_{Z^*}(Q^*)}^2 \\
\nonumber && -\frac{1}{2}(1-\xi)\norm{\X_Z(Q)-\theta^*}^2+\frac{1}{2}(1+\xi)\norm{\X_{Z^*}(Q^*)-\theta^*}^2 + \xi\iprod{\X_Z(Q)-\theta^*}{\theta^*-\X_{Z^*}(Q^*)} \\
\nonumber && -\lambda\norm{\X_{Z^*}(Q^*)} \\
\label{eq:numU3} &\leq& \xi^{-1}\left(\sqrt{\epsilon^*(\Z_{\tau^*})+t}+\lambda\right)^2-\frac{1}{4}\xi\norm{\X_Z(Q)-\X_{Z^*}(Q^*)}^2-\lambda\norm{\X_{Z^*}(Q^*)} \\
\nonumber && -\frac{1}{2}(1-2\xi)\norm{\X_Z(Q)-\theta^*}^2+\frac{1}{2}(1+2\xi)\norm{\X_{Z^*}(Q^*)-\theta^*}^2 \\
\label{eq:numU4} &\leq& 16\delta_2^{-1}\lambda^2 - \frac{1}{8}M\epsilon(\Z_{\tau^*}) -\frac{1}{8}Mt -\frac{1}{16}\delta_2\norm{\X_{Z^*}(Q^*)-\theta^*}^2 \\
\nonumber && -\frac{1}{32}\delta_2\norm{\X_Z(Q)-\X_{Z}(\bar{Q}_Z)}^2-\lambda\norm{\X_{Z^*}(Q^*)}.
\end{eqnarray}
We have used the event $F_Z$ to get (\ref{eq:numU1}). Now we explain the inequality (\ref{eq:numU2}). Due to the fact that $ab\leq \xi^{-1}a^2+\xi b^2/4$, we have
\begin{eqnarray*}
&& -\frac{1}{2}\xi\norm{\X_Z(Q)-\X_{Z^*}(Q^*)}^2 + (\sqrt{\epsilon^*(\Z_{\tau^*})+t}+\lambda)\norm{\X_Z(Q)-\X_{Z^*}(Q^*)}\\
&\leq& \xi^{-1}\left(\sqrt{\epsilon^*(\Z_{\tau^*})+t}+\lambda\right)^2-\frac{1}{4}\xi\norm{\X_Z(Q)-\X_{Z^*}(Q^*)}^2.
\end{eqnarray*}
It is easy to check the following equality
\begin{eqnarray*}
&& \iprod{\theta^*-\X_{Z^*}(Q^*)}{\X_Z(Q)-\X_{Z^*}(Q^*)}-\frac{1}{2}(1-\xi)\norm{\X_Z(Q)-\X_{Z^*}(Q^*)}^2\\
&=&-\frac{1}{2}(1-\xi)\norm{\X_Z(Q)-\theta^*}^2+\frac{1}{2}(1+\xi)\norm{\X_{Z^*}(Q^*)-\theta^*}^2 + \xi\iprod{\X_Z(Q)-\theta^*}{\theta^*-\X_{Z^*}(Q^*)}.
\end{eqnarray*}
Finally, by triangle inequality, we get
$$-\lambda\norm{\X_Z(Q)-\X_{Z^*}(Q^*)} - \lambda\norm{\X_Z(Q)}\leq -\lambda\norm{\X_{Z^*}(Q^*)}.$$
Then, (\ref{eq:numU3}) is by rearranging (\ref{eq:numU2}) together with the inequality
$$\iprod{\X_Z(Q)-\theta^*}{\theta^*-\X_{Z^*}(Q^*)}\leq \frac{1}{2}\norm{\X_Z(Q)-\theta^*}^2+\frac{1}{2}\norm{\theta^*-\X_{Z^*}(Q^*)}^2.$$
Finally, we have set
\begin{equation}
\xi=\frac{1}{8}\delta_2\quad\text{and}\quad C_2=\frac{1}{128}\delta_2^2\label{eq:C2}
\end{equation}
and used (\ref{eq:pyth})
to obtain (\ref{eq:numU4}) on the event $U(t)$ for all $M>\max\left\{128\delta_2^{-1}C_1,128\delta_2^{-1}\right\}$. Note that we require $\delta_2\in (0,1/4)$ for the inequality (\ref{eq:numU4}). Using the above bound, we have
\begin{eqnarray*}
R(Z,U(t))\mathbb{I}_{F_Z(t)} &\leq& \left(\frac{\lambda}{\sqrt{\pi}}\right)^{\ell(\Z_{\tau})}e^{-\lambda\norm{\X_{Z^*}(Q^*)}+16\delta_2^{-1}\lambda^2 - \frac{1}{8}M\epsilon(\Z_{\tau^*})-\frac{1}{8}Mt-\frac{1}{16}\delta_2\norm{\X_{Z^*}(Q^*)-\theta^*}^2} \\
&& \times \sqrt{\det(\X_Z^T\X_Z)}\int e^{-\frac{1}{32}\delta_2\norm{\X_Z(Q)-\X_Z(\bar{Q}_Z)}^2} dQ \\
&=& \left(\frac{4\lambda}{\sqrt{\delta_2/2}}\right)^{\ell(\Z_{\tau})}e^{-\lambda\norm{\X_{Z^*}(Q^*)}+16\delta_2^{-1}\lambda^2 - \frac{1}{8}M\epsilon(\Z_{\tau^*})-\frac{1}{8}Mt-\frac{1}{16}\delta_2\norm{\X_{Z^*}(Q^*)-\theta^*}^2}.
\end{eqnarray*}
by the same argument in deriving (\ref{eq:num-upper}). By $\ell(\Z_{\tau})\leq \epsilon^*(\Z_{\tau})$ from (\ref{eq:larger}), we reach the conclusion
\begin{equation}
R(Z,U(t))\mathbb{I}_{F_Z(t)} \leq e^{-\lambda\norm{\X_{Z^*}(Q^*)} - \frac{1}{16}M\epsilon(\Z_{\tau^*})-\frac{1}{8}Mt-\frac{1}{16}\delta_2\norm{\X_{Z^*}(Q^*)-\theta^*}^2}, \label{eq:numU-upper}
\end{equation}
for all $M>\max\left\{128\delta_2^{-1}(C_1+1), 16\log(4\lambda/\sqrt{\delta_2/2})+256\delta_2^{-1}\lambda^2\right\}$.

After obtaining the bounds (\ref{eq:denom-lower}), (\ref{eq:num-upper}) and (\ref{eq:numU-upper}), we are ready to prove the main results.
\begin{proof}[Proof of (\ref{eq:main1})]
First, we use (\ref{eq:denom-lower}) and (\ref{eq:num-upper}) to bound the ratio $R(Z)\mathbb{I}_{E_Z(t)}/R(Z^*)$,
\begin{eqnarray*}
|\bar{\Z}_{\tau^*}|\frac{R(Z)\mathbb{I}_{E_Z(t)}}{R(Z^*)} &\leq& e^{4\lambda^2}|\Z_{\tau^*}|\frac{e^{\left[4C_1+2C_1/C_2+C_1|\log(2\lambda)|\right](\epsilon(\Z_{\tau}))+\left[4C_2+2+C_2|\log(2\lambda)|\right]\norm{\X_{Z^*}(Q^*)-\theta^*}^2+4t}}{e^{-(1+\lambda+\lambda^{-1})\ell(\Z_{\tau^*})}} \\
&\leq& e^{4\lambda^2}\exp\left((1+\lambda+\lambda^{-1})\epsilon(\Z_{\tau^*})+C_1'\epsilon(\Z_{\tau})+C_2'\norm{\X_{Z^*}(Q^*)-\theta^*}^2+4t\right),
\end{eqnarray*}
where $C_1'=4C_1+2C_1/C_2+C_1|\log(2\lambda)|$ and $C_2'=4C_2+2+C_2|\log(2\lambda)|$. Consider (\ref{eq:form-dim}) with $\mathcal{A}(t)$. Here, we require that $\delta_1\in(0,1/3)$.
By $Z^*\in\bar{\Z}_{\tau^*}$, we have
\begin{eqnarray}
\label{eq:pf1.1}\mathbb{E}\Pi(\tau\in\mathcal{A}(t)|Y) &\leq& \sum_{\tau\in\mathcal{A}(t)}\frac{\exp\left(-D\epsilon(\Z_{\tau})\right)}{\exp\left(-D\epsilon(\Z_{\tau^*})\right)}\frac{|\bar{\Z}_{\tau^*}|}{|\bar{\Z}_{\tau}|}\sum_{Z\in\bar{\Z}_{\tau}}\mathbb{E}\frac{R(Z)\mathbb{I}_{E_Z(t)}}{R(Z^*)} \\
\label{eq:pf1.2}&& +\sum_{\tau\in\mathcal{A}(t)}\sum_{Z\in\bar{\Z}_{\tau}}\mathbb{P}(E_Z(t)^c).
\end{eqnarray}
According to previous calculations, (\ref{eq:pf1.1}) can be bounded by
\begin{eqnarray}
\label{eq:esfr} && \exp\left(4\lambda^2+(D+\lambda+\lambda^{-1}+1)\epsilon(\Z_{\tau^*})+C_2'\norm{\X_{Z^*}(Q^*)-\theta^*}^2+4t\right) \\
\nonumber && \times\sum_{\tau\in\mathcal{A}(t)}\exp\left(-(D-C_1')\epsilon(\Z_{\tau})\right).
\end{eqnarray}
Then, we can bound the sum in the above display by Lemma \ref{lem:growth}. We take $\alpha=(1+\delta_1)\epsilon(\Z_{\tau^*})+\delta_1\norm{\X_{\Z^*}(Q^*)-\theta^*}^2+ct$ and $\beta=D-C_1'$. Then, Lemma \ref{lem:growth} gives
\begin{eqnarray*}
&& \sum_{\tau\in\mathcal{A}(t)}\exp\left(-(D-C_1')\epsilon(\Z_{\tau})\right) \leq 4\alpha\exp\left(-\beta\floor{\alpha}\right) \leq 4e^{\beta}\exp(-(\beta-1)\alpha) \\
&\leq& 4e^D\exp\left(-(D-C_1'-1)(1+\delta_1)\epsilon(\Z_{\tau^*})-(D-C_1'-1)\delta_1\norm{\X_{\Z^*}(Q^*)-\theta^*}^2-(D-C_1'-1)ct\right).
\end{eqnarray*}
This leads to a bound for (\ref{eq:esfr}) as
\begin{eqnarray*}
&& 4e^{D+4\lambda^2}\exp\left(-\left((D-C_1'-1)\delta_1-C_2'\right)\norm{\X_{Z^*}(Q^*)-\theta^*}^2\right) \\
&& \times \exp\left(-\left((D-C_1'-1)(1+\delta_1)-(D+\lambda+\lambda^{-1}+1)\right)\epsilon(\Z_{\tau^*})-((D-C_1'-1)c-4)t\right) \\
&\leq& 4e^{D+4\lambda^2}\exp\left(-\frac{\delta_1 D}{2}\norm{\X_{Z^*}(Q^*)-\theta^*}^2-\frac{\delta_1D}{2}\epsilon(\Z_{\tau^*})-\frac{Dc}{2}t\right),
\end{eqnarray*}
for $D>\max\left\{\frac{\lambda+\lambda^{-1}+1+2(C_1'+1)}{\delta_1/2}, 2(C_1'+1)+\frac{2C_2'}{\delta_1},\frac{8}{c}+2(C_1'+1)\right\}$.
Using Lemma \ref{lem:noise}, Lemma \ref{lem:growth} and (\ref{eq:larger}), we bound the second term (\ref{eq:pf1.2})  by
\begin{equation}
2\exp\left(-\rho C_2\norm{\X_{Z^*}(Q^*)-\theta^*}^2/16-\rho t/16\right)\sum_{\tau\in\mathcal{A}(t)}\exp\left(-(\rho C_1/16-6)\epsilon(\Z_{\tau})\right).\label{eq:raag}
\end{equation}
Again, we will bound the sum in the above display by Lemma \ref{lem:growth} with $\alpha=(1+\delta_1)\epsilon(\Z_{\tau^*})+\delta_1\norm{\X_{\Z^*}(Q^*)-\theta^*}^2+ct$ and $\beta=\rho C_1/16-6$. That is,
\begin{eqnarray*}
&& \sum_{\tau\in\mathcal{A}(t)}\exp\left(-(\rho C_1/16-6)\epsilon(\Z_{\tau})\right)  \leq 4\alpha\exp\left(-\beta\floor{\alpha}\right) \leq 4e^{\beta}\exp(-(\beta-1)\alpha) \\
&\leq& 4e^{\rho C_1/16}\exp\left(-(\rho C_1/16-7)(1+\delta_1)\epsilon(\Z_{\tau^*})-(\rho C_1/16-7)\delta_1\norm{\X_{\Z^*}(Q^*)-\theta^*}^2-(\rho C_1/16-7)ct\right).
\end{eqnarray*}
Therefore, (\ref{eq:raag}) can be bounded by
\begin{eqnarray*}
&& 8e^{\rho C_1/16}\exp\left(-(\rho C_1/16-7)(1+\delta_1)\epsilon(\Z_{\tau^*})\right)\\
&& \times \exp\left(-(\rho C_1/16+\rho C_2/16-7)\delta_1\norm{\X_{\Z^*}(Q^*)-\theta^*}^2-(\rho (C_1+1)/16-7)ct\right) \\
&\leq & 8\exp\left(-(\rho C_1/16-8)(1+\delta_1)\epsilon(\Z_{\tau^*})\right) \\
&& \times\exp\left(-(\rho C_1/16+\rho C_2/16-7)\delta_1\norm{\X_{\Z^*}(Q^*)-\theta^*}^2-(\rho (C_1+1)/16-7)ct\right) \\
&\leq& 8\exp\left(-7\delta_1\norm{\X_{Z^*}(Q^*)-\theta^*}^2-6\epsilon(\Z_{\tau^*})-7ct\right),
\end{eqnarray*}
for $C_1=\max\{1,224/\rho\}$. We obtain the desired result by combining the bounds of (\ref{eq:pf1.1}) and (\ref{eq:pf1.2}) and setting $t=0$.
\end{proof}

\begin{proof}[Proof of (\ref{eq:main2})]
Let us first use (\ref{eq:denom-lower}) and (\ref{eq:numU-upper}) to bound the ratio $R(Z,U(t))\mathbb{I}_{F_Z(t)}/R(Z^*)$, i.e.,
\begin{eqnarray*}
\frac{R(Z,U(t))\mathbb{I}_{F_Z(t)}}{R(Z^*)} &\leq& \exp\left(-\left(M/16-(1+\lambda+\lambda^{-1})\right)\epsilon(\Z_{\tau^*})-\frac{1}{8}Mt-\frac{1}{16}\delta_2\norm{\X_{Z^*}(Q^*)-\theta^*}^2\right) \\
&\leq& \exp\left(-\frac{M}{32}\epsilon(\Z_{\tau^*})-\frac{1}{8}Mt-\frac{1}{16}\delta_2\norm{\X_{Z^*}(Q^*)-\theta^*}^2\right),
\end{eqnarray*}
for $M>\max\left\{128\delta_2^{-1}(C_1+1), 16\log(4\lambda/\sqrt{\delta_2/2})+256\delta_2^{-1}\lambda^2, 32(1+\lambda+\lambda^{-1})\right\}$. By (\ref{eq:form-pred}), we have
\begin{eqnarray}
\label{eq:pf2.1} \mathbb{E}\Pi(\X_Z(Q)\in U(t)|Y) &\leq& \sum_{\tau\in\T\cap\mathcal{A}(t)^c }\frac{\exp\left(-D\epsilon(\Z_{\tau})\right)}{\exp\left(-D\epsilon(\Z_{\tau^*})\right)}\frac{|\bar{\Z}_{\tau^*}|}{|\bar{\Z}_{\tau}|}\sum_{Z\in\bar{\Z}_{\tau}}\mathbb{E}\frac{R(Z,U(t))\mathbb{I}_{F_Z}}{R(Z^*)} \\
\label{eq:pf2.2} && + \sum_{\tau\in\T\cap\mathcal{A}(t)^c }\sum_{Z\in\bar{\Z}_{\tau}}\mathbb{P}(F_Z(t)^c) \\
\label{eq:pf2.3} && + \mathbb{E}\Pi(\tau\in\mathcal{A}(t)|Y).
\end{eqnarray}
The bound for (\ref{eq:pf2.3}) has been derived in the proof of (\ref{eq:main1}).
Using Lemma \ref{lem:growth}, we bound (\ref{eq:pf2.1}) by
\begin{eqnarray*}
&& \exp\left(-\left(\frac{M}{32}-D-1\right)\epsilon(\Z_{\tau^*})-\frac{1}{8}Mt-\frac{1}{16}\delta_2\norm{\X_{Z^*}(Q^*)-\theta^*}^2\right)\sum_{\tau\in\T\cap\mathcal{A}(t)^c }\exp\left(-D\epsilon(\Z_{\tau})\right) \\
&\leq& 6\exp\left(-\frac{M}{64}\epsilon(\Z_{\tau^*})-\frac{1}{8}Mt-\frac{1}{16}\delta_2\norm{\X_{Z^*}(Q^*)-\theta^*}^2\right),
\end{eqnarray*}
for $M>\max\left\{128\delta_2^{-1}(C_1+1), 16\log(4\lambda/\sqrt{\delta_2/2})+256\delta_2^{-1}\lambda^2, 32(1+\lambda+\lambda^{-1}),64(D+1)\right\}$.
Using Lemma \ref{lem:noise} and (\ref{eq:larger}), the term (\ref{eq:pf2.2}) is bounded by
$$2\exp\left(-\rho C_1\epsilon(\Z_{\tau^*})/16-\rho C_2\norm{\X_{Z^*}(Q^*)-\theta^*}^2-\frac{\rho t}{16}\right)\sum_{\tau\in\T\cap\mathcal{A}(t)^c }\exp\left(5\epsilon(\Z_{\tau})\right).$$
We use Lemma \ref{lem:growth} to bound the sum in the above display with $\alpha=(1+\delta_1)\epsilon(\Z_{\tau^*})+\delta_1\norm{\X_{\Z^*}(Q^*)-\theta^*}^2+ct$ and $\beta=5$.
\begin{eqnarray*}
&& \sum_{\tau\in\T\cap\mathcal{A}(t)^c }\exp\left(5\epsilon(\Z_{\tau})\right) \leq 4(\alpha+1)\exp\left(\beta(\alpha+1)\right) \leq 4e^{\beta+1}\exp\left((\beta+1)\alpha\right) \\
&=& 4e^6\exp\left(6(1+\delta_1)\epsilon(\Z_{\tau^*})+6\delta_1\norm{\X_{\Z^*}(Q^*)-\theta^*}^2+6ct\right).
\end{eqnarray*}
Therefore, we can bound (\ref{eq:pf2.2}) by
\begin{eqnarray*}
&&  8e^6\exp\left(-\left(\frac{\rho C_1}{16}-8\right)\epsilon(\Z_{\tau^*})-\left(\rho C_2-6\delta_1\right)\norm{\X_{Z^*}(Q^*)-\theta^*}^2-\left(\frac{\rho}{16}-6c\right)t\right) \\
&\leq& 8e^6\exp\left(-6\epsilon(\Z_{\tau^*})-\delta_1\norm{\X_{Z^*}(Q^*)-\theta^*}^2-\frac{\rho}{32}t\right),
\end{eqnarray*}
where we set $C_2=\delta_2^2/128$, $C_1=\max\{1,224/\rho\}$, $\delta_2=8\sqrt{14\delta_1/\rho}=8\sqrt{14\delta/\rho}$, and $c=\rho/192$. The proof is complete by combining the bounds of (\ref{eq:pf2.1}), (\ref{eq:pf2.2}) and (\ref{eq:pf2.3}) and setting $t=0$.
\end{proof}

\begin{proof}[Proof of (\ref{eq:main33})]
In the proof of (\ref{eq:main2}), we obtain a general bound for $\mathbb{E}\Pi(\X_Z(Q)\in U(t)|Y)$ for any $t\geq 0$. The result of (\ref{eq:main33}) can be obtained by integrating out the tail probability $\mathbb{E}\Pi(\X_Z(Q)\in U(t)|Y)$. The details of the argument is given in Section \ref{sec:app-pf-9} in the supplement.
\end{proof}

\section*{Acknowledgement}

This work was done during the first author's visit in Leiden University in 2015. Many ideas were originated from the weekly problem sessions with Johannes Schmidt-Hieber and Kolyan Ray. Johannes Schmidt-Hieber suggested using $\ell_2$ norm in the exponent of the prior. Isma{\"e}l Castillo pointed out that the rate in Corollary \ref{cor:slr-pred} can be improved, which leads to Corollary \ref{cor:general-agg}.

\bibliographystyle{plainnat}
\bibliography{reference}

\newpage

\begin{center}
{\Large Supplement to ``A General Framework for Bayes Structured Linear Models''}\\
~\\
Chao Gao, Aad W. van der Vaart \& Harrison H. Zhou
\end{center}

\appendix

\section{Elliptical Laplace distribution}\label{sec:ELD}

Consider a distribution in $\mathbb{R}^{\ell(\Z_{\tau})}$, whose density function is proportional to $\exp\left(-\lambda\|\X_Z(Q)\|\right)$. The design matrix $\X_Z\in\mathbb{R}^{N\times \ell(\Z_{\tau})}$ is of rank $\ell(\Z_{\tau})$, which is no greater then $N$. Consider the singular value decomposition of $\X_Z$, which is $\X_Z=\mathcal{U}\Lambda \mathcal{V}^T$, where $\mathcal{U}\in\mathbb{R}^{N\times \ell(\Z_{\tau})}$ and $\mathcal{V}\in\mathbb{R}^{\ell(\Z_{\tau})\times \ell(\Z_{\tau})}$ are orthonormal matrices, and $\Lambda$ is an $\ell(\Z_{\tau})\times \ell(\Z_{\tau})$ diagonal matrix with positive diagonal entries. The normalizing constant of the distribution is determined by the integral $\int_{\mathbb{R}^{\ell(\Z_{\tau})}}\exp\left(-\lambda\|\X_Z(Q)\|\right)dQ$. Since $\mathcal{U}$ is orthonormal, we have
$$\int_{\mathbb{R}^{\ell(\Z_{\tau})}}\exp\left(-\lambda\|\X_Z(Q)\|\right)dQ=\int_{\mathbb{R}^{\ell(\Z_{\tau})}}\exp\left(-\lambda\|\Lambda \mathcal{V}^TQ\|\right)dQ.$$
Consider a change of variable $b=\Lambda \mathcal{V}^TQ\in\mathbb{R}^{\ell(\Z_{\tau})}$, and we have
$$\int_{\mathbb{R}^{\ell(\Z_{\tau})}}\exp\left(-\lambda\|\Lambda \mathcal{V}^TQ\|\right)dQ=\frac{1}{\sqrt{\det(\mathcal{V}\Lambda^2\mathcal{V}^T)}}\int_{\mathbb{R}^{\ell(\Z_{\tau})}}e^{-\lambda\|b\|}db.$$
It is easy to see that $\det(\mathcal{V}\Lambda^2\mathcal{V}^T)=\det(\X_Z^T\X_Z)$. Moreover,
$$\int_{\mathbb{R}^{\ell(\Z_{\tau})}}e^{-\lambda\|b\|}db=\frac{2\pi^{\ell(\Z_{\tau})/2}}{\Gamma(\ell(\Z_{\tau})/2)}\int r^{\ell(\Z_{\tau})-1}e^{-\lambda r}dr=2\left(\frac{\sqrt{\pi}}{\lambda}\right)^{\ell(\Z_{\tau})}\frac{\Gamma(\ell(\Z_{\tau}))}{\Gamma(\ell(\Z_{\tau})/2)}.$$
Therefore,
$$\int_{\mathbb{R}^{\ell(\Z_{\tau})}}\exp\left(-\lambda\|\X_Z(Q)\|\right)dQ=\frac{2}{\sqrt{\det(\X_Z^T\X_Z)}}\left(\frac{\sqrt{\pi}}{\lambda}\right)^{\ell(\Z_{\tau})}\frac{\Gamma(\ell(\Z_{\tau}))}{\Gamma(\ell(\Z_{\tau})/2)},$$
and thus, the density function of the elliptical Laplace distribution is
$$\frac{\sqrt{\det(\X_Z^T\X_Z)}}{2}\left(\frac{\lambda}{\sqrt{\pi}}\right)^{\ell(\Z_{\tau})}\frac{\Gamma(\ell(\Z_{\tau})/2)}{\Gamma(\ell(\Z_{\tau}))}e^{-\lambda\|\X_Z(Q)\|}.$$

\section{Proof of (\ref{eq:main33}) in Theorem \ref{thm:main}}\label{sec:app-pf-9}

In the proof of (\ref{eq:main2}), we obtain a general bound for $\mathbb{E}\Pi(\X_Z(Q)\in U(t)|Y)$ for any $t\geq 0$. That is,
\begin{eqnarray}
\nonumber && \mathbb{E}\Pi\left(\norm{\X_Z(Q)-\theta^*}^2>(1+\delta_2)\norm{\X_{Z^*}(Q^*)-\theta^*}^2+M\epsilon(\Z_{\tau^*})+Mt\Big|Y\right) \\
\label{eq:main3} &\leq& \exp\left(-\tilde{C}\left(\epsilon(\Z_{\tau^*})+\norm{\X_{Z^*}(Q^*)-\theta^*}^2+t\right)\right),
\end{eqnarray}
for some constant $\tilde{C}>0$.

Now we give a bound for the risk of the posterior mean. First, using Jensen's inequality, we get
$$\mathbb{E}\|\mathbb{E}_{\Pi}(\X_Z(Q)|Y)-\theta^*\|^2\leq \mathbb{E}\mathbb{E}_{\Pi}\left(\|\X_Z(Q)-\theta^*\|^2|Y\right).$$
By the property of expectation, we have
$$\mathbb{E}\mathbb{E}_{\Pi}\left(\|\X_Z(Q)-\theta^*\|^2|Y\right)=\int_0^{\infty}\mathbb{E}\Pi\left(\|\X_Z(Q)-\theta^*\|^2>x|Y\right)dx.$$
Therefore, it is sufficient to bound the integral on the right hand side of the above equality. Define
$$R=(1+\delta_2)\norm{\X_{Z^*}(Q^*)-\theta^*}^2+M\epsilon(\Z_{\tau^*}).$$
Then,
\begin{eqnarray*}
&& \int_0^{\infty}\mathbb{E}\Pi\left(\|\X_Z(Q)-\theta^*\|^2>x|Y\right)dx \\
&=& \int_0^{R}\mathbb{E}\Pi\left(\|\X_Z(Q)-\theta^*\|^2>x|Y\right)dx + \int_R^{\infty}\mathbb{E}\Pi\left(\|\X_Z(Q)-\theta^*\|^2>x|Y\right)dx \\
&\leq& R + \int_R^{\infty}\mathbb{E}\Pi\left(\|\X_Z(Q)-\theta^*\|^2>x|Y\right)dx \\
&=& R + \int_0^{\infty}\mathbb{E}\Pi\left(\|\X_Z(Q)-\theta^*\|^2>R+x|Y\right)dx \\
&\leq& R + \int_0^{\infty}\exp\left(-\tilde{C}\left(\epsilon(\Z_{\tau^*})+\norm{\X_{Z^*}(Q^*)-\theta^*}^2+\frac{x}{M}\right)\right)dx \\
&=& R + \frac{M}{\tilde{C}}\exp\left(-\tilde{C}\left(\epsilon(\Z_{\tau^*})+\norm{\X_{Z^*}(Q^*)-\theta^*}^2\right)\right).
\end{eqnarray*}
This completes the proof.

\section{Proofs of corollaries}

\begin{proof}[Proofs of Corollary \ref{cor:main} and Corollaries \ref{cor:SBM}-\ref{cor:dic}]
Corollary \ref{cor:main} is a direct consequence of Theorem \ref{thm:main} by letting $\theta^*=\X_{Z^*}(Q^*)$. Except Corollary \ref{cor:slr-esti}, Corollaries \ref{cor:SBM}-\ref{cor:dic} are special cases of Corollary \ref{cor:main} in different model settings. By the definitions of $\kappa_1$ and $\kappa_2$, we have $\norm{\beta-\beta^*}^2\leq \kappa_2^{-2}\norm{X\beta-X\beta^*}^2/n$ and $\norm{\beta-\beta^*}^2_1\leq \kappa_1^{-2}s^*\norm{X\beta-X\beta^*}^2/n$, which implies Corollary \ref{cor:slr-esti} from Corollary \ref{cor:slr-pred}.
\end{proof}

\begin{proof}[Proof of Corollary \ref{cor:graphon}]
For any $\xi$, recall that $f(\xi_i,\xi_j)=\theta_{ij}=Q_{z(i)z(j)}$. Then, (\ref{eq:main2}) of Theorem \ref{thm:main} implies that
$$\sum_{i,j}\left(f(\xi_i,\xi_j)-f^*(\xi_i,\xi_j)\right)^2\leq (1+\delta_2)\sum_{i,j}\left(Q^*_{z^*(i)z^*(j)}-f^*(\xi_i,\xi_j)\right)^2+M\left((k^*)^2+n\log k^*\right)$$
under the posterior distribution for any $k^*\in[n]$, any $z^*\in\bar{\Z}_{k^*}$ and any $Q^*\in\mathbb{R}^{(k^*)^2}$.  Lemma 2.1 of \cite{gao2014rate} implies there exist some $z^*\in[k^*]^n$ and some $Q^*\in\mathbb{R}^{(k^*)^2}$ such that
$$\sum_{i,j}\left(Q^*_{z^*(i)z^*(j)}-f^*(\xi_i,\xi_j)\right)^2\leq C_3L^2n^2\left(\frac{1}{k^*}\right)^{\alpha\wedge 1},$$
for any $f^*\in\mathcal{F}_{\alpha}(L)$ and some absolute constant $C_3>0$. Therefore,
$$\frac{1}{n^2}\sum_{i,j}\left(f(\xi_i,\xi_j)-f^*(\xi_i,\xi_j)\right)^2\leq M'\left(\left(\frac{1}{k^*}\right)^{\alpha\wedge 1}+\left(\frac{k^*}{n}\right)^2+\frac{\log k^*}{n}\right).$$
The proof is complete by choosing $k^*=\ceil{n^{\frac{1}{\alpha\wedge 1+1}}}$.
\end{proof}

To prove Corollary \ref{cor:slrq-pred}, we need the following result, which is Lemma 7.2 of \cite{tsybakov2014aggregation}.
\begin{lemma}\label{lem:maurey}
Assume $\max_{j\in[p]}n^{-1/2}\|X_{*j}\|\leq L$ for some constant $L>0$. For any $\beta^*\in\mathcal{B}_q(k)$ with $q\in(0,1]$, and any $s\in[p]$, there exists a $\beta_0\in\mathcal{B}_0(2s)$, such that
$$\frac{1}{n}\|X\beta^*-X\beta_0\|^2\leq L^2k^{2/q}s^{1-2/q}.$$
\end{lemma}

\begin{proof}[Proof of Corollary \ref{cor:slrq-pred}]
The case $q=0$ is Corollary \ref{cor:slr-pred}. We consider $q\in(0,1]$. For the effective sparsity defined in Section \ref{sec:weak-ball}, (\ref{eq:main2}) of Theorem \ref{thm:main} implies that
$$\norm{X\beta-X\beta^*}^2\leq (1+\delta_2)\norm{X\beta_0-X\beta^*}^2+Ms^*\log\frac{ep}{s^*}$$
under the posterior distribution for any $\beta_0\in\mathcal{B}_0(s^*)$. We choose $\beta_0\in\mathcal{B}_0(s^*)$ so that the bound of Lemma \ref{lem:maurey} is satisfied.
This leads to
\begin{equation}
\frac{1}{n}\norm{X\beta-X\beta^*}^2\leq M'\left(k^{2/q}(s^*)^{1-2/q}+\frac{s^*\log\frac{ep}{s^*}}{n}\right).\label{eq:mmm}
\end{equation}
Since $s^*=\ceil{x^*}$, we have $s^*+1>x^*$. By the definition, $x^*$ is the largest number $x$ such that $x\leq k\left(\frac{n}{\log(ep/x)}\right)^{q/2}$. This implies $s^*+1>k\left(\frac{n}{\log(ep/(s^*+1))}\right)^{q/2}$. After rearrangement, we get
$$\left(\frac{k}{s^*+1}\right)^{q/2}<\frac{\log\frac{ep}{s^*+1}}{n},$$
which implies that $\frac{s^*\log\frac{ep}{s^*}}{n}$ dominates $k^{2/q}(s^*)^{1-2/q}$ in (\ref{eq:mmm}). This leads to the desired result.
\end{proof}

\begin{proof}[Proof of Corollary \ref{cor:wavelet}]
For every $j<\log_2n$, the model induced by the prior can be represented in the general framework by letting $Z_j=S_j$, $\tau_j=s_j$, $\T_j=[2^j]$, $\Z_{s_j}=\{S_j\subset[2^j]:|S_j|=s_j\}$, $\ell(\Z_{s_j})=s_j$ and $Q_j=\theta_{jS_j}$. Then, we have the representation $\X_{Z_j}(Q_j)=\sqrt{n}(\theta_{jS_j}^T,0^T_{jS_j^c})^T$. The complexity function is $\epsilon_j(\Z_{s_j})=2s_j\log\frac{e2^j}{s_j}$, which satisfies (\ref{eq:larger}) and (\ref{eq:capacity}). By (\ref{eq:main3}) and letting $t=n^{\frac{1}{2\alpha+1}}/\log_2n$, we have
\begin{eqnarray*}
&& \mathbb{E}\Pi\left(n\norm{\theta_{j*}-\theta_{j*}^*}^2>(1+\delta_2)n\norm{\bar{\theta}_{j*}-\theta_{j*}^*}^2+2Ms_j^*\log\frac{e2^j}{s_j^*}+\frac{n^{\frac{1}{2\alpha+1}}}{\log_2n}\Big|Y_{j*}\right) \\
&\leq& \exp\left(-C''\frac{n^{\frac{1}{2\alpha+1}}}{\log_2n}\right),
\end{eqnarray*}
for any $\bar{\theta}_{j*}\in\mathbb{R}^{2^j}$ with sparsity $s_j^*$. 
We define a control function in \cite{johnstone2011gaussian}[Equation (11.35)], $r_{m,p}(t,\epsilon)=\epsilon^2r_{m,p}(t/\epsilon)$. For $p<2$,
$$r_{m,p}(t)=\begin{cases}
t^2, & t\leq\sqrt{1+\log m},\\
t^p[1+\log(m/t^p)]^{1-p/2}, & \sqrt{1+\log m}\leq t\leq m^{1/p}, \\
m, & t\geq m^{1/p}.
\end{cases}$$
For $p\geq 2$,
$$r_{m,p}(t)=\begin{cases}
m^{1-2/p}t^2, & t\leq m^{1/p},\\
m, & t\geq m^{1/p}.
\end{cases}$$
Since $\theta^*\in\Theta_{p,q}^{\alpha}(L)$ implies $\norm{\theta_{j*}^*}_p\leq L2^{-aj}$, we have
$$\norm{\bar{\theta}_{j*}-\theta_{j*}^*}^2\leq C^*r_{2^j,p}(L2^{-aj},n^{-1/2})$$
for some absolute constant $C^*>0$ by the proof of Theorem 11.7 in \cite{johnstone2011gaussian}. Therefore,
$$\mathbb{E}\Pi(G_j^c|Y_{j*})\leq \exp\left(-C''\frac{n^{\frac{1}{2\alpha+1}}}{\log_2n}\right)$$
for all $j<\log_2n$, where
$$G_j=\left\{\norm{\theta_{j*}-\theta_{j*}^*}^2\leq M'r_{2^j,p}(L2^{-aj},n^{-1/2})+\frac{n^{-\frac{2\alpha}{2\alpha+1}}}{\log_2n}\right\}.$$
Moreover, $\Pi(\theta_{j*}=0|Y_{j*})=1$ for all $j\geq \log_2n$ by the definition of the prior. Using the independence structure of the posterior distribution, we have
\begin{eqnarray*}
&&\mathbb{E}\Pi\left(\left(\cap_{j<\log_2n}G_j\right)^c|Y\right) \leq \sum_{j<\log_2n}\mathbb{E}\Pi(G_j^c|Y) = \sum_{j<\log_2n}\mathbb{E}\Pi(G_j^c|Y_{j*}) \\
&\leq& (\log_2n)\exp\left(-C''\frac{n^{\frac{1}{2\alpha+1}}}{\log_2n}\right) \leq \exp\left(-\bar{C}\frac{n^{\frac{1}{2\alpha+1}}}{\log n}\right).
\end{eqnarray*}
Finally, the event $\cap_{j<\log_2n}G_j$ and $\theta_{j*}=0$ for all $j\geq\log_2n$ implies
\begin{eqnarray*}
\norm{\theta-\theta^*}^2 &\leq& \sum_{j<\log_2n}\norm{\theta_{j*}-\theta_{j*}^*}^2+\sum_{j\geq\log_2n}\norm{\theta_{j*}^*}^2 \\
&\leq& M'\sum_{j<\log_2n}\left(r_{2^j,p}(L2^{-aj},n^{-1/2})+\frac{n^{-\frac{2\alpha}{2\alpha+1}} }{\log_2n}\right)+\sum_{j\geq\log_2n}\norm{\theta_{j*}^*}^2\\
&\leq& M''n^{-\frac{2\alpha}{2\alpha+1}},
\end{eqnarray*}
where the last inequality follows the proof of Theorem 12.1 in  \cite{johnstone2011gaussian} under the assumption $\alpha\geq \frac{1}{p}$. Hence, the proof is complete.
\end{proof}

\begin{proof}[Proof of Corollary \ref{cor:sobolev}]
Let us write the model induced by the prior distribution in the general framework by letting $Z=[k]$, $\tau=k$, $\T=[n]$, $\Z_k=\{[k]\}$, $\ell(\Z_k)=k$ and $Q=\theta_{[k]}$. Then, we have the representation $\X_Z(Q)=\sqrt{n}(\theta_{[k]}^T,0_{[k]^c}^T)^T$. The complexity function $\epsilon(\Z_k)$ is $2k$, which satisfies (\ref{eq:larger}) and (\ref{eq:capacity}). Then, (\ref{eq:main2}) of Theorem \ref{thm:main} implies that
$$\mathbb{E}\Pi\left(n\norm{\theta-\theta^*}^2>(1+\delta_2)n\norm{\bar{\theta}-\theta^*}^2+2Mk^*\Big|Y\right)\leq \exp\left(-C''\left(k^*+\norm{\bar{\theta}-\theta^*}^2\right)\right)$$
for any $\bar{\theta}$ satisfying $\bar{\theta}_j=0$ for $j>k^*$. Since $\theta^*\in\mathcal{S}_{\alpha}(L)$, there exists some $\bar{\theta}$ satisfying $\bar{\theta}_j=0$ for $j>k^*$ such that $\norm{\bar{\theta}-\theta^*}^2\leq L^2(k^*)^{-2\alpha}$. Therefore, $\norm{\theta-\theta^*}^2\leq M'\left((k^*)^{-2\alpha}+\frac{k^*}{n}\right)$ under the posterior distribution. Let $k^*=\ceil{n^{\frac{1}{2\alpha+1}}}$, and the proof is complete.
\end{proof}

\begin{proof}[Proof of Corollary \ref{cor:general-agg}]
Note that the model induced by the prior distribution can be written in a general way by letting $Z=S$, $\tau=s$, $\T=[r]$, $\Z_s=\{S\subset[p]:|S|=s\}$ if $s<r$ and $\Z_r=\{[r]\}$, $\ell(\Z_s)=s$ and $Q=\beta_S$. Then, we have the representation $\X_Z(Q)=X_{*S}\beta_S=X\beta$. The complexity function we choose is $\epsilon(\Z_s)=2s\log\frac{ep}{s}$ for $s<r$ and $\epsilon(\Z_r)=2r$. It is easy to check that $\epsilon(\Z_s)$ satisfies (\ref{eq:larger}) and (\ref{eq:capacity}). Using (\ref{eq:main2}) of Theorem \ref{thm:main}, we have
\begin{eqnarray*}
&& \mathbb{E}\Pi\left(\norm{f_{\beta}-f^*}_n^2>(1+\delta_2)\norm{f_{\beta^*}-f^*}_n^2+2M\frac{s^*\log(ep/s^*)}{n}\Big|Y\right) \\
&\leq& \exp\left(-C''\left(n\norm{f_{\beta^*}-f^*}_n^2+s^*\log\frac{ep}{s^*}\right)\right),
\end{eqnarray*}
for any $\beta^*$ with sparsity $s^*$. For this $\beta^*$, there exists some $\beta_1$ such that $\supp(\beta_1)\subset[r]$ and $f_{\beta^*}=f_{\beta_1}$. Therefore, (\ref{eq:main2}) of Theorem \ref{thm:main} implies
\begin{eqnarray*}
&& \mathbb{E}\Pi\left(\norm{f_{\beta}-f^*}_n^2>(1+\delta_2)\norm{f_{\beta_1}-f^*}_n^2+2M\frac{r}{n}\Big|Y\right) \\
&\leq& \exp\left(-C''\left(n\norm{f_{\beta_1}-f^*}_n^2+r\right)\right).
\end{eqnarray*}
Combining the two results by union bound, the proof is complete.
\end{proof}

\begin{proof}[Proof of Corollary \ref{cor:universal}]
Using the corresponding arguments in \cite{rigollet2012sparse,tsybakov2014aggregation}, Corollary \ref{cor:universal} is implied by Corollary \ref{cor:general-agg}.
\end{proof}

\section{Proofs of technical results}\label{sec:pftech}

\begin{proof}[Proof of Lemma \ref{lem:noise}]
Consider $\bar{Q}_Z$ defined in (\ref{eq:defQZ}). Then, we have the bound
\begin{eqnarray*}
&& \left|\iprod{W}{\X_Z(Q)-\X_{Z^*}(Q^*)}\right| \\
&\leq& \norm{\X_Z(Q)-\X_Z(\bar{Q}_Z)}\left|\iprod{W}{\frac{\X_Z(Q)-\X_Z(\bar{Q}_Z)}{\norm{\X_Z(Q)-\X_Z(\bar{Q}_Z)}}}\right| \\
&& + \norm{\X_Z(\bar{Q}_Z)-\X_{Z^*}(Q^*)}\left|\iprod{W}{\frac{\X_Z(\bar{Q}_Z)-\X_{Z^*}(Q^*)}{\norm{\X_Z(\bar{Q}_Z)-\X_{Z^*}(Q^*)}}}\right| \\
&\leq& \max\left\{\left|\iprod{W}{\frac{\X_Z(Q)-\X_Z(\bar{Q}_Z)}{\norm{\X_Z(Q)-\X_Z(\bar{Q}_Z)}}}\right|,\left|\iprod{W}{\frac{\X_Z(\bar{Q}_Z)-\X_{Z^*}(Q^*)}{\norm{\X_Z(\bar{Q}_Z)-\X_{Z^*}(Q^*)}}}\right|\right\} \\
&& \times\sqrt{2}\sqrt{\norm{\X_Z(Q)-\X_Z(\bar{Q}_Z)}^2+\norm{\X_Z(\bar{Q}_Z)-\X_{Z^*}(Q^*)}^2} \\
&=&\sqrt{2} \max\left\{\left|\iprod{W}{\frac{\X_Z(Q)-\X_Z(\bar{Q}_Z)}{\norm{\X_Z(Q)-\X_Z(\bar{Q}_Z)}}}\right|,\left|\iprod{W}{\frac{\X_Z(\bar{Q}_Z)-\X_{Z^*}(Q^*)}{\norm{\X_Z(\bar{Q}_Z)-\X_{Z^*}(Q^*)}}}\right|\right\}  \norm{\X_Z(Q)-\X_{Z^*}(Q^*)},
\end{eqnarray*}
where the last equality is due to (\ref{eq:pyth}).
By (\ref{eq:subG}), $\left|\iprod{W}{\frac{\X_Z(\bar{Q}_Z)-\X_{Z^*}(Q^*)}{\norm{\X_Z(\bar{Q}_Z)-\X_{Z^*}(Q^*)}}}\right|\leq \frac{1}{\sqrt{2}}\sqrt{\epsilon^*(\Z_{\tau})+t}$ with probabilty at least $1-\exp(-\rho (\epsilon^*(\Z_{\tau})+t)/4)$. Now it is sufficient to bound
$$\sup_{Q\in\mathbb{R}^{\ell(\Z_{\tau})}}\left|\iprod{W}{\frac{\X_Z(Q)-\X_Z(\bar{Q}_Z)}{\norm{\X_Z(Q)-\X_Z(\bar{Q}_Z)}}}\right|=\sup_{\{Q\in\mathbb{R}^{\ell(\Z_{\tau})}: \norm{\X_Z(Q)}\leq 1\}}\left|\iprod{W}{\X_Z(Q)}\right|.$$
A standard discretization argument as Lemma A.1 in \cite{gao2014rate} gives
\begin{equation}
\sup_{\{Q\in\mathbb{R}^{\ell(\Z_{\tau})}: \norm{\X_Z(Q)}\leq 1\}}\left|\iprod{W}{\X_Z(Q)}\right|\leq 2\max_{1\leq l\leq L}\left|\iprod{W}{\X_Z(Q_l)}\right|,\label{eq:ref-asked}
\end{equation}
where $\{Q_l\}_{1\leq l\leq L}$ is a subset of $\{Q\in\mathbb{R}^{\ell(\Z_{\tau})}: \norm{\X_Z(Q)}\leq 1\}$ such that for any $Q\in\mathbb{R}^{\ell(\Z_{\tau})}$ with $\norm{\X_Z(Q)}\leq 1$, there exists an $l\in[L]$ that satisfies $\norm{\X_Z(Q-Q_l)}\leq 1/2$ and a covering number argument gives the bound $L\leq \exp\left(5\ell(\Z_{\tau})\right)$. We will give a rigorous argument for (\ref{eq:ref-asked}) at the end of the proof. Using union bound together with (\ref{eq:subG}), we have 
\begin{eqnarray*}
&& \mathbb{P}\left(\max_{1\leq l\leq L}\left|\iprod{W}{\X_Z(Q_l)}\right|> \frac{1}{2\sqrt{2}}\sqrt{\epsilon^*(\Z_{\tau})+t}\right) \\
&\leq& \sum_{l=1}^L\mathbb{P}\left(\left|\iprod{W}{\X_Z(Q_l)}\right|> \frac{1}{2\sqrt{2}}\sqrt{\epsilon^*(\Z_{\tau})+t}\right) \\
&\leq& \exp\left(5\ell(\Z_{\tau})-\rho (\epsilon^*(\Z_{\tau})+t)/16\right) \\
&\leq& \exp\left(-(\rho C_1/16-5)\epsilon(\Z_{\tau})-\rho C_2\norm{\X_{Z^*}(Q^*)-\theta^*}^2/16-\rho t/16\right),
\end{eqnarray*}
where we have used the condition (\ref{eq:larger}).
Using union bound again, we have
$$\sqrt{2}\max\left\{\left|\iprod{W}{\frac{\X_Z(Q)-\X_Z(\bar{Q}_Z)}{\norm{\X_Z(Q)-\X_Z(\bar{Q}_Z)}}}\right|,\left|\iprod{W}{\frac{\X_Z(\bar{Q}_Z)-\X_{Z^*}(Q^*)}{\norm{\X_Z(\bar{Q}_Z)-\X_{Z^*}(Q^*)}}}\right|\right\}\leq \sqrt{\epsilon^*(\Z_{\tau})+t},$$
with probability at least $1-2\exp\left(-(\rho C_1/16-5)\epsilon(\Z_{\tau})-\rho C_2\norm{\X_{Z^*}(Q^*)-\theta^*}^2/16-\rho t/16\right).$ This leads to the bound $\mathbb{P}(E_Z^c)\leq 2\exp\left(-(\rho C_1/16-5)\epsilon(\Z_{\tau})-\rho C_2\norm{\X_{Z^*}(Q^*)-\theta^*}^2/16-\rho t/16\right)$. A similar argument also leads to the bound
$$\mathbb{P}(F_Z^c)\leq 2\exp\left(5\ell(\Z_{\tau})-\rho C_1\epsilon(\Z_{\tau^*})/16-\rho C_2\norm{\X_{Z^*}(Q^*)-\theta^*}^2/16-\rho t/16\right).$$

Finally, we provide a argument for (\ref{eq:ref-asked}). Note that this discretization technique is standard in the literature \citep{vershynin2010introduction,tao2012topics}. We still give a proof here to be self-contained. Consider a subset $\mathcal{Q}=\{Q_1,...,Q_L\}\subset\{Q\in\mathbb{R}^{\ell(\Z_{\tau})}: \norm{\X_Z(Q)}\leq 1\}$, such that for any $Q\in\mathbb{R}^{\ell(\Z_{\tau})}$ with $\norm{\X_Z(Q)}\leq 1$, there exists a $Q_l\in\mathcal{Q}$ that satisfies $\norm{\X_Z(Q-Q_l)}\leq 1/2$. In particular, we choose $\mathcal{Q}$ to be the set with such property with the smallest cardinality $|\mathcal{Q}|=L$. Thus, the number $L$ is called the covering number, whose bound is given by the quantity of volume ratio \citep{geer2000empirical},
$$L\leq \frac{\text{Vol}\left(\{Q\in\mathbb{R}^{\ell(\Z_{\tau})}:\norm{\X_Z(Q)}\leq 5/4\}\right)}{\text{Vol}\left(\{Q\in\mathbb{R}^{\ell(\Z_{\tau})}:\norm{\X_Z(Q)}\leq 1/4\}\right)}\leq 5^{\ell(\Z_{\tau})}\leq \exp\left(5\ell(\Z_{\tau})\right).$$
To derive (\ref{eq:ref-asked}), let $Q_l\in\mathcal{Q}$ be the one that is close to $Q$ such that $\norm{\X_Z(Q-Q_l)}\leq 1/2$. Then,
\begin{eqnarray*}
\left|\iprod{W}{\X_Z(Q)}\right| &\leq& \left|\iprod{W}{\X_Z(Q_l)}\right| + \left|\iprod{W}{\X_Z(Q-Q_l)}\right| \\
&=& \left|\iprod{W}{\X_Z(Q_l)}\right| + \norm{\X_Z(Q-Q_l)}\left|\iprod{W}{\frac{\X_Z(Q-Q_l)}{\norm{\X_Z(Q-Q_l)}}}\right| \\
&\leq& \left|\iprod{W}{\X_Z(Q_l)}\right| + \frac{1}{2}\sup_{\{Q\in\mathbb{R}^{\ell(\Z_{\tau})}: \norm{\X_Z(Q)}\leq 1\}}\left|\iprod{W}{\X_Z(Q)}\right|.
\end{eqnarray*}
Taking $\sup$ with respect to $Q$ and $\max$ with respect to $Q_l$, we get
\begin{eqnarray*}
&& \sup_{\{Q\in\mathbb{R}^{\ell(\Z_{\tau})}: \norm{\X_Z(Q)}\leq 1\}}\left|\iprod{W}{\X_Z(Q)}\right| \\
&\leq& \max_{1\leq l\leq L}\left|\iprod{W}{\X_Z(Q_l)}\right| + \frac{1}{2}\sup_{\{Q\in\mathbb{R}^{\ell(\Z_{\tau})}: \norm{\X_Z(Q)}\leq 1\}}\left|\iprod{W}{\X_Z(Q)}\right|,
\end{eqnarray*}
which leads to (\ref{eq:ref-asked}) after rearrangement.
\end{proof}

\begin{proof}[Proof of Lemma \ref{lem:growth}]
The first inequality holds because
\begin{eqnarray*}
\sum_{\{\tau\in\T:\epsilon(\Z_{\tau})\leq\alpha\}} \exp\left(\beta\epsilon(\Z_{\tau})\right) &\leq& \sum_{t=1}^{\ceil{\alpha}}\sum_{\{\tau\in\T: t-1<\epsilon(\Z_{\tau})\leq t\}}e^{\beta \epsilon(\Z_{\tau})} + e^{\beta} \\
&\leq& \sum_{t=1}^{\ceil{\alpha}}te^{\beta t} + e^{\beta} \\
&\leq& 2\ceil{\alpha}\frac{e^{\beta}}{e^{\beta}-1}e^{\beta\ceil{\alpha}} \\
&\leq& 4\ceil{\alpha}\exp(\beta\ceil{\alpha}),
\end{eqnarray*}
by $\beta\geq 2$. The second inequality holds because
\begin{eqnarray}
\nonumber \sum_{\{\tau\in\T:\epsilon(\Z_{\tau})>\alpha\}}\exp\left(-\beta\epsilon(\Z_{\tau})\right) &\leq& \sum_{t=\floor{\alpha}}^{\infty}\sum_{\{\tau\in\T:t<\epsilon(\Z_{\tau})\leq t+1\}}e^{-\beta\epsilon(\Z_{\tau})} \\
\nonumber &\leq& \sum_{t=\floor{\alpha}}^{\infty}(t+1)e^{-\beta t} \\
\label{eq:logt/t} &\leq& 2\sum_{t=\floor{\alpha}}^{\infty}\exp\left(-\left(\beta-\frac{\log\floor{\alpha}}{\floor{\alpha}}\right)t\right) \\
\nonumber &\leq& 4\alpha\exp(\beta\floor{\alpha}),
\end{eqnarray}
for $\beta\geq 2$ and $\alpha\geq 1$. The inequality (\ref{eq:logt/t}) is because $\log t\leq \frac{\log\floor{\alpha}}{\floor{\alpha}}t$ for all $t\geq\floor{\alpha}$. Finally,
\begin{eqnarray*}
\sum_{\{\tau\in\T:\epsilon(\Z_{\tau})\leq\alpha\}} \exp\left(-\beta\epsilon(\Z_{\tau})\right) &\leq& 1+\sum_{t=1}^{\infty}te^{-\beta(t-1)} \\
&\leq& 1+e^{\beta}\sum_{t=1}^{\infty}e^{-(\beta-1)t} \\
&\leq& 6,
\end{eqnarray*}
for $\beta\geq 2$.
\end{proof}

\section{Proofs in Section \ref{sec:more}}\label{sec:pfmore}

\begin{proof}[Proof of Theorem \ref{thm:sup}]
The assumption $\tau s^*\leq 1/9$ and the argument in the proof of Theorem 1 of \cite{lounici2008sup} implies
\begin{equation}
\max_{|S|\leq (2+\delta)s^*}\opnorm{\left(n^{-1}X_{*S}^TX_{*S}\right)^{-1}}\leq \max_{|S|\leq (2+\delta)s^*}\norm{\left(n^{-1}X_{*S}^TX_{*S}\right)^{-1}}_{\ell_1}\leq 4\label{eq:lounici}
\end{equation}
for $\delta\leq 1/4$.
Define $\hat{\beta}_S=\min_b\norm{Y-X_{*S}b}^2$. Then it is easy to see that $\norm{Y-X_{*S}\beta_S}^2=\norm{Y-X_{*S}\hat{\beta}_S}^2+\norm{X_{*S}(\beta_S-\hat{\beta}_S)}^2$. Define the distribution $\mathcal{L}(\hat{\beta}_S,X_{*S},\lambda)$ of $\beta_S$ that has density function
\begin{equation}
\frac{\exp\left(-\frac{1}{2}\norm{X_{*S}\beta_S-X_{*S}\hat{\beta}_S}^2-\lambda\norm{X_{*S}\beta_S}\right)}{\int\exp\left(-\frac{1}{2}\norm{X_{*S}\beta_S-X_{*S}\hat{\beta}_S}^2-\lambda\norm{X_{*S}\beta_S}\right)d\beta_S}.\label{eq:sup-density}
\end{equation}
Then, according to the formula of the posterior distribution, to sample $\beta$ from the posterior distribution is equivalent to first sample $S$ from $\Pi(S|Y)$ and then sample $\beta_S\sim \mathcal{L}(\hat{\beta}_S,X_{*S},\lambda)$ to form $\beta^T=(\beta_S^T,0_{S^c}^T)$. Hence, the posterior distribution can be represented as
$$\sum_S\Pi(S|Y)\Pi_S(\cdot|Y)=\sum_S\omega(S) \mathcal{L}(\hat{\beta}_S,X_{*S},\lambda)\otimes\delta_{S^c},$$
where $\Pi(S|Y)=\omega(S)$ and $\Pi_S(\cdot|Y)= \mathcal{L}(\hat{\beta}_S,X_{*S},\lambda)\otimes\delta_{S^c}$ with
\begin{equation}
\omega(S)\propto \frac{\pi(|S|)}{|\bar{\Z}_{|S|}|}\left(\frac{\lambda}{\sqrt{\pi}}\right)^{|S|}\mathcal{N}_{X_{*S}\hat{\beta}_S,\lambda}e^{-\frac{1}{2}\norm{Y-X_{*S}\hat{\beta}_S}^2}\mathbb{I}\{|\bar{\Z}_{|S|}|>0\}.\label{eq:post-omega}
\end{equation}
The number $\mathcal{N}_{y,\lambda}$ for any vector $y$ and any scalar $\lambda$ is defined as
\begin{equation}
\mathcal{N}_{y,\lambda}=\int\exp\left(-\frac{1}{2}\norm{t-y}^2-\lambda\norm{t}\right)dt.\label{eq:post-n}
\end{equation}
Define the event
\begin{equation}
E=\left\{\max_{j\in[p]}\left|\frac{X_{*j}^TW}{\sqrt{n}}\right|\leq C_1\sqrt{\log p}\right\}\label{eq:sup-noise}
\end{equation}
for some constant $C_1>0$ to be determined later.
We have
\begin{eqnarray}
\nonumber && \mathbb{E}\Pi\left(\norm{\beta-\beta^*}_{\infty}>M\sqrt{\frac{\log p}{n}}\Big|Y\right) \\
\nonumber &=& \mathbb{E}\sum_{|S|\leq (1+\delta)s^*}\omega(S)\Pi_S\left(\norm{\beta_S-\beta_S^*}_{\infty}\vee\norm{\beta^*_{S^c}}_{\infty}>M\sqrt{\frac{\log p}{n}}\Big|Y\right) + \mathbb{E}\Pi(|S|>(1+\delta)s^*|Y) \\
\nonumber &\leq&  \mathbb{E}\sum_{\substack{|S|\leq (1+\delta)s^*\\ \norm{\beta_{S^c}^*}_{\infty}\leq C_2\sqrt{\frac{\log p}{n}}}}\omega(S)\Pi_S\left(\norm{\beta_S-\hat{\beta}_S}_{\infty}>\frac{1}{2}M\sqrt{\frac{\log p}{n}}\Big|Y\right)\mathbb{I}_E+\mathbb{E}\sum_{\substack{|S|\leq (1+\delta)s^*\\ \norm{\beta^*_{S^c}}_{\infty}>C_2\sqrt{\frac{\log p}{n}}}}\omega(S)\mathbb{I}_E \\
\label{eq:plan} && +\mathbb{P}(E^c)+\mathbb{E}\Pi(|S|>(1+\delta)s^*|Y)
\end{eqnarray}
for some constant $C_2>0$ to be determined later.
The inequality (\ref{eq:plan}) is due to the inequality $\norm{\beta_S-\beta^*_S}_{\infty}\leq \norm{\beta_S-\hat{\beta}_S}_{\infty}+\norm{\hat{\beta}_S-\beta_S^*}$ and
\begin{equation}
E\subset\left\{\norm{\hat{\beta}_S-\beta_S^*}_{\infty}\leq \frac{1}{2}M\sqrt{\frac{\log p}{n}}\right\},\label{eq:middle}
\end{equation}
for all $S$ that satisfies $|S|\leq (1+\delta)s^*$ and $\norm{\beta_{S^c}^*}_{\infty}\leq C_2\sqrt{\frac{\log p}{n}}$.
Let us give a proof for (\ref{eq:middle}). By the definition of $\hat{\beta}_S$, we have $X_{*S}^TX_{*S}\hat{\beta}_S=X_{*S}^TY=X_{*S}^TX_{*S}\beta^*_S+X_{*S}^TX_{*S^c}\beta^*_{S^c}+X_{*S}^TW$, which implies
$$\norm{\hat{\beta}_S-\beta_S^*}_{\infty}\leq 4\norm{X_{*S}^TX_{*S}(\hat{\beta}_S-{\beta_S^*})}_{\infty}/n\leq \frac{4}{n}\norm{X_{*S}^TX_{*S^c}\beta^*_{S^c}}_{\infty}+\frac{4}{n}\norm{X_{*S}^TW}_{\infty}.$$
Note that $\frac{4}{n}\norm{X_{*S}^TX_{*S^c}\beta^*_{S^c}}_{\infty}=\frac{4}{n}\norm{X_{*S}^TX_{*S^*\cap S^c}\beta^*_{S^*\cap S^c}}_{\infty}\leq 8s^*\tau\norm{\beta^*_{S^c}}_{\infty}\leq C_2\sqrt{\frac{\log p}{n}}$ due to $\norm{\beta_{S^c}^*}_{\infty}\leq C_2\sqrt{\frac{\log p}{n}}$. We also have $\frac{4}{n}\norm{X_{*S}^TW}_{\infty}\leq \frac{4}{n}\max_{j\in[p]}|X_{*j}^TW|\leq 4C_1\sqrt{\frac{\log p}{n}}$. Therefore, (\ref{eq:middle}) is proved for some $M/2\geq 4C_1+C_2$.

In view of (\ref{eq:plan}), it is sufficient to bound the four terms in (\ref{eq:plan}). The last term is bounded as a result of (\ref{eq:slr-comp}). The third term is bounded by $p^{-\left(\frac{C_1\rho}{2}-1\right)}$ using (\ref{eq:subG}) and a union bound argument. Let us give a bound for the first term.
\begin{eqnarray}
\nonumber && \Pi_S\left(\norm{\beta_S-\hat{\beta}_S}_{\infty}>\frac{1}{2}M\sqrt{\frac{\log p}{n}}\Big|Y\right) \\
\nonumber &\leq& \sum_{j\in S}\Pi_S\left(|\beta_j-\hat{\beta}_j|>\frac{1}{2}M\sqrt{\frac{\log p}{n}}\Big|Y\right) \\
\label{eq:used-g} &\leq& \sum_{j\in S}\exp\left(-\frac{1}{2}tM\sqrt{\log p}\right)\mathbb{E}_{\Pi_S}\left(e^{\sqrt{n}t|\beta_j-\hat{\beta}_j|}\Big|Y\right),
\end{eqnarray}
where $\mathbb{E}_{\Pi_S}(\cdot|Y)$ is the posterior expectation with the distribution $\Pi_S(\cdot|Y)=\mathcal{L}(\hat{\beta}_S,X_{*S},\lambda)$ and $t>0$ is some number to be specified later. Using the formula of the density (\ref{eq:sup-density}), for any unit vector $v\in\mathbb{R}^{|S|}$, we have
\begin{eqnarray}
\nonumber && \mathbb{E}_{\Pi_S}\left(e^{\sqrt{n}tv^T(\beta_S-\hat{\beta}_S)}\Big|Y\right) \\
\nonumber &=& \frac{\int\exp\left(\sqrt{n}tv^T(\beta_S-\hat{\beta}_S)-\frac{1}{2}\norm{X_{*S}\beta_S-X_{*S}\hat{\beta}_S}^2-\lambda\norm{X_{*S}\beta_S}\right)d\beta_S}{\int\exp\left(-\frac{1}{2}\norm{X_{*S}\beta_S-X_{*S}\hat{\beta}_S}^2-\lambda\norm{X_{*S}\beta_S}\right)d\beta_S} \\
\nonumber &=& e^{\frac{1}{2}t^2\norm{(n^{-1}X_{*S}^TX_{*S})^{-1/2}v}^2}\frac{\int\exp\left(-\frac{1}{2}\norm{X_{*S}(\beta_S-\hat{\beta}_S-tn^{-1/2}(n^{-1}X_{*S}^TX_{*S})^{-1}v)}^2-\lambda\norm{X_{*S}\beta_S}\right)d\beta_S}{\int \exp\left(-\frac{1}{2}\norm{X_{*S}(\beta-\hat{\beta}_S)}^2-\lambda\norm{X_{*S}\beta_S}\right)d\beta_S} \\
\label{eq:sup-triangle} &\leq& \exp\left(\frac{1}{2}t^2\norm{(n^{-1}X_{*S}^TX_{*S})^{-1/2}v}^2+\lambda t\norm{(n^{-1}X_{*S}^TX_{*S})^{-1/2}v}\right) \\
\nonumber &\leq&  \exp\left(\frac{1}{2}\lambda^2+t^2\norm{(n^{-1}X_{*S}^TX_{*S})^{-1/2}v}^2\right) \\
\label{eq:sup-from} &\leq& \exp\left(\frac{1}{2}\lambda^2+16t^2\right),
\end{eqnarray}
where the inequality (\ref{eq:sup-triangle}) is due to a change of variable and triangle inequality and the inequality (\ref{eq:sup-from}) is by (\ref{eq:lounici}). Specializing $v$ so that $v^T(\beta_S-\hat{\beta}_S)=\pm(\beta_j-\hat{\beta}_j)$, we have
$$\mathbb{E}_{\Pi_S}\left(e^{\sqrt{n}t|\beta_j-\hat{\beta}_j|}\Big|Y\right)\leq \mathbb{E}_{\Pi_S}\left(e^{\sqrt{n}t(\beta_j-\hat{\beta}_j)}\Big|Y\right)+\mathbb{E}_{\Pi_S}\left(e^{-\sqrt{n}t(\beta_j-\hat{\beta}_j)}\Big|Y\right)\leq 2e^{\frac{1}{2}\lambda^2+16t^2}.$$
Letting $t=\sqrt{\log p}$, we have
$$\Pi_S\left(\norm{\beta_S-\hat{\beta}_S}_{\infty}>\frac{1}{2}M\sqrt{\frac{\log p}{n}}\Big|Y\right)\leq 2e^{\lambda^2/2}p^{-\left(\frac{M}{2}-17\right)},$$
which bounds the first term of (\ref{eq:plan}).

Now, let us give a bound for the second term of (\ref{eq:plan}). Given $j=\argmax_{l\in[p]}|\beta_j^*|$, for any $S\subset[p]$ such that $j\notin S$, define $S'=S\cup\{j\}$. We are going to provide a bound for $\omega(S)/\omega(S')$ on the event $E$ to argue the model $S'$ is favored over the model $S$ under the posterior distribution if $|\beta_j^*|$ is large. Because of (\ref{eq:lounici}), $|\bar{\Z}_{|S|}|=|\Z_{|S|}|={p\choose |S|}$ for all $|S|\leq (1+\delta)s^*$. By (\ref{eq:post-omega}), we have
$$\frac{\omega(S)}{\omega(S')}=\frac{\pi(|S|)}{\pi(|S'|)}\frac{{p\choose|S'|}}{{p\choose|S|}}\frac{\sqrt{\pi}}{\lambda}\frac{\mathcal{N}_{X_{*S}\hat{\beta}_S,\lambda}}{\mathcal{N}_{X_{*S'}\hat{\beta}_{S'},\lambda}}e^{-\frac{1}{2}\norm{Y-X_{*S}\hat{\beta}_S}^2+\frac{1}{2}\norm{Y-X_{*S'}\hat{\beta}_{S'}}^2}.$$
Since $\frac{\pi(|S|)}{\pi(|S'|)}\leq \exp\left(2D\log(ep)\right)$, $\frac{{p\choose|S'|}}{{p\choose|S|}}\leq p$, 
$\frac{\mathcal{N}_{X_{*S}\hat{\beta}_S,\lambda}}{\mathcal{N}_{X_{*S'}\hat{\beta}_{S'},\lambda}}\leq e^{\lambda\norm{X_{*S}\hat{\beta}_{S}-X_{*S'}\hat{\beta}_{S'}}}$
by the definition (\ref{eq:post-n}) and a change of variable, and $-\frac{1}{2}\norm{Y-X_{*S}\hat{\beta}_S}^2+\frac{1}{2}\norm{Y-X_{*S'}\hat{\beta}_{S'}}^2=\frac{1}{2}\norm{X_{*S}\hat{\beta}_S}^2-\frac{1}{2}\norm{X_{*S'}\hat{\beta}_{S'}}^2$, we have
\begin{equation}
\frac{\omega(S)}{\omega(S')} \leq \frac{\sqrt{\pi}}{\lambda}(ep)^{2D+1}e^{\lambda\norm{X_{*S}\hat{\beta}_{S}-X_{*S'}\hat{\beta}_{S'}}+\frac{1}{2}\norm{X_{*S}\hat{\beta}_S}^2-\frac{1}{2}\norm{X_{*S'}\hat{\beta}_{S'}}^2}.\label{eq:post-ratio}
\end{equation}
Let $P_S$ and $P_{S'}$ stand for the projection matrix onto the column spaces of $X_{*S}$ and $X_{*S'}$, respectively. Then $X_{*S}\hat{\beta}_S=P_SY$ and $X_{*S'}\hat{\beta}_{S'}=P_{S'}Y$. Let $F$ be the orthogonal complement of the columns space of $X_{*S}$ in the column space of $X_{*S'}$, and then define $P_F$ to be the associated projection matrix. It is easy to see that $P_{S'}=P_S+P_F$ and $P_SP_F=0$. Thus, the exponent of (\ref{eq:post-ratio}) equals $\lambda\norm{P_FY}-\frac{1}{2}\norm{P_FY}^2\leq -\frac{1}{4}\norm{P_FY}^2+\lambda^2\leq -\frac{1}{8}\norm{P_FX\beta^*}^2+\frac{1}{4}\norm{P_FW}^2+\lambda^2$. We are going to give a lower bound on $\norm{P_FX\beta^*}^2$ and an upper bound on $\norm{P_FW}^2$. To facilitate the proof, we bound $\norm{P_SX_{*j}}^2$ as
\begin{eqnarray}
\nonumber \norm{P_SX_{*j}}^2 &=& X_{*j}^TX_{*S}(X_{*S}^TX_{*S})^{-1}X_{*S}^TX_{*j} \\
\nonumber &\leq& 4n\norm{n^{-1}X_{*S}^TX_{*j}}^2 \\
\label{eq:8/9} &\leq& 8ns^*\tau^2\leq \frac{n}{10}
\end{eqnarray}
by (\ref{eq:lounici}) and $\tau s^*\leq 1/9$. The noise part $\norm{P_FW}^2$ is bounded as
\begin{eqnarray}
\nonumber \norm{P_FW}^2 &=& \left\|\frac{(I-P_S)X_{*j}X_{*j}^T(I-P_S)}{\norm{(I-P_S)X_{*j}}^2}W\right\|^2 \\
\nonumber &\leq& \frac{|X_{*j}^T(I-P_S)W|^2}{\norm{(I-P_S)X_{*j}}^2} \\
\label{eq:pn2} &\leq& \frac{2|X_{*j}^TW|^2+2|X_{*j}^TP_SW|^2}{9n/10} \\
\label{eq:pn3} &\leq& 8C_1^2\log p,
\end{eqnarray}
where (\ref{eq:pn2}) is because of (\ref{eq:8/9}) and (\ref{eq:pn3}) is derived from the event $E$ and the following argument that
\begin{eqnarray*}
|X_{*j}^TP_SW|^2 &=& |X_{*j}^TX_{*S}(X_{*S}^TX_{*S})^{-1}X_{*S}^TW|^2 \\
&\leq& 16n\norm{X_{*j}^TX_{*S}/n}^2\norm{X_{*S}^TW/\sqrt{n}}^2 \\
&\leq& 32C_1^2(s^*\tau)^2n\log p\leq \frac{1}{2}C_1^2n\log p
\end{eqnarray*}
by (\ref{eq:lounici}) and the event $E$. The signal part $\norm{P_FX\beta^*}^2$ is lower bounded by
$$\norm{P_FX\beta^*}\geq \norm{(I-P_S)X_{*j}}|\beta_j^*|-\sum_{l\in S^*\cap (S\cup\{j\})^c}\frac{|X_{*j}^T(I-P_S)X_{*l}|}{\norm{(I-P_S)X_{*j}}}|\beta_l^*|,$$
where the first term on the right hand side above is lower bounded by $\sqrt{9n/10}|\beta_j^*|$ by (\ref{eq:8/9}), and the second term is upper bounded by
$$ \sum_{l\in S^*\cap \{j\}^c}|\beta_l^*|\frac{|X_{*j}^TX_{*l}|}{\norm{(I-P_S)X_{*j}}} + \sum_{l\in S^*\cap S^c}|\beta_l^*|\frac{|X_{*j}^TP_SX_{*l}|}{\norm{(I-P_S)X_{*j}}} \leq 7\sqrt{n}|\beta_j^*|/9$$
due to (\ref{eq:8/9}), (\ref{eq:lounici}), $\tau s^*\leq 1/9$ and the fact $|\beta_{j}^*|=\max_{l\in[p]}|\beta_l^*|$. Therefore, $\norm{P_FX\beta^*}\geq \sqrt{n}|\beta_j^*|/7$. When $|\beta_j^*|\geq400 C_1\sqrt{\frac{\log p}{n}}$, we have $-\frac{1}{8}\norm{P_FX\beta^*}^2+\frac{1}{4}\norm{P_FW}^2\leq -2C_1^2\log p$. Plugging this bound into (\ref{eq:post-ratio}), we have $\frac{\omega(S)}{\omega(S')}\leq \frac{\sqrt{\pi}}{\lambda}e^{2D+1+\lambda^2}p^{-(2C_1^2-2D-1)}$, which implies
$$\sum_{\substack{|S|\leq (1+\delta)s^*\\ j\notin S}}\omega(S)\mathbb{I}_E=\sum_{\substack{|S|\leq (1+\delta)s^*\\ j\notin S}}\frac{\omega(S)}{\omega(S\cup\{j\})}\omega(S\cup\{j\})\mathbb{I}_E\leq \frac{\sqrt{\pi}}{\lambda}e^{2D+1+\lambda^2}p^{-(2C_1^2-2D-1)}.$$
By letting $C_2=400C_1$, a mathematical induction argument in \cite{castillo2014} leads to a bound on the second term of (\ref{eq:plan}) that
$$\sum_{\substack{|S|\leq (1+\delta)s^*\\ \norm{\beta^*_{S^c}}_{\infty}>C_2\sqrt{\frac{\log p}{n}}}}\omega(S)\mathbb{I}_E\leq p^{-C_3},$$
for some constant $C_3$ depending on $C_1,D,\lambda$. Moreover, $C_3$ is increasing with $C_1$. Note that the constant $400$ can be significantly improved through a more careful analysis, especially when the design matrix is nearly orthogonal. However, we focus on the rate of the problem and does not make such effort.

Finally, combining the bounds for the four terms in (\ref{eq:plan}), we get
\begin{eqnarray*}
&& \mathbb{E}\Pi\left(\norm{\beta-\beta^*}_{\infty}>M\sqrt{\frac{\log p}{n}}\Big|Y\right) \\
&\leq& 2e^{\lambda^2/2}p^{-\left(\frac{M}{2}-17\right)}+p^{-C_3}+p^{-\left(\frac{C_1\rho}{2}-1\right)}+e^{-C's^*\log\frac{ep}{s^*}} \\
&\leq& p^{-C_4},
\end{eqnarray*}
for some $M,C_4$ depending on $\rho,\lambda,D$.
\end{proof}

\begin{proof}[Proof of Theorem \ref{thm:sup-g}]
For $B_T$, we use $\norm{\cdot}$ to denote the $\ell_2$ norm as $\norm{B_T}=\sqrt{\sum_{(i,j)\in T}B_{ij}^2}$.
Let us first establish (\ref{eq:sup-g-dim}) and (\ref{eq:sup-g-2}). The proof is close to that of Theorem \ref{thm:main}. By the definition of the prior, the posterior distribution has formula
\begin{equation}
\Pi(B\in U|Y) = \frac{\sum_T\alpha(T)R(T,U)}{\sum_T\alpha(T)R(T)},\label{eq:sup-form1}
\end{equation}
where $R(T,U)$ is defined by
$$\left(\frac{\lambda}{\sqrt{\pi}}\right)^{|T|}\int_{(B_T,0_{T^c})\in U}e^{-\frac{1}{2}\norm{(B_T,0_{T^c})-B^*}^2+\iprod{W}{(B_T,0_{T^c})-B^*}-\lambda\norm{B_T}}dB_T,$$
$R(T)=R(T,\mathbb{R}^{p\times m})$ and
$$\alpha(T)=\exp\left(-D\left(|r(T)|\log\frac{ep}{|r(T)|}+|T|\log\frac{em|r(T)|}{|T|}\right)\right).$$
Moreover, for a set of subsets $\mathcal{A}$, the posterior distribution can be written as
\begin{equation}
\Pi(T\in\mathcal{A}|Y) = \frac{\sum_{T\in\mathcal{A}}\alpha(T)R(T)}{\sum_T\alpha(T)R(T)}. \label{eq:sup-form2}
\end{equation}
We need to give a lower bound for $R(T^*)$ with $T^*=S^*\times[m]$ and give upper bounds for $R(T)$ and $R(T,U)$. For each subset $T$, define the following events
\begin{eqnarray*}
E_T &=& \left\{\left|\iprod{W}{(B_T,0_{T^c})-B^*}\right|\leq\sqrt{C_1\left(m|r(T)|+|r(T)|\log\frac{ep}{|r(T)|}\right)}\norm{(B_T,0_{T^c})-B^*}\text{ for all }B_T\in\mathbb{R}^{|T|}\right\}, \\
F_T &=& \left\{\left|\iprod{W}{(B_T,0_{T^c})-B^*}\right|\leq\sqrt{C_1\left(ms^*+s^*\log\frac{ep}{s^*}\right)}\norm{(B_T,0_{T^c})-B^*}\text{ for all }B_T\in\mathbb{R}^{|T|}\right\}
\end{eqnarray*}
for some constant $C_1>0$ to be determined later. A special case of Lemma \ref{lem:noise} gives
\begin{equation}
\mathbb{P}(E_T^c) \leq 2e^{-(\rho C_1/16-5)\left(m|r(T)|+|r(T)|\log\frac{ep}{|r(T)|}\right)}\quad\text{and}\quad \mathbb{P}(F_T^c) \leq 2e^{5m|r(T)|-\frac{\rho C_1}{16}\left(ms^*+s^*\log\frac{ep}{s^*}\right)}.
\end{equation}
The same arguments used for deriving (\ref{eq:denom-lower}), (\ref{eq:num-upper}) and (\ref{eq:numU-upper}) imply
\begin{eqnarray}
\label{eq:denom-g-lower} R(T^*) &\geq& e^{-\lambda\norm{B^*}-(1+\lambda+\lambda^{-1})ms^*}, \\
\label{eq:num-g-upper} R(T)\mathbb{I}_{E_T} &\leq& e^{4\lambda^2-\lambda\norm{B^*}+C_1'\left(m|r(T)|+|r(T)|\log\frac{ep}{|r(T)|}\right)}, \\
\label{eq:numU-g-upper} R(T,U)\mathbb{I}_{F_T} &\leq& e^{-\lambda\norm{B^*}-\frac{1}{16}M\left(ms^*+s^*\log\frac{ep}{s^*}\right)},
\end{eqnarray}
with $U=\left\{\norm{B-B^*}>M\left(ms^*+s^*\log\frac{ep}{s^*}\right)\right\}$ for some sufficiently large $M$. The constant $C_1'$ us defined as $4C_1+2C_1/C_2+C_1|\log(2\lambda)|$.
Let $\mathcal{A}=\{|r(T)|>(1+\delta)s^*\}$. By the formula (\ref{eq:sup-form2}) and the inequalities (\ref{eq:denom-g-lower}) and (\ref{eq:num-g-upper}), we have
\begin{eqnarray*}
&& \mathbb{E}\Pi(T\in\mathcal{A}|Y) \\
&\leq& \sum_{T\in\mathcal{A}}\frac{\alpha(T)}{\alpha(T^*)}\mathbb{E}\frac{R(T)}{R(T^*)}\mathbb{I}_{E_T}+\sum_{T\in\mathcal{A}}\mathbb{P}(E_T^c) \\
&\leq& e^{(\tilde{C}_2+D)\left(ms^*+s^*\log\frac{ep}{s^*}\right)}\sum_{s>(1+\delta)s^*}\sum_{S:|S|=s}e^{-(D-\tilde{C}_2)(ms+s\log\frac{ep}{s})}\sum_{T:r(T)=S}e^{-D|T|\log\frac{ems}{|T|}} \\
&& +2\sum_{s>(1+\delta)s^*}\sum_{S:|S|=s}e^{-(\rho C_1/16-7)\left(ms+s\log\frac{ep}{s}\right)} \\
&\leq& e^{-C'\left(ms^*+s^*\log\frac{ep}{s^*}\right)}
\end{eqnarray*}
for some sufficiently large $D$ with $\tilde{C}_2$ only depending on $C_1,C_2$ and $\lambda$ and $C'$ only depending on $D,\rho,\lambda$. By the formula (\ref{eq:sup-form1}) and the inequalities (\ref{eq:denom-g-lower}) and (\ref{eq:numU-upper}), we have
\begin{eqnarray*}
&& \mathbb{E}\Pi(B\in U|Y) \\
&\leq& \sum_{T\in\mathcal{A}^c}\frac{\alpha(T)}{\alpha(T^*)}\mathbb{E}\frac{R(T,U)}{R(T^*)}\mathbb{I}_{F_T} + \sum_{T\in\mathcal{A}^c}\mathbb{P}(F_T^c)+e^{-C'\left(ms^*+s^*\log\frac{ep}{s^*}\right)} \\
&\leq& e^{-\left(\frac{1}{16}M-C_3-D\right)\left(ms^*+s^*\log\frac{ep}{s^*}\right)}\sum_{s\leq (1+\delta)s^*}\sum_{S:|S|=s}e^{-D(ms+s\log\frac{ep}{s})}\sum_{T:r(T)=S}e^{-D|T|\log\frac{ems}{|T|}} \\
&& +2e^{-\frac{\rho C_1}{16}\left(ms^*+s^*\log\frac{ep}{s^*}\right)}\sum_{s\leq (1+\delta)s^*}\sum_{S:|S|=s}e^{6ms}+e^{-C'\left(ms^*+s^*\log\frac{ep}{s^*}\right)}\\
&\leq& e^{-C''\left(ms^*+s^*\log\frac{ep}{s^*}\right)}
\end{eqnarray*}
for some sufficiently large $M$ with $C_3$ only depending on $C_1,C_2$ and $\lambda$ and $C''$ only depending on $D,\rho,\lambda$. Hence, (\ref{eq:sup-g-dim}) and (\ref{eq:sup-g-2}) are proved.

Now let us proceed to prove (\ref{eq:sup-g-inf}). We are going to use the similar argument as that of Theorem \ref{thm:sup}. Note that the posterior distribution can be represented as
$$\sum_T\Pi(T|Y)\Pi_T(\cdot|Y)=\sum_T\omega(T)\mathcal{L}(Y_T,\lambda)\otimes \delta_{T^c},$$
where $\Pi(T|Y)=\omega(T)$ and $\Pi_T(\cdot|Y)=\mathcal{L}(Y_T,\lambda)\otimes \delta_{T^c}$ with
$$\omega(T)\propto\left(\frac{\lambda}{\sqrt{\pi}}\right)^{|T|}\alpha(T)\mathcal{N}_{Y_T,\lambda}e^{\frac{1}{2}\norm{Y_T}^2}.$$
The distribution $B_T\sim\mathcal{L}(Y_T,\lambda)$ is defined through the density function
$$\mathcal{N}_{Y_T,\lambda}^{-1}e^{-\frac{1}{2}\norm{B_T-Y_T}^2-\lambda\norm{B_T}},$$
where $\mathcal{N}_{Y_T,\lambda}$ is the normalizing constant defined in (\ref{eq:post-n}). Define the event
$$E=\left\{\max_{(i,j)\in[p]\times[m]}|W_{ij}|\leq C_1\sqrt{\log(pm)}\right\}$$
for some constant $C_1>0$. We have
\begin{eqnarray}
\nonumber&& \mathbb{E}\Pi\left(\norm{B-B^*}_{\infty}>M\sqrt{\log (pm)}\Big|Y\right) \\
\nonumber&\leq& \mathbb{E}\sum_{|r(T)|\leq (1+\delta)s^*}\omega(T)\Pi_T\left(\norm{B_T-Y_T}_{\infty}>\frac{1}{2}M\sqrt{\log (pm)}\Big|Y\right)\mathbb{I}_E + \mathbb{E}\sum_{\substack{|r(T)|\leq(1+\delta)s^*\\ \norm{B^*_{T^c}}_{\infty}>C_2\sqrt{\log(pm)}}}\omega(T)\mathbb{I}_E \\
\label{eq:plan2}&& + \mathbb{P}(E^c) + \mathbb{E}\Pi\left(|r(T)|>(1+\delta)s^*|Y\right).
\end{eqnarray}
It is sufficient to bound the four terms in (\ref{eq:plan2}). The last term is bounded by (\ref{eq:sup-g-dim}). Using (\ref{eq:subG}) and a union bound argument, we bound the third term in (\ref{eq:plan2}) as $\mathbb{P}(E^c)\leq (pm)^{-\left(\frac{\rho C_1^2}{2}-1\right)}$. Using the same arguments in deriving (\ref{eq:used-g}) and (\ref{eq:sup-from}), we have
\begin{eqnarray*}
&& \Pi_T\left(\norm{B_T-Y_T}_{\infty}>\frac{1}{2}M\sqrt{\log (pm)}\Big|Y\right) \\
&\leq& \sum_{(i,j)\in T}\exp\left(-\frac{1}{2}tM\sqrt{\log(pm)}\right)\mathbb{E}_{\Pi_T}\left(e^{\sqrt{n}t|B_{ij}-Y_{ij}|}\Big|Y\right) \\
&\leq& 2e^{\lambda^2/2}pm e^{-\frac{1}{2}tM\sqrt{\log(pm)}+t^2} \leq 2e^{\lambda^2/2}(pm)^{-\left(\frac{M}{2}-2\right)}
\end{eqnarray*}
by choosing $t=\sqrt{\log(pm)}$. This bounds the first term of (\ref{eq:plan2}). Now let us provide a bound for the first term of (\ref{eq:plan2}). Given some $(i,j)\in[p]\times[m]$, for any subset $T$ such that $(i,j)\notin T$, use the notation $T'=T\cup\{(i,j)\}$. To facilitate the proof, we need an upper bound for $\omega(T)/\omega(T')$ on the event $E$. Direct calculation gives
$$\frac{\omega(T)}{\omega(T')}=\frac{\sqrt{\pi}}{\lambda}\frac{\alpha(T)}{\alpha(T')}\frac{\mathcal{N}_{Y_T,\lambda}}{\mathcal{N}_{Y_{T'},\lambda}}e^{-\frac{1}{2}Y_{ij}^2}.$$
Since $\frac{\alpha(T)}{\alpha(T')}\leq(epm)^{3D}$, and
\begin{equation}
\frac{\mathcal{N}_{Y_T,\lambda}}{\mathcal{N}_{Y_{T'},\lambda}}\leq C_{\lambda}e^{\lambda|Y_{ij}|}\label{eq:final...}
\end{equation}
for some constant $C_{\lambda}$ only depending on $\lambda$, we have $\omega(T)/\omega(T')\leq C_{\lambda}\frac{\sqrt{\pi}}{\lambda}(epm)^{3D}e^{\lambda|Y_{ij}|-\frac{1}{2}Y_{ij}^2}$. The inequality (\ref{eq:final...}) will be established in the end of the proof. Since
\begin{eqnarray*}
\lambda|Y_{ij}|-\frac{1}{2}Y_{ij}^2 &\leq& \lambda^2-\frac{1}{4}Y_{ij}^2 \\
&\leq& \lambda^2-\frac{1}{8}(B_{ij}^*)^2+\frac{1}{4}W_{ij}^2\leq \lambda^2-\frac{1}{4}C_1^2\log(pm) 
\end{eqnarray*}
when $|B_{ij}^*|>2C_1\sqrt{\log(pm)}$ on the event $E$. Hence,
\begin{equation}
\frac{\omega(T)}{\omega(T')}\mathbb{I}_E\leq C_{\lambda}\frac{\sqrt{\pi}}{\lambda}e^{3D+\lambda^2}(pm)^{-\left(\frac{1}{4}C_1^2-3D\right)}.\label{eq:final-ratio}
\end{equation}
Let $C_2=2C_1$ and define $\{(i_1,j_1),...,(i_q,j_q)\}$ to be the set such that $|B_{i_lj_l}^*|>2C_1\sqrt{\log(pm)}$ for all $l\in[q]$. Then, we have
$$\{\norm{B^*_{T^c}}_{\infty}>C_2\sqrt{\log(pm)}\}\subset \cup_{l\in[q]}\{(i_l,j_l)\notin T\},$$
which implies
\begin{eqnarray*}
\sum_{\substack{|r(T)|\leq(1+\delta)s^*\\ \norm{B^*_{T^c}}_{\infty}>C_2\sqrt{\log(pm)}}}\omega(T) &\leq& \sum_{l\in[q]}\sum_{T\in\{T: (i_l,j_l)\notin T\}}\frac{\omega(T)}{\omega(T\cup\{(i_l,j_l)\})}\omega(T\cup\{(i_l,j_l)\}) \\
&\leq& C_{\lambda}\frac{\sqrt{\pi}}{\lambda}e^{3D+\lambda^2}(pm)^{-\left(\frac{1}{4}C_1^2-3D\right)}\sum_{l\in[q]}\sum_{T\in\{T: (i_l,j_l)\notin T\}}\omega(T\cup\{(i_l,j_l)\}) \\
&\leq& (pm)^{-\bar{C}}
\end{eqnarray*}
by (\ref{eq:final-ratio}) for some constant $\bar{C}$ with sufficiently large $C_1$. Combining the bounds for the four terms in (\ref{eq:plan}), we reach the conclusion (\ref{eq:sup-g-inf}).

Finally, let us establish (\ref{eq:final...}) to close the proof. By change of variable, we have
$$\mathcal{N}_{Y_T,\lambda}=\int_{\mathbb{R}^{|T|-1}}\int_{\mathbb{R}} e^{-\frac{1}{2}b_1^2-\frac{1}{2}\norm{b_2}^2-\lambda\sqrt{(b_1+\norm{Y_T})^2+\norm{b_2}^2}}db_1db_2,$$
and
$$\mathcal{N}_{Y_{T'},\lambda}=\int_{\mathbb{R}}\int_{\mathbb{R}^{|T|-1}}\int_{\mathbb{R}} e^{-\frac{1}{2}(b_1^2+b_3^2)-\frac{1}{2}\norm{b_2}^2-\lambda\sqrt{(b_1+\norm{Y_{T'}})^2+\norm{b_2}^2+b_3^2}}db_1db_2db_3.$$
Therefore, triangle inequality implies
$$\mathcal{N}_{Y_{T'},\lambda}\geq \mathcal{N}_{Y_T,\lambda} \int_{\mathbb{R}}e^{-\frac{1}{2}b^2-\lambda|b|}db e^{-\lambda\left|\norm{Y_T}-\norm{Y_{T'}}\right|}\geq C_{\lambda}^{-1}e^{-\lambda|Y_{ij}|},$$
where $C_{\lambda}=\left(\int_{\mathbb{R}}e^{-\frac{1}{2}b^2-\lambda|b|}db\right)^{-1}$. Thus, the proof is complete.
\end{proof}





\end{document}